\documentclass[10pt]{amsart}
\usepackage{amssymb,bm,mathrsfs}
\usepackage{amsthm}
\usepackage{amsmath}
\usepackage{amscd}
\usepackage{verbatim}
\usepackage[all]{xy}

\numberwithin{equation}{section}

\theoremstyle{plain}
\newtheorem{theorem}{Theorem}[section]
\newtheorem{corollary}[theorem]{Corollary}
\newtheorem{lemma}[theorem]{Lemma}
\newtheorem{proposition}[theorem]{Proposition}

\newtheorem{BigThm}{Theorem}
\newtheorem{BigConj}[BigThm]{Conjecture}

\theoremstyle{definition}
\newtheorem{definition}[theorem]{Definition}

\theoremstyle{remark}
\newtheorem{remark}[theorem]{Remark}

\newcommand{\map}[1]{\xrightarrow{#1}}
\newcommand{\iso}{\cong}

\newcommand{\co}{\mathcal O}
\newcommand{\alg}{\mathrm{alg}}

\newcommand{\Lie}{\mathrm{Lie}}
\newcommand{\kk}{{\bm{k}}}
\newcommand{\univ}{\mathrm{univ}}
\newcommand{\fiber}{ \beta_z}
\newcommand{\pol}{\psi}

\newcommand{\LL}{{\mathbb{L}}}
\newcommand{\VV}{{\mathbb{V}}}

\newcommand{\A}{\mathbb{A}}
\newcommand{\R}{\mathbb{R}}
\newcommand{\B}{\mathbb{B}}
\newcommand{\Q}{\mathbb{Q}}
\newcommand{\Z}{\mathbb{Z}}

\newcommand{\C}{\mathbb{C}}
\renewcommand{\H}{\mathbb{H}}
\renewcommand{\P}{\mathbb{P}}
\newcommand{\F}{\mathbb{F}}

\newcommand{\kzxz}[4]{\left(\begin{smallmatrix} #1 & #2 \\ #3 & #4\end{smallmatrix}\right) }
\newcommand{\kabcd}{\kzxz{a}{b}{c}{d}}

\renewcommand{\Im}{\operatorname{Im}}
\renewcommand{\Re}{\operatorname{Re}}

\newcommand{\calD}{\mathcal{D}}

\newcommand{\calF}{\mathcal{F}}

\newcommand{\calH}{\mathcal{H}}

\newcommand{\calO}{\mathcal{O}}

\newcommand{\calR}{\mathcal{R}}
\newcommand{\calS}{\mathcal{S}}

\newcommand{\calV}{\mathcal{V}}

\newcommand{\fraka}{\mathfrak a}
\newcommand{\frakb}{\mathfrak b}

\newcommand{\frakd}{\mathfrak d}

\newcommand{\frakn}{\mathfrak n}

\newcommand{\frakz}{\mathfrak z}

\newcommand{\xx}{\lambda}
\newcommand{\eps}{\varepsilon}
\newcommand{\bs}{\backslash}
\newcommand{\norm}{\operatorname{N}}
\newcommand{\vol}{\operatorname{vol}}
\newcommand{\tr}{\operatorname{tr}}

\newcommand{\Log}{\operatorname{log}}

\newcommand{\Sl}{\operatorname{SL}}

\newcommand{\SL}{\operatorname{SL}}

\newcommand{\Orth}{\operatorname{O}}
\newcommand{\Uni}{\operatorname{U}}
\newcommand{\Hom}{\operatorname{Hom}}
\newcommand{\CT}{\operatorname{CT}}
\newcommand{\Aut}{\operatorname{Aut}}

\newcommand{\Spec}{\operatorname{Spec}}

\newcommand{\End}{\operatorname{End}}
\newcommand{\Iso}{\operatorname{Iso}}
\newcommand{\dv}{\operatorname{div}}

\newcommand{\Gr}{\operatorname{Gr}}

\newcommand{\Gal}{\operatorname{Gal}}

\newcommand{\Pic}{\operatorname{Pic}}

\newcommand{\Div}{\operatorname{Div}}

\newcommand{\ord}{\operatorname{ord}}

\newcommand{\gen}{\operatorname{gen}}

\newcommand{\reg}{\operatorname{reg}}

\begin{document}

\title[Heights of Kudla-Rapoport divisors]{Heights of Kudla-Rapoport divisors and derivatives of $L$-functions}

\author[Jan H.~Bruinier, Benjamin Howard, and Tonghai Yang]{Jan Hendrik Bruinier, Benjamin Howard, and Tonghai Yang}
\address{Fachbereich Mathematik, Technische Universit\"at Darmstadt, Schlossgartenstr.~7,
D--64289 Darmstadt, Germany}
\email{bruinier@mathematik.tu-darmstadt.de}
\address{Department of Mathematics, Boston College, 140 Commonwealth Ave, Chestnut Hill, MA 02467, USA}
\email{howardbe@bc.edu}
\address{Department of Mathematics, University of Wisconsin Madison, Van Vleck Hall, Madison, WI 53706, USA}
\email{thyang@math.wisc.edu}

\subjclass[2000]{14G35, 14G40, 11G18, 11F27}

\thanks{The first author is partially supported by DFG grant BR-2163/4-1.
The second author is partially supported by NSF grant DMS-1201480.
The third author is partially supported by a NSF grant DMS-1200380 and a Chinese grant.}


\begin{abstract}
We study special cycles on integral models of Shimura varieties associated with unitary
similitude groups of signature $(n-1,1)$.
We construct an arithmetic theta lift from harmonic Maass forms of weight $2-n$ to the
arithmetic Chow group of  the integral model of a unitary Shimura variety, by associating to a harmonic Maass form
$f$ a linear combination of Kudla-Rapoport divisors, equipped with the Green function given by the
regularized theta lift of $f$.

Our main result is an equality of two complex numbers: (1) the height pairing of the arithmetic theta lift of $f$ against a
CM cycle, and (2) the central  derivative of the  convolution $L$-function of a weight
$n$ cusp form (depending  on $f$) and the theta function of a positive definite hermitian lattice of rank $n-1$.
When specialized to the case $n=2$, this result can be viewed as a variant of the Gross-Zagier formula for
Shimura curves associated to unitary groups of signature $(1,1)$.
The proof relies on, among other things, a new method for computing improper arithmetic intersections.
\end{abstract}

\maketitle

\setcounter{tocdepth}{1}
\tableofcontents


\section{Introduction}


Let $\kk \subset \C$ be an imaginary quadratic field of odd discriminant  $d_\kk$, and
let $\mathfrak{d}_\kk$ be the different of $\kk$. Let $\chi_\kk$ be the
quadratic Dirichlet  character determined by  $\kk/\Q$.


\subsection{Motivation: heights of Heegner points}


To motivate the results of this paper, we first recall the famous results of Gross and Zagier
\cite{GZ}.  Fix  a  normalized new eigenform
\[
g \in S_2(\Gamma_0(N)),
\]
 and assume that $N$ and $\kk$ satisfy the usual Heegner hypothesis:  every prime divisor of $N$ splits in $\kk$.
This allows us to fix an ideal $\mathfrak{n} \subset \co_\kk$ satisfying $\co_\kk/\mathfrak{n} \iso \Z/N\Z$.
For any fractional $\co_\kk$-ideal $\mathfrak{a}$, the cyclic $N$-isogeny of elliptic curves
\[
y_\mathfrak{a} = [ \C/\mathfrak{a} \to \C/\mathfrak{n}^{-1} \mathfrak{a} ]
\]
defines a \emph{Heegner point} on $X_0(N)(H)$, where $H$ is the Hilbert class field of $\kk$.
If we define a weight $2$ cuspform
\[
\phi^{\mathrm{Heeg}}(\tau) = \sum_{m\ge 1} T_m(y_{\co_\kk} -\infty) \cdot q^m
\]
valued in $J_0(N)(H)$, where the $T_m$ are Hecke operators, then the Petersson inner product
\[
\phi^{\mathrm{Heeg}}(g) = \langle \phi^{\mathrm{Heeg}}, g \rangle_{\mathrm{Pet}} \in J_0(N)(H)\otimes \C
\]
is essentially the projection of the divisor $y_{\co_\kk} -\infty$ to the $g$-isotypic component of the Jacobian $J_0(N)$.

After  endowing the fractional ideal $\mathfrak{a}$ with the self-dual hermitian form
$
\langle x,y\rangle =  \norm(\fraka)^{-1} x\overline{y}  ,
$
we may construct the weight one theta series
\[
\theta_\mathfrak{a} (\tau)
= \sum_{x\in \mathfrak{a}} q^{\langle x,x\rangle} \in M_1(\Gamma_0(|d_\kk|) ,\chi_\kk).
\]
The  Rankin-Selberg convolution $L$-function $L(g,\theta_\mathfrak{a} ,s)$
 satisfies a functional equation  forcing it to vanish at $s=1$, and
the Gross-Zagier theorem implies
\[
\big[ \phi^{\mathrm{Heeg}}(g)  : y_{\mathfrak a} -\infty \big]_{\mathrm{NT}} = c\cdot L'(g,\theta_\mathfrak{a} ,1).
\]
Here $c$ is some explicit nonzero constant, and the pairing on the left is the N\'eron-Tate height.

The goal of this paper is to obtain similar results when $g$ is replaced by  a cusp form of weight $n\ge 2$,
the weight $1$ theta series  $\theta_\mathfrak{a}$ is replaced by a weight $n-1$ theta
series determined by a hermitian lattice of rank $n-1$,
and the Heegner points on modular curves are replaced by special cycles on Shimura
varieties associated to groups of unitary  similitudes.
There are earlier results of Zhang \cite{Zh} and Nekov\'a\v{r} \cite{Nek} on Gross-Zagier theorems for higher weight modular forms, but
those results differ from ours in two essential ways: (1) those authors work with height pairings of cycles on Kuga-Sato varieties
fibered over modular curves, while we work with height pairings on unitary Shimura varieties, and (2)
they work with theta series of weight $1$, while we work with theta series of weight $n-1$.


\subsection{Statement of the main result}


Our main result will be a Gross-Zagier-type formula for the central derivative of the convolution $L$-function of
a cusp form of any weight $n\ge 2$ with a theta series of weight $n-1$.  This formula will involve
 the intersection multiplicities of special cycles on a unitary Shimura  variety.  We begin by describing the
 Shimura variety.

For a pair of nonnegative integers $(p,q)$, denote by $M_{(p,q)}$ the moduli space of principally polarized
abelian varieties $A \to S$ over $\kk$-schemes, equipped with an action of $\co_\kk$ satisfying
the \emph{signature $(p,q)$ condition}: every $a\in \co_\kk$ acts on $\Lie(A)$ with characteristic polynomial
$(T-a)^p(T-\overline{a})^q$.  We require also that the Rosati involution on $\End(A)\otimes \Q$ restrict to
complex conjugation on the image of $\co_\kk$. The moduli space $M_{(p,q)}$ is a Deligne-Mumford stack,
smooth over $\kk$ of dimension $pq$, and is a disjoint union of Shimura varieties associated to unitary similitude groups.

The theory of integral models of the stacks $M_{(p,q)}$ remains incomplete, but we only need two special cases:
\begin{enumerate}
\item
 there is a smooth and proper stack $\mathcal{M}_{(p,0)}$ over $\co_\kk$ with generic fiber $M_{(p,0)}$,
\item
there is a regular and flat stack $\mathcal{M}_{(p,1)}$ over $\co_\kk$ with generic fiber $M_{(p,1)}$.
\end{enumerate}
The product
\[
\mathcal{M} = \mathcal{M}_{(1,0)} \times_{\co_\kk} \mathcal{M}_{(n-1,1)}
\]
is  an $n$-dimensional regular algebraic stack, flat over $\co_\kk$, and  is typically disconnected.
Moreover,   $\mathcal{M}$ has a canonical toroidal compactification $\mathcal{M}^*$,
whose boundary is a smooth divisor.

Let $\Lambda$ be a positive definite self-dual hermitian lattice  of rank $n-1$; that is,
a projective $\co_\kk$-module of rank $n-1$ endowed with a positive definite hermitian form
$\langle\cdot,\cdot \rangle$ inducing an isomorphism  $\Lambda\iso \Hom_{\co_\kk}(\Lambda,\co_\kk)$.
The   $\co_\kk$-stack
\[
\mathcal{Y} = \mathcal{M}_{(1,0)} \times_{\co_\kk} \mathcal{M}_{(0,1)} \times_{\co_\kk} \mathcal{M}_{(n-1,0)}
\]
is smooth and proper of relative dimension $0$, and the morphism  $\mathcal{Y} \to \mathcal{M}$ defined by
\[
(A_0,A_1,B) \mapsto (A_0, A_1\times B)
\]
allows us to view $\mathcal{Y}$ as a $1$-dimensional cycle on $\mathcal{M}$.  To every geometric point
$(A_0,A_1,B)$ of $\mathcal{Y}$ there is an associated self-dual hermitian $\co_\kk$-module $\Hom_{\co_\kk}(A_0,B)$
of signature $(n-1,0)$, whose  isomorphism class is constant on each connected component of $\mathcal{Y}$.
Let $\mathcal{Y}_\Lambda \subset \mathcal{Y}$
be the union of all connected components on which $\Hom_{\co_\kk}(A_0,B) \iso \Lambda$.

To a  hermitian module $\VV$ over the adele ring $\A_\kk$ there is an associated
\emph{invariant} $\mathrm{inv}(\VV) \in \{\pm 1\}$, defined as a product of local invariants.
If $\mathrm{inv}(\VV)=1$ then $\VV$ is \emph{coherent}, in the sense that $\VV$
arises as the adelization of a hermitian space over $\kk$.  Otherwise, $\VV$ is
\emph{incoherent}.  In Section \ref{ss:hermitian} we define the notion of a
\emph{hermitian $(\kk_\R, \widehat{\co}_\kk)$-module} $\LL$.   Essentially,
$\LL$ is an integral structure on a  hermitian $\A_\kk$-module.  It consists of
an archimedean part $\LL_\infty$, which is a hermitian space over $\kk_\R=\kk\otimes_\Q\R$,
and a finite part $\LL_f$, which is a hermitian $\widehat{\co}_\kk$-module.

As explained in Section \ref{s:moduli spaces}, to each point of the moduli space $\mathcal{M}$ there is
associated an incoherent  hermitian $(\kk_\R, \widehat{\co}_\kk)$-module, whose isomorphism
class is constant on the connected components of $\mathcal{M}$.  Thus we obtain a decomposition
$
\mathcal{M} = \bigsqcup_\LL \mathcal{M}_\LL
$
where  $\LL$ runs over all incoherent self-dual hermitian $(\kk_\R,\widehat{\co}_\kk)$-modules
of signature $(n,0)$, and similarly for the compactification
\[
\mathcal{M}^* = \bigsqcup_\LL \mathcal{M}_\LL^*.
\]
The stack $\mathcal{Y}_\Lambda$ admits an analogous decomposition
\[
\mathcal{Y}_\Lambda = \bigsqcup_{\LL_0} \mathcal{Y}_{(\LL_0,\Lambda)},
\]
where  $\LL_0$ runs over all incoherent self-dual hermitian $(\kk_\R,\widehat{\co}_\kk)$-modules
of signature $(1,0)$.  From now on we fix one such $\LL_0$, and set
$\LL = \LL_0\oplus\Lambda$; for the meaning of the direct sum, see Remark \ref{rem:sum}.
The morphism $\mathcal{Y}_\Lambda \to \mathcal{M}^*$ restricts to a morphism
\begin{equation}\label{intro:cm cycle}
\mathcal{Y}_{(\LL_0,\Lambda)} \to \mathcal{M}_\LL^* ,
\end{equation}
which allows us to view $\mathcal{Y}_{(\LL_0,\Lambda)}$ as a cycle on $ \mathcal{M}_\LL^*$ of dimension $1$.

Let $\widehat{\mathrm{CH}}^1_\C(\mathcal{M}_\LL^*)$ be the codimension one arithmetic Chow group with complex coefficients,
defined,  as in the work of Gillet-Soul\'e \cite{SABK}, as the space of rational equivalence classes
of divisors on $\mathcal{M}_\LL^*$ endowed with Green functions.   In fact, we use the more general
arithmetic Chow groups defined by Burgos-Kramer-K\"uhn \cite{BKK}, which allow for Green functions
with  $\log$-$\log$ singularities along the boundary.   The map  (\ref{intro:cm cycle}) induces a  linear functional
\[
\widehat{\mathrm{CH}}^1_\C(\mathcal{M}_\LL^*) \to \C
\]
called the \emph{arithmetic degree along $\mathcal{Y}_{  (\LL_0,\Lambda)  }$}, and  denoted
$\widehat{\mathcal{Z}} \mapsto [\widehat{\mathcal{Z}}    : \mathcal{Y}_{  (\LL_0,\Lambda)  } ]$.

The hermitian form on the $\widehat{\co}_\kk$-module $\LL_f$ determines a $\Q/\Z$-valued quadratic form on the finite
discriminant group  $\frakd_\kk^{-1}\LL_f/\LL_f$.  If we denote by $S_\LL$  the (finite dimensional)
space of complex valued functions on this finite quadratic space, there is a Weil representation
\[
\omega_\LL : \SL_2(\Z) \to \Aut( S_\LL).
\]
 Let $H_{2-n}(\omega_\LL)$ be the space of  harmonic Maass forms for $\SL_2(\Z)$ of weight $2-n$
 with values in the vector space $S_\LL$, transforming according to $\omega_\LL$.

As explained in Section~\ref{s:green functions}, there is an arithmetic theta lift of harmonic Maass forms
\[
H_{2-n}(\omega_\LL)^\Delta \longrightarrow \widehat{\mathrm{CH}}_\C^1(\mathcal{M}^*_\LL),
\]
denoted $f\mapsto  \widehat{\Theta}_\LL(f)$, whose definition is roughly as follows.
There is a theta lift from  functions on the upper half plane to functions on the
Shimura variety $\mathcal{M}_\LL(\C)$.  If one attempts to lift an element
$f\in H_{2-n}(\omega_\LL)^\Delta$, the theta integral diverges due to the growth of $f$ at the
cusp.  There is a natural way to regularize the divergent integral in order to obtain a function $\Phi_\LL(f)$
on $\mathcal{M}_\LL(\C)$, but the  regularization process introduces singularities
of logarithmic type into the function $\Phi_\LL(f)$;  see \cite{Bo1} and \cite{BF}.  In fact
$\Phi_\LL(f)$ is a Green function for a certain divisor   $\mathcal{Z}_\LL(f)(\C)$ on $\mathcal{M}_\LL(\C)$,
which can be written in an explicit way as a linear combination of the complex Kudla-Rapoport divisors
$\mathcal{Z}_\LL(m,\mathfrak{r})(\C)$   introduced in \cite{KR1} and studied further in \cite{KR2}, \cite{Ho2}, and \cite{Ho3}.
Here $\mathfrak{r}$ is an $\co_\kk$-ideal dividing  $\mathfrak{d}_\kk$, and $m\in \norm(\mathfrak{r})^{-1} \Z$ is positive.
The complex Kudla-Rapoport  divisors are defined in terms of a moduli problem, and so have natural
extensions to the integral model $\mathcal{M}_\LL$.  Thus we obtain an extension of $\mathcal{Z}_\LL(f)(\C)$
to the integral model as well. The result is a divisor  $\mathcal{Z}_\LL(f)$ on  $\mathcal{M}_\LL$ together with a Green function $\Phi_\LL(f)$.
 The arithmetic theta lift of $f$ is then defined by first
adding boundary components with appropriate multiplicities in order to define a compactified
arithmetic divisor
\[
\widehat{\mathcal{Z}}^\mathrm{total}_\LL(f)\in \widehat{\mathrm{CH}}^1(\mathcal{M}_\LL^*),
\]
and then adding a certain multiple (depending on the constant term of $f$) of the
metrized cotautological bundle $\widehat{\mathcal{T}}^*_\LL$  of Section \ref{ss:cotaut bundle} to obtain
\[
\widehat{\Theta}_\LL(f)\in \widehat{\mathrm{CH}}^1(\mathcal{M}_\LL^*).
\]

\begin{remark}
One of the minor miracles of the construction of $\Phi_{\LL}(f)$ is that, despite having
a logarithmic singularity along $\mathcal{Z}_{\LL}(f)$, it is defined at \emph{every} point of the
complex  Shimura variety $\mathcal{M}_{\LL}(\C)$.
Expressed differently, the smooth function $\Phi_{\LL}(f)$, initially defined on the complement of
$\mathcal{Z}_{\LL}(f)$,
has a natural discontinuous extension to all points.  The behavior of $\Phi_{\LL}(f)$
at the points of $\mathcal{Z}_{\LL}(f)$, as described in Corollary \ref{cor:sing},
plays an essential role in our calculation of  improper intersections.
\end{remark}

\begin{remark}
The Green functions used here are constructed  as regularized theta lifts of harmonic Maass forms, as in \cite{Br1}, \cite{BF}, and
\cite{BY1}, and so are different from the  Kudla-style Green functions used in  \cite{Ho2} and \cite{Ho3}.
\end{remark}

Let $S_n(\overline{\omega}_\LL)$
be the space of weight $n$ cusp forms for $\SL_2(\Z)$ with values in $S_\LL$, transforming according
to the complex conjugate representation $\overline{\omega}_\LL$.
Denote by $\Delta$  the automorphism group of the finite group $\frakd_\kk^{-1}\LL_f/\LL_f$
with its $\Q/\Z$-valued quadratic form.   Any $\Delta$-invariant cusp form
\[
g (\tau)=\sum_{m\in \Q_{>0}} a(m) q^m  \in S_n(\overline{\omega}_\LL)^\Delta
\]
has  Fourier coefficients  $a(m) \in S_\LL^\Delta$.
Similarly, the space  $S_\Lambda$  of complex valued functions on $\mathfrak{d}_\kk^{-1} \Lambda/\Lambda$
has a natural action $\omega_\Lambda : \SL_2(\Z) \to \Aut(S_\Lambda)$, and
there  is a vector valued theta series
\[
\theta_\Lambda (\tau) = \sum_{m\in \Q_{>0}} R_\Lambda(m) q^m \in M_{n-1}(\omega_\Lambda^\vee)
\]
taking values in the dual space  $S_\Lambda^\vee$, whose
$m$-th Fourier coefficient $R_\Lambda(m) : S_\Lambda \to \C$ is  the representation number
\[
R_\Lambda(m , \varphi) =
\sum_{\substack{\lambda\in\frakd_\kk^{-1}\Lambda\\ \langle \lambda,\lambda \rangle=m}} \varphi(\lambda).
\]
We define the \emph{Rankin-Selberg convolution $L$-function}
\begin{equation} \label{eq:vectorRankin-Selbergintro}
L( g, \theta_\Lambda , s) =
 \Gamma \big(\frac{s}{2} +n-1 \big)
\sum_{m \in \Q_{>0}}  \frac{ \big\{\overline{ a(m)},  R_\Lambda(m) \big\}  }
{ (4\pi m)^{\frac{s}{2}+n-1 }},
\end{equation}
 where the pairing  $\{\cdot,\cdot\}$ is the tautological pairing between $S_\LL$ and $S_\LL^\vee$,
and $R_\Lambda(m)$ is viewed as an element of $S_\LL^\vee$ using the
natural surjection $S_\LL\to S_{\Lambda}$.
The $L$-function \eqref{eq:vectorRankin-Selbergintro} satisfies a simple functional equation in $s\mapsto -s$,
which forces it to vanish at $s=0$.

As in \cite{BF}, there is a $\Delta$-invariant surjective differential operator
\[
\xi :
H_{2-n}(\omega_\LL) \longrightarrow S_n(\overline{\omega}_\LL )
\]
defined by
\[
\xi(f)(\tau) = 2iv^{2-n} \frac{ \overline{\partial f} }{ \partial \overline{\tau}},
\]
where $\tau=u+iv$ is the variable on the upper half-plane.

The following is our main result.  It is stated in the text as Theorem \ref{thm:arithmetic degree}.

\begin{BigThm}\label{BigA}
Fix a $g\in S_n(\overline{\omega}_\LL)^\Delta $, and let $f \in H_{2-n}(\omega_\LL)^\Delta$ be any
harmonic form satisfying $\xi(f) = g$.
The arithmetic theta lift of $f$ and the $L$-function (\ref{eq:vectorRankin-Selbergintro}) are related by
\begin{equation}\label{intro main formula}
[ \widehat{\Theta}_\LL (f):  \mathcal{Y}_{(\LL_0,\Lambda)}    ] =
- \deg_\C \mathcal{Y}_{(\LL_0 , \Lambda) } \cdot L'( g , \theta_\Lambda , 0).
\end{equation}
The constant appearing on the right is
\[
\deg_\C \mathcal{Y}_{(\LL_0 , \Lambda) }
=
\sum_{ y\in \mathcal{Y}_{(\LL_0,\Lambda)}(\C)} \frac{1}{|\Aut(y)|}.
\]
An explicit formula for this constant is given in Remark \ref{rem:degree}.
\end{BigThm}

We prove   Theorem  \ref{BigA} by  first verifying (\ref{intro main formula}) for certain distinguished harmonic Maass forms $f=f_{m,\mathfrak{r}}$
satisfying $\mathcal{Z}_\LL(f) = \mathcal{Z}_\LL(m,\mathfrak{r})$.  The calculation of the left hand side
of (\ref{intro main formula})  is seriously complicated by the fact that the cycles $\mathcal{Z}_\LL(m,\mathfrak{r})$
and $\mathcal{Y}_{(\LL_0, \Lambda)}$  typically  intersect improperly.
 Calculations of improper intersection have been done in some low-dimensional situations elsewhere in the literature
(for example in \cite{GZ}, \cite{KRY2}, and \cite{Ho1}),  but our methods are new, and seem considerably more flexible than the
laborious  calculations of earlier authors.    The idea is to use deformation theory to show that the metrized line bundle
 \begin{equation}\label{intro twist}
\widehat{\mathcal{Z}}_\LL^\heartsuit(m,\mathfrak{r}) =
\widehat{\mathcal{Z}}_\LL(f_{m,\mathfrak{r}}) \otimes \widehat{\mathcal{T}}_\LL^{-R_\Lambda(m,\mathfrak{r})}
\end{equation}
on $\mathcal{M}_\LL$ acquires a canonical nonzero section $\bm{\sigma}_{m,\mathfrak{r}}$ when restricted to
$\mathcal{Y}_{(\LL_0, \Lambda)}$.  To compute the intersection multiplicity of (\ref{intro twist})
with $\mathcal{Y}_{(\LL_0,\Lambda)}$,  it suffices to compute the degree
of the $0$-cycle  $\mathrm{div}(\bm{\sigma}_{m,\mathfrak{r}})$ on $\mathcal{Y}_{(\LL_0,\Lambda)}$, and  the
norm $|| \bm{\sigma}_{m,\mathfrak{r}} ||_y$ at each $y \in \mathcal{Y}_{(\LL_0,\Lambda)}(\C)$.  The divisor
$\mathrm{div}(\bm{\sigma}_{m,\mathfrak{r}})$ turns out to be exactly the divisor obtained by
intersecting $\mathcal{Z}_\LL(m,\mathfrak{r}) \cap \mathcal{Y}_{(\LL_0,\Lambda)}$ and then throwing away
all components of the intersection having dimension $>0$. In other words, it is the proper
part of the intersection, which can be computed directly using results of Gross.  The norm
$|| \bm{\sigma}_{m,\mathfrak{r}} ||_y$ turns out to be the value of the Green function  $\Phi_\LL(f_{m,\mathfrak{r}})$
at $y$, even when $y$ lies on $\mathcal{Z}_\LL(m,\mathfrak{r})$, the singularity of the Green function!
  Thus we are able to compute the intersection multiplicity of (\ref{intro twist}) with $\mathcal{Y}_{(\LL_0,\Lambda)}$
by computing only proper intersections and   the CM values of Green functions.


\subsection{Applications and further directions of study}


In the spirit of \cite{Ku4}, let us consider only those Kudla-Rapoport divisors
\[
\widehat{\mathcal{Z}}^\mathrm{total}_\LL(m,\mathfrak{r}) = \widehat{\mathcal{Z}}^\mathrm{total}_\LL( f_{ m,\mathfrak{r}} )
\]
with  $\mathfrak{r}=\co_\kk$, and form the formal generating series
\[
\widehat{\phi}(\tau) = \widehat{\mathcal{T}}_\LL + \sum_{ m>0 }\widehat{\mathcal{Z}}^\mathrm{total}_\LL(m,\co_\kk) \cdot  q^m
\in  \widehat{\mathrm{CH}}_\C^1(\mathcal{M}^*_\LL) [[q]].
\]
When $n=2$ there is some mild ambiguity in the choice of harmonic Maass form $f_{m,\mathfrak{r}}$, and hence in the
choice of Green function in the arithmetic divisor $\widehat{\mathcal{Z}}^\mathrm{total}_\LL(m,\mathfrak{r})$.
See Lemma \ref{lem:f-mr} and the remark that follows it.
Because of this technical issue, in this subsection we assume that $n>2$.

\begin{BigConj}\label{BigB}
The formal generating series $\widehat{\phi}$ is a modular form of weight $n$, level
$\Gamma_0( |d_\kk|)$, and character $\chi_\kk^n$.  In other words
\[
\widehat{\phi} \in \widehat{\mathrm{CH}}_\C^1(\mathcal{M}^*_\LL)  \otimes M_n(\Gamma_0( |d_\kk| ) ,\chi_\kk^n ).
\]
\end{BigConj}

This conjecture should be taken with a small grain of salt:  to achieve modularity it may be necessary to
slightly modify the formal  generating series by  vertical divisors  on
$\mathcal{M}^*_\LL$ supported at the primes dividing $\mathfrak{d}_\kk$.  In any case, some form of this conjecture
is  certainly true, and is the subject of ongoing investigations of Kudla, Rapoport, and the three authors.  Indeed,
if one replaces the unitary Shimura variety by an orthogonal Shimura variety, and works only in the Chow group of the generic fiber
rather than in the arithmetic Chow group of an integral model, the corresponding modularity result is due to
Borcherds \cite{Bo2}.

Theorem \ref{BigA} gives evidence for Conjecture \ref{BigB} as it is currently stated.
 Indeed, the theorem implies that $[\widehat{\Theta}_\LL(f) : \mathcal{Y}_{ ( \LL_0,\Lambda) } ] =0$
for all $f\in H_{2-n}(\omega_\LL)^\Delta$ with $\xi(f)=0$.    The following corollary of Theorem \ref{BigA} can be deduced
from this and the modularity criterion \cite{Bo2} of Borcherds.
We omit the details of the proof, as we expect to prove some form Conjecture \ref{BigB} in the near future.

\begin{BigThm}
The formal $q$-expansion
\[
[ \widehat{\phi}(\tau) : \mathcal{Y}_{ ( \LL_0,\Lambda) } ] = [ \widehat{\mathcal{T}}_\LL : \mathcal{Y}_{ ( \LL_0,\Lambda) } ]
+  \sum_{m>0}  [ \widehat{\mathcal{Z}}_\LL(m,\co_\kk) : \mathcal{Y}_{ ( \LL_0,\Lambda) } ] \cdot q^m
\]
defines an element of $M_n(\Gamma_0( |d_\kk| ) ,\chi_\kk^n )$.
\end{BigThm}

Suppose that Conjecture \ref{BigB} is true.  Given a scalar valued  form $g_0 \in S_n(\Gamma_0(|d_\kk|) , \chi_\kk^n)$
we may then imitate \cite{Ku4} and form the Petersson inner product
\[
\widehat{\phi}(g_0) = \langle  \widehat{\phi} , g_0 \rangle_{\mathrm{Pet}} \in \widehat{\mathrm{CH}}_\C^1(\mathcal{M}^*_\LL) .
\]
The form $g_0$ determines a vector valued form
\[
g(\tau) =\sum_{\gamma \in \Gamma_0(D) \backslash \SL_2(\Z)}
   (g_0|_n\gamma)(\tau)  \cdot \overline{\omega_\LL(\gamma^{-1}) \varphi_0}\in S_n(\overline{\omega}_\LL)^\Delta,
\]
where $\varphi_0\in S_\LL$ is the characteristic function of $0\in \mathfrak{d}_\kk^{-1}\LL_f/\LL_f$.
Now pick any $f\in H_{n-2}(\omega_\LL)^\Delta$ satisfying $\xi(f)=g$.   Using \cite[Theorem 1.1]{BF} one can show that
$\widehat{\Theta}_\LL(f) = \widehat{\phi}(g_0) $
and so, assuming Conjecture \ref{BigB},  Theorem \ref{BigA} may be restated as
\[
[ \widehat{\phi}(g_0) :  \mathcal{Y}_{(\LL_0,\Lambda)}    ] =
- \deg_\C \mathcal{Y}_{(\LL_0 , \Lambda) } \cdot L'( g , \theta_\Lambda , 0).
\]
Under some mild restrictions
(for example, assuming that $n$ is even and that $g_0$ is a newform) the $L$-function on the right can be
expressed in terms of the classical Rankin-Selberg $L$-function of the scalar valued form $g_0$ and the scalar valued
theta series
\[
\sum_{ \lambda \in \Lambda} q^{\langle \lambda,\lambda \rangle } \in M_{n-1}(\Gamma_0(|d_\kk|) , \chi_\kk^{n-1} ).
\]
The statement and the proof of the precise relation between $L$-functions are slightly involved.  We hope to explore this
reformulation of Theorem \ref{BigA} in terms of scalar valued holomorphic forms in a future work, after  Conjecture \ref{BigB}
has been proved.

Apart from providing evidence for Conjecture \ref{BigB}, our  methods have applications to Colmez's conjectural
extension \cite{Co} of the Chowla-Selberg formula to CM abelian varieties of arbitrary dimension.  Very roughly,
the idea is this:  after fixing a totally real field $F/\Q$ of degree $n$, one can replace the  cycle $\mathcal{Y}_{( \LL_0,\Lambda)}$
by a cycle $\mathcal{Y}_E$ on $\mathcal{M}^*_\LL$ formed from abelian varieties with complex multiplication by the CM field $E=\kk\otimes_\Q F$.
It is  expected that a variant of Theorem \ref{BigA} holds for this new cycle $\mathcal{Y}_E$, and some results in this direction
can be found in \cite{Ho2}.  However, the proof of Theorem \ref{BigA} uses the Chowla-Selberg formula
in an essential way, and so without \emph{a priori} knowledge of Colmez's conjectural extension,
one cannot complete the proof of the desired variant of Theorem \ref{BigA} without using some additional tools.  The results of \cite{YaColmez}
suggest that Conjecture \ref{BigB} is the new tool needed, and that a proof of new cases of
Colmez's conjecture can be deduced as a byproduct of the proof of the variant of Theorem \ref{BigA}.
In short, once Conjecture \ref{BigB} is proved, the methods of this paper will yield the proof of Colmez's conjecture
for all CM abelian varieties that appear as points of the moduli space $M_{(n-1,1)}$.  Again, this application
is being investigated by Kudla, Rapoport, and the three authors.


\subsection{Notation and terminology}

We write $\H$ for the complex upper half plane. For a complex number $z$ we put $e(z)=e^{2\pi i z}$.
As usual, we denote by $\A$  the ring of adeles of $\Q$  and write $\A_f$ for the finite adeles.

The quadratic imaginary field $\kk$ and its embedding $\kk\hookrightarrow \C$ are
fixed throughout the paper, and $\mathfrak{d}_\kk$ and $d_\kk$ denote the
different and discriminant of $\kk$.
In Section \ref{s:green functions} we make no restriction on $d_\kk$, but throughout the rest of the paper
 we assume that $d_\kk$ is odd.   Write $\calO_\kk$, $\A_\kk$ and $\A_{\kk,f}$ for the ring of integers,
 adeles, and finite adeles  of $\kk$, respectively. The class number of $\kk$ is $h_\kk$, and
$w_\kk=|\mu(\kk)|$ is the number of roots of unity in $\kk$.
Denote by $o(d_\kk)$ the number of distinct prime divisors of $d_\kk$, and by
\[
\chi_\kk:  \A^\times \longrightarrow \{\pm 1\}
\]
the quadratic character determined by the extension $\kk/\Q$. For any $m\in \Q_{>0}$ define
\begin{equation}\label{rho}
\rho(m) =| \{ \mathfrak{b} \subset\co_\kk : \mathrm{N}(\mathfrak{b}) = m \} | .
\end{equation}
Obviously $\rho(m)=0$ unless $m\in \Z_{>0}$.
Abbreviate $\kk_\R=\kk\otimes_\Q \R$.
For a positive integer $m$ we denote by $\sigma_1(m)$ the sum of the positive divisors of $m$,
and set  $\sigma_1(0)=- 1/ 24$.


\subsection*{Acknowledgements}


We thank the referee for his/her careful reading of our manuscript and for the insightful
comments.


\section{Hermitian spaces and modular forms.}


\label{sect:prelims}

This section contains some basic definitions and notation concerning hermitian spaces,
theta series, and vector valued modular forms.


\subsection{Invariants of hermitian spaces}
\label{ss:hermitian}


A \emph{hermitian $\co_\kk$-module} is a projective $\co_\kk$-module $L$ of finite
rank equipped with a hermitian  form
$\langle\cdot,\cdot\rangle: L\times L\to \co_\kk$.
Our convention is that hermitian forms are $\co_\kk$-linear in the first variable
and $\co_\kk$-conjugate-linear in the  second variable.  All hermitian forms are assumed to be nondegenerate.
For an $\co_\kk$-ideal  $\mathfrak{r} \mid \mathfrak{d}_\kk$,
every vector $x\in \mathfrak{r}^{-1} L$ satisfies
\begin{equation}\label{denom bound}
\langle x,x\rangle \in  \mathrm{N}( \mathfrak{r} )^{-1} \Z,
\end{equation}
and $Q(x) = \langle x,x\rangle$ defines a $d_\kk^{-1} \Z/\Z$-valued quadratic form on $\mathfrak{d}_\kk^{-1} L/L$.
A hermitian $\co_\kk$-module $L$  is \emph{self-dual} if it satisfies
\[
L = \{ x\in L\otimes_\Z\Q : \langle x, L\rangle\subset \co_\kk \}.
\]
We can similarly talk about self-dual hermitian $\widehat{\co}_\kk$-modules, and
hermitian spaces over $\kk$, over its completions,  and over $\A_\kk$.

If $\mathfrak{A}_0$ and $\mathfrak{A}$ are hermitian $\co_\kk$-modules with
hermitian forms $h_{\mathfrak{A}_0}$ and $h_\mathfrak{A}$, the $\co_\kk$-module
\begin{equation}\label{betti hom}
L(\mathfrak{A}_0, \mathfrak{A}) = \Hom_{\co_\kk}(\mathfrak{A}_0,\mathfrak{A})
\end{equation}
carries a hermitian form $\langle \cdot,\cdot\rangle$ characterized by the relation
\[
\langle f , g \rangle \cdot h_{\mathfrak{A}_0} (x,y) =
h_\mathfrak{A}  (f(x), g(y) )
\]
for all $x,y\in \mathfrak{A}_0$.  If $\mathfrak{A}_0$ and $\mathfrak{A}$ are self-dual
then so is $L(\mathfrak{A}_0,\mathfrak{A})$.  Of course a similar discussion holds for
hermitian $\widehat{\co}_\kk$-modules.

A hermitian space $\VV$ over $\A_\kk$ has an
archimedean part $\VV_\infty$ and a nonarchimedean part
$\VV_f = \prod_p \VV_p$, which are hermitian spaces over
$\kk_\R$ and $\A_{\kk,f}$, respectively.
The archimedean part is uniquely determined by its signature, while each factor $\VV_p$
is uniquely determined by its dimension  and the \emph{local invariant}
\[
\mathrm{inv}_p(\VV) = \chi_{\kk,p}(\det( \VV_p)) \in \{\pm 1\}.
\]
Of course the invariant is also defined for $p=\infty$, but carries less information than
the signature.  The \emph{invariant} of $\VV$ is the product of local invariants:
\[
\mathrm{inv}(\VV)  = \prod_{p\le \infty} \mathrm{inv}_p(\VV).
\]
If $\mathrm{inv}(\VV)=1$ then there is a hermitian space $V$ over $\kk$, unique
up to isomorphism, satisfying  $\VV \iso V\otimes_{\Q} \A$.
In this case we say that $\VV$ is \emph{coherent}. If instead
$\mathrm{inv}(\VV)=-1$ then no such $V$ exists, and we say that $\VV$ is
 \emph{incoherent}.

We will need a notion of a hermitian space over $\A_\kk$ with an integral structure.

\begin{definition}
A \emph{hermitian $(\kk_\R,\widehat{\co}_\kk)$-module}
is a  hermitian space $\VV$
over $\A_\kk$ together with a finitely generated $\widehat{\co}_\kk$-submodule $\LL_f\subset \VV_f$
of maximal rank on which the hermitian form is $\widehat{\co}_\kk$-valued.
\end{definition}

Equivalently, we could define a  hermitian $(\kk_\R,\widehat{\co}_\kk)$-module as a pair
$\LL=(\LL_\infty,\LL_f)$
in which  $\LL_\infty$ is a hermitian space over $\kk_\R$,
and $\LL_f=\prod_p\LL_p$  is a hermitian space over $\widehat{\co}_\kk$
of the same rank  as $\LL_\infty$.
One recovers the first definition from the second by setting
$\VV_\infty=\LL_\infty$ and $\VV_f = \LL_f\otimes_{\widehat{\Z}} \A_f$.
We use the following terminology.
\begin{enumerate}
\item
The \emph{signature} of  a hermitian $(\kk_\R,\widehat{\co}_\kk)$-module $\LL$
is the signature of $\LL_\infty$,
\item
$\LL$ is \emph{self-dual} if $\LL_f$ is a self-dual
hermitian $\widehat{\co}_\kk$-module,
\item
$\LL$ is \emph{coherent} (or \emph{incoherent}) if $\VV$ is.
\end{enumerate}

Obviously, every hermitian $\co_\kk$-module $L$ gives rise to a coherent hermitian
$(\kk_\R,\widehat{\co}_\kk)$-module $\LL$ determined by
$\LL_\infty = L\otimes_{\Z} \R$ and $\LL_f = L \otimes_\Z \widehat{\Z}$.
Conversely, for each hermitian $(\kk_\R,\widehat{\co}_\kk)$-module $\LL$
there is a (possibly empty) finite collection of hermitian $\co_\kk$-modules that give rise to it.
This finite collection is the \emph{genus of $\LL$},
and is denoted
\begin{equation}\label{genus def}
\mathrm{gen}(\LL) = \left\{
\begin{array}{l}
\mbox{isomorphism classes of} \\
\mbox{hermitian $\co_\kk$-modules $L$ }
\end{array}
:
\begin{array}{c}
\LL_\infty \iso L \otimes_\Z \R \\
\LL_f \iso L\otimes_{\Z} \widehat{\Z}
\end{array}
\right\}.
\end{equation}
The genus is nonempty if and only if $\LL$ is coherent, and any  two
$L, L' \in \mathrm{gen}(\LL)$ satisfy $L\otimes_\Z\Q \iso L'\otimes_\Z \Q$
as hermitian spaces over $\kk$.

\begin{remark}\label{rem:nearby}
Given a hermitian space $\VV$ over $\A_\kk$ and a rational prime $p$
nonsplit in $\kk$, there is a   \emph{nearby} hermitian space
$\VV(p)$ over $\A_\kk$  determined up to isomorphism by the conditions
\begin{enumerate}
\item
$\VV(p)_\ell \iso \VV_\ell$ for every place $\ell\not=p$,
\item
$\VV(p)_p \not\iso \VV_p$.
\end{enumerate}
In other words,  $\VV(p)$ is obtained from $\VV$ by changing
the local  invariant at $p$, and so
\[
\mathrm{inv}(\VV(p)) = -\mathrm{inv}(\VV).
\]
If instead we take $p=\infty$ then there is no single notion of $\VV(\infty)$.  However, in the applications
 $\VV$ will be positive definite,  and $\VV(\infty)$ will be
obtained from $\VV$ by switching the signature from $(n,0)$ to $(n-1,1)$.
\end{remark}


\subsection{Theta functions and vector valued modular forms}


Let $(M,Q)$ be an even integral lattice, that is, a free $\Z$-module of finite rank equipped with a non-degenerate $\Z$-valued quadratic form $Q$.  For simplicity we assume here that the rank of $M$ is even. We denote the signature of $M$ by $(b^+,b^-)$. Let $M'$ be the dual lattice of $M$. The quadratic form $Q$ induces a $\Q/\Z$-valued quadratic form on the discriminant group $M'/M$.

Let $\omega$ be the restriction to $\Sl_2(\Z)$ of the
Weil representation of $\Sl_2(\widehat \Q)$ (associated with the standard additive character of $\A/\Q$) on the Schwartz-Bruhat functions on $M\otimes_\Z\widehat\Q$. The restriction of $\omega$ to $\Sl_2(\Z)$ takes the subspace
$S_M$ of Schwartz-Bruhat functions which are supported on $\widehat M'$ and invariant under translations by $\widehat M$ to itself.
We obtain a representation $\omega_M:\Sl_2(\Z)\to \Aut(S_M)$. Throughout we identify $S_M$ with the space of functions $M'/M\to \C$. Let $S_M^\vee$ be the dual space of $S_M$, and denote by
\[
\{\cdot,\cdot\} : S_M\times S_M^\vee\longrightarrow \C
\]
the tautological $\C$-bilinear pairing. The group $\Sl_2(\Z)$ acts on $S_M^\vee$ through the dual representation $\omega_M^\vee$, given by
$\omega_M^\vee(\gamma)(f)=f\circ \omega_M(\gamma^{-1})$ for $f\in S_M^\vee$. On the space $S_M$ we also have the conjugate representation $\overline{\omega}_M$ given by
\[
\overline{\omega}_M(\gamma)(\varphi)= \overline{\omega_M(\gamma)(\overline{\varphi})}
\]
for $\varphi\in S_M$. Note that $\overline{\omega}_M$ is the representation denoted $\rho_M$ in \cite{Bo1}, \cite{Br1}, \cite{BF}. The same construction can also be applied in slightly greater generality.
For instance, in later applications we will use it when $M$ is a quadratic module over $\widehat \Z$.

Let $\Gr(M)$ be the Grassmannian of negative definite $b^-$-dimensional subspaces of $M\otimes_\Z \R$.
For $z\in \Gr(M)$ and $\xx\in M\otimes_\Z \R$, we denote by $\xx_{z}$ and $\xx_{z^\perp}$ the orthogonal projection of $\xx$ to $z$ and $z^\perp$, respectively.
If $\varphi\in S_M$, and $\tau\in \H$ with $v=\Im(\tau)$, we let
\[
\Theta_M(\tau,z,\varphi) = v^{b^-/2} \sum_{\xx\in M'} \varphi(\xx)
e\big(Q(\xx_{z^\perp})\tau+Q(\xx_{z})\overline{ \tau}\big)
\]
be the associated Siegel theta function.
For $\gamma\in \Sl_2(\Z)$ it satisfies the transformation law
\[
\Theta_M(\gamma\tau,z,\varphi) = (c\tau +d)^{\frac{b^+-b^-}{2}} \Theta_M(\tau,z,\omega_M(\gamma)\varphi).
\]
Following \cite{Ku:Integrals}, we view the Siegel theta function as a function
\[
 \H \times\mathrm{Gr}(M) \longrightarrow S_M^\vee,\quad (\tau,z)\mapsto \Theta_M(\tau,z).
\]
The above transformation law implies that
$\Theta_M(\tau,z)$ transforms as a (non-holomorphic) modular form of weight $(b^+-b^-)/2$ for the group $\Sl_2(\Z)$ with values in $S_M^\vee$.

Let $k\in \Z$, and let $\sigma$ be a finite dimensional representation of $\Sl_2(\Z)$
on a complex vector space $V_\sigma$,  which factors through a finite quotient of $\Sl_2(\Z)$.
We denote by $H_k(\sigma)$ the vector space of harmonic Maass forms\footnote{More precisely, these are the \emph{harmonic weak Maass forms} of \cite{BY1}. For simplicity we omit the adjective `weak'.} of weight $k$ for the group $\Sl_2(\Z)$ with representation $\sigma$ as in \cite{BY1}. We write $M_k^!(\sigma)$, $M_k(\sigma)$, and $S_k(\sigma)$ for the subspaces of weakly holomorphic modular forms, holomorphic modular forms, and cusp forms, respectively.
Taking $V_\sigma=M$ and $\sigma$ to be the Weil representation, the natural action of the orthogonal group of $M$
on  $S_M$ commutes with the action of $\Sl_2(\Z)$, and
hence there is an induced action on the above spaces of $S_M$-valued modular forms.

A harmonic Maass form $f\in H_k(\sigma)$ has a Fourier expansion of the form
\begin{align}
\label{eq:fourierf}
f(\tau)=\sum_{\substack{m\in \Q\\ m\gg -\infty }} c^+(m) q^m
+\sum_{\substack{m\in \Q \\ m<0 } } c^-(m) \Gamma(1-k, 4\pi |m| v) q^m
\end{align}
with Fourier coefficients $c^\pm(m)\in V_\sigma$. Here $q=e^{2\pi i\tau}$, and $\Gamma(s,x)=\int_x^\infty e^{-t}t^{s-1}dt$ denotes the incomplete gamma function.
The coefficients are supported on rational numbers with uniformly bounded denominators. The first summand on the right hand side of \eqref{eq:fourierf}
is denoted by $f^+$ and is called the \emph{holomorphic part} of $f$, the second summand is denoted by $f^-$ and is called the \emph{non-holomorphic part}.

Recall from \cite{BF} the conjugate-linear differential operator $\xi_k: H_k(\omega_M)\to S_{2-k}(\overline{\omega}_M)$ defined  by
\begin{align}
\label{defxi}
\xi_k(f)(\tau)=2iv^{k} \overline{\frac{\partial f}{\partial \overline{\tau}}}.
\end{align}
The kernel of $\xi_k$ is equal to $M^!_{k}(\omega_M)$.
According to \cite[Corollary~3.8]{BF}  there is  an exact sequence
\[
\xymatrix{ 0\ar[r]& M^!_{k}(\omega_M) \ar[r]& H_{k}(\omega_M)
\ar[r]^{\xi_k}& S_{2-k}(\overline{\omega}_M) \ar[r] & 0 }.
\]
If $f\in H_k(\omega_M)$ has Fourier coefficients $c^\pm(m)\in S_M$ as in \eqref{eq:fourierf},
 we abbreviate
$
c^\pm(m,\mu)=c^\pm(m)(\mu)\in \C
$
for all  $\mu\in M'/M$.


\section{Divisors on unitary Shimura varieties}
\label{s:moduli spaces}


In this section we introduce the arithmetic Shimura variety $\mathcal{M}$ on which we
will be doing intersection theory, and introduce the Kudla-Rapoport divisors on $\mathcal{M}$.
Recall  that $d_\kk$ is odd.  This hypothesis will be used in several places, but the primary reason for imposing it is that without
this assumption  the integral model $\mathcal{M}$ is not known (or necessarily expected)  to be flat or regular.


\subsection{The stack $\mathcal{M}$ and the Kudla-Rapoport divisors}
\label{ss:moduli}


We first define $\co_\kk$-stacks $\mathcal{M}_{(m,0)}$ and $\mathcal{M}_{(m,1)}$ as moduli
spaces of abelian schemes with additional structure.

\begin{definition}
Let $\mathcal{M}_{(m,0)}$ be the algebraic stack over $\co_\kk$ whose functor of points assigns to
an $\co_\kk$-scheme $S$ the groupoid of triples $(A,\pol,i)$, in which
\begin{itemize}
\item
$A$ is an abelian scheme over $S$ of relative dimension $m$,
\item
$\pol:A \to A^\vee$ is a principal polarization,
\item
$i:\co_\kk\to \End(A)$ is an action of $\co_\kk$ on $A$.
\end{itemize}
We insist  that  the polarization  $\pol$ be  $\co_\kk$-linear, in the sense that
$\pol \circ i (\overline{x}) = i (x)^\vee \circ \pol$  for every  $x\in \co_\kk$.  We further
 insist that the action of $\co_\kk$ satisfy the \emph{signature $(m,0)$ condition}: the
 induced action of $\co_\kk$ on the $\co_S$-module $\Lie(A)$ is through the
structure morphism $\co_\kk\to \co_S$.
\end{definition}

We usually just write $A\in \mathcal{M}_{(m,0)}(S)$
 for an $S$-valued point, and suppress $\pol$ and $i$ from the notation.
It is proved in   \cite{Ho3} that the stack $\mathcal{M}_{(m,0)}$  is  smooth and proper of relative dimension $0$
over    $\co_\kk$. The stack $\mathcal{M}_{(0,m)}$ is defined in  the same way,
but the signature condition is replaced by the \emph{signature $(0,m)$ condition}: the induced action of
$\co_\kk$ on $\Lie(A)$ is through the complex conjugate of the structure morphism $\co_\kk\to \co_S$.


\begin{definition}
Let $\mathcal{M}_{(m,1)}$ be the algebraic stack over $\co_\kk$ whose functor of points assigns to
an $\co_\kk$-scheme $S$ the groupoid of quadruples $(A,\pol,i,\mathcal{F})$ in which
\begin{itemize}
\item
$A$ is an abelian scheme over $S$ of relative dimension $m+1$,
\item
$\pol:A \to A^\vee$ is a principal polarization of $A$,
\item
$i:\co_\kk\to \End(A)$ is an action of $\co_\kk$ on $A$,
\item
$\mathcal{F} \subset \Lie(A)$ is an $\co_\kk$-stable $\co_S$-submodule, which is locally an $\co_S$-module
direct summand of rank $m$.
\end{itemize}
We again insist that $\pol$ be $\co_\kk$-linear, and that the subsheaf $\mathcal{F}$
satisfy  \emph{Kr\"amer's signature $(m,1)$ condition}: the action of $\co_\kk$ on $\mathcal{F}$
is through the structure morphism $\co_\kk\to \co_S$, while the action of $\co_\kk$ on the
line bundle $\Lie(A)/\mathcal{F}$ is through the complex conjugate of the structure morphism.
\end{definition}

When no confusion will arise, we denote $S$-valued points simply by $A\in \mathcal{M}_{(m,1)}(S)$.
By work of Pappas \cite{Pa} and Kr\"amer \cite{Kr}, the stack $\mathcal{M}_{(m,1)}$ is known to be
 regular and flat over  $\co_\kk$ of relative dimension $m$, and to be  smooth over $\co_\kk[1/d_\kk]$.

From now on we fix an integer $n\ge 2$ and define a regular and flat $\co_\kk$-stack
\[
\mathcal{M} = \mathcal{M}_{(1,0)} \times_{\co_\kk} \mathcal{M}_{(n-1,1)}
\]
of dimension $n$.  If  $S$ is a connected $\co_\kk$-scheme and $(A_0,A)\in \mathcal{M}(S)$,  the  $\co_\kk$-module
 \[
L(A_0,A) = \Hom_{\co_\kk}(A_0,A)
 \]
carries a positive definite  hermitian form
$\langle x ,  y\rangle =   \pol_0^{-1} \circ y^\vee \circ \pol    \circ x,$
where the composition on the right is viewed as an element of $\co_\kk\iso \End_{\co_\kk}(A_0)$.

In the special case where $S=\Spec(\F)$ for an algebraically closed field $\F$,
and $\ell\not=\mathrm{char}(\F)$ is a prime, the $\co_{\kk,\ell}$-module
$ \Hom_{\co_{\kk,\ell}}( T_\ell(A_0), T_\ell(A) )$ carries  a  hermitian form defined in a similar way.
 Here $T_\ell$ denotes $\ell$-adic Tate module.

The following proof  is left for the reader; compare with Proposition 2.12(ii) of \cite{KR2}.

\begin{proposition}\label{prop:genus type}
For every algebraically closed field $\F$ and every $(A_0,A) \in \mathcal{M}(\F)$,  there is a unique incoherent
self-dual hermitian $(\kk_\R,\widehat{\co}_\kk)$-module  $\LL(A_0,A)$ of signature $(n,0)$ satisfying
 \[
 \LL(A_0,A)_\ell \iso   \Hom_{\co_{\kk,\ell}} (  T_\ell ( A_0) , T_\ell( A) )
 \]
for every prime $\ell\not=\mathrm{char}(\F)$.  Furthermore,  $\LL(A_0,A)$
depends only on the connected component of $\mathcal{M}$ containing $(A_0,A)$, and not on $(A_0,A)$ itself.
\end{proposition}

From Proposition \ref{prop:genus type} we obtain a decomposition
\begin{equation}\label{genus decomp}
\mathcal{M} = \bigsqcup_\LL \mathcal{M}_\LL
\end{equation}
where  $\LL$ runs over all incoherent  self-dual hermitian $(\kk_\R,\widehat{\co}_\kk)$-modules
 of signature $(n,0)$, and $\mathcal{M}_\LL$ is the union of those
connected components of $\mathcal{M}$ for which $\LL(A_0,A) \iso \LL$ at every
geometric point $(A_0,A)$.

If $(A_0,A)\in \mathcal{M}(\C)$ then we may form the Betti  homology groups
\begin{equation} \label{homology}
\mathfrak{A}_0 = H_1(A_0(\C),\Z),   \qquad
\mathfrak{A} = H_1(A(\C), \Z). \
\end{equation}
 Each  is a self-dual hermitian
$\co_\kk$-module.  Indeed, the polarization on $A_0$
induces a perfect $\Z$-valued symplectic form $\pol_0$ on $\mathfrak{A}_0$, and there
is a unique hermitian form $h_{\mathfrak{A}_0}$ on $\mathfrak{A}_0$ satisfying
\[
\pol_0(x,y) = \mathrm{Tr}_{\kk/\Q}   h_{\mathfrak{A}_0}(\delta_\kk^{-1} x,y),
\]
where $\delta_\kk=\sqrt{d_\kk}$ is the square root lying in the upper half complex
plane\footnote{More precisely, there is a choice of $i=\sqrt{-1}$ such that
$\pol_0(ix,x)$ and $\pol(ix,x)$ are positive definite, and we choose
$\delta_\kk$ to lie in the same connected component of $\C\smallsetminus \R$
as $i$}.
Similarly $\mathfrak{A}$ is equipped with a perfect symplectic form $\pol$, and a
hermitian form $h_\mathfrak{A}$ satisfying
\begin{equation}\label{symplectic}
\pol(x,y) = \mathrm{Tr}_{\kk/\Q}   h_{\mathfrak{A}}(\delta_\kk^{-1} x,y).
\end{equation}
The hermitian $\co_\kk$-modules $\mathfrak{A}_0$ and $\mathfrak{A}$ have
signatures $(1,0)$ and $(n-1,1)$.  As in (\ref{betti hom}), the $\co_\kk$-module
\[
L(\mathfrak{A}_0,\mathfrak{A}) = \Hom_{\co_\kk}(\mathfrak{A}_0,\mathfrak{A})
\]
carries a self-dual hermitian form of signature $(n-1,1)$, and the pair
$(A_0,A)$ lies on $\mathcal{M}_\LL(\C)$ if and only if
\begin{equation}\label{betti genus}
\widehat{L}(\mathfrak{A}_0,\mathfrak{A}) \iso  \LL_f.
\end{equation}

We now define divisors on $\mathcal{M}$  following  Kudla-Rapoport \cite{KR2}.

\begin{definition}\label{def:KR divisors}
For each positive  $m\in \Q$ and each  $\mathfrak{r} \mid \mathfrak{d}_\kk$, define the
\emph{Kudla-Rapoport divisor} $\mathcal{Z}(m, \mathfrak{r} )$ as the algebraic stack over $\co_\kk$
whose functor of points assigns to every connected $\co_\kk$-scheme $S$ the groupoid
 of  triples $(A_0,A,\lambda)$ in which
 \begin{itemize}
 \item
 $(A_0,A) \in \mathcal{M}(S)$,
 \item
$\lambda \in \mathfrak{r}^{-1} L (A_0,A)$  satisfies $\langle \lambda , \lambda \rangle =m$.
\end{itemize}
We further require that the morphism $\delta_\kk \lambda : A_0 \to A$ induce the trivial map
\begin{equation}\label{extra vanishing}
\delta_\kk \lambda :\Lie(A_0) \to \Lie(A)/\mathcal{F},
\end{equation}
where $\delta_\kk$ is any $\co_\kk$-module generator of $\mathfrak{d}_\kk$.
 \end{definition}

\begin{remark}\label{rem:trivially empty}
Of course (\ref{denom bound}) implies that  $\mathcal{Z}(m, \mathfrak{r} )=\emptyset$ unless
$m\in \mathrm{N}(\mathfrak{r})^{-1} \Z$.
\end{remark}

\begin{remark}\label{rem:mostly vanishing}
The vanishing of (\ref{extra vanishing})  is automatic if $\mathrm{N}(\mathfrak{r}) \in \co_S^\times$.
Indeed, if $\mathrm{N}(\mathfrak{r}) \in \co_S^\times$ then any $\lambda \in \mathfrak{r}^{-1}L(A_0,A)$
induces an $\co_\kk$-linear map
$
\lambda : \Lie(A_0) \to \Lie(A)/\mathcal{F}.
$
The action of $\co_\kk$ on the image of this map is through \emph{both} the structure map
$\co_\kk\to \co_S$ \emph{and} through its conjugate, and so the image is annihilated by
all  $\alpha-\overline{\alpha}$ with $\alpha\in \co_\kk$.  These elements
generate the ideal $\mathfrak{d}_\kk=\delta_\kk\co_\kk$.
\end{remark}

The forgetful map $j:\mathcal{Z}(m, \mathfrak{r} ) \to \mathcal{M}$ is finite, unramified, and representable,
as in  \cite[Proposition 2.9]{KR2}.   By \cite[Lemma 1.19]{Vi},
any geometric point of $\mathcal{M}$ admits an \'etale neighborhood $U \to\mathcal{M}$
such  that  $\mathcal{Z}(m, \mathfrak{r} )_{/U} \to U$ restricts to a closed immersion of schemes
on each connected component  of $\mathcal{Z}(m, \mathfrak{r} )_{/U}$.  Moreover, each of these
components is locally defined by a single equation
(when $\mathfrak{r}=\co_\kk$   this is proved in \cite{Ho3}; the general case is similar),
and so defines a divisor on $U$.  Adding them up defines a divisor  $\mathcal{Z}(m, \mathfrak{r} )_{/U}$ on $U$,
which by \'etale descent defines a divisor on $\mathcal{M}$.
When no confusion is possible we use the same letter $\mathcal{Z}(m, \mathfrak{r} )$ to denote the stack,
the associated divisor, and the  associated line bundle. For any $\LL$ as in (\ref{genus decomp}), define
$\mathcal{Z}_{\LL}(m,\mathfrak{r} ) = \mathcal{Z}(m, \mathfrak{r} ) \times_\mathcal{M}\mathcal{M}_\LL$
so that
\[
\mathcal{Z}(m,\mathfrak{r}) = \bigsqcup_\LL \mathcal{Z}_\LL (m,\mathfrak{r}).
\]


\subsection{Complex uniformization}
\label{ss:complex uniformization}


Fix one $\LL$ as in (\ref{genus decomp}).
Here we recall the uniformization of the smooth complex orbifold
$\mathcal{M}_\LL(\C)$ and its Kudla-Rapoport divisors.
The complex uniformization is explained in   \cite{KR2} and \cite{Ho2},
and so we only sketch the main ideas.

Recalling that $\LL$ has signature $(n,0)$, let
$\LL(\infty)$ be the coherent hermitian $(\kk_\R,\widehat{\co}_\kk)$-module with archimedean component
of signature $(n-1,1)$, but with the same finite part as $\LL$.
 To each point $(A_0,A)\in \mathcal{M}_\LL(\C)$  there is an associated pair
$(\mathfrak{A}_0,\mathfrak{A})$ of self-dual hermitian $\co_\kk$-modules
as in (\ref{homology}), and a  self-dual hermitian $\co_\kk$-module $L(\mathfrak{A}_0,\mathfrak{A})$
of signature $(n-1,1)$.  In the notation of (\ref{genus def}),  the isomorphism (\ref{betti genus}) is equivalent to
\[
L(\mathfrak{A}_0,\mathfrak{A}) \in \mathrm{gen}(\LL(\infty)).
\]
The pair $(\mathfrak{A}_0,\mathfrak{A})$
depends on the connected component of $\mathcal{M}_\LL(\C)$ containing $(A_0,A)$,
but not on $(A_0,A)$ itself, and the formation of $(\mathfrak{A}_0,\mathfrak{A})$ from
$(A_0,A)$ establishes a bijection from the set of
connected components of $\mathcal{M}_\LL(\C)$ to the set of isomorphism classes of
pairs $(\mathfrak{A}_0,\mathfrak{A})$ in which
\begin{itemize}
\item
$\mathfrak{A}_0$ is a self-dual hermitian $\co_\kk$-module of signature $(1,0)$,
\item
$\mathfrak{A}$ is a self-dual hermitian $\co_\kk$-module of signature $(n-1,1)$,
\item
$L(\mathfrak{A}_0,\mathfrak{A}) \in \mathrm{gen}(\LL(\infty))$.
\end{itemize}

We now give an explicit  parametrization of the connected component of
$\mathcal{M}_\LL(\C)$ indexed by one pair $(\mathfrak{A}_0,\mathfrak{A})$.
Let $\mathcal{D}_{(\mathfrak{A}_0,\mathfrak{A})}$
be the space of negative $\kk_\R$-lines in  $L(\mathfrak{A}_0,\mathfrak{A})_\R$.
The group
\[
\Gamma_{(\mathfrak{A}_0,\mathfrak{A})} = \Aut(\mathfrak{A}_0) \times \Aut(\mathfrak{A})
\]
sits in a short exact sequence
\[
1 \to \mu(\kk) \to \Gamma_{(\mathfrak{A}_0,\mathfrak{A})} \to \Aut(L(\mathfrak{A}_0,\mathfrak{A})) \to 1
\]
in which the arrow $\mu(\kk) \to \Gamma_{(\mathfrak{A}_0,\mathfrak{A})}$
is the diagonal inclusion, and $\Gamma_{(\mathfrak{A}_0,\mathfrak{A})} \to \Aut(L(\mathfrak{A}_0,\mathfrak{A}))$
sends $(\gamma_0,\gamma)$ to the automorphism $\lambda \mapsto  \gamma \circ \lambda \circ \gamma_0^{-1}$.

There is a morphism of complex orbifolds
\[
\Gamma_{(\mathfrak{A}_0,\mathfrak{A})} \backslash \mathcal{D}_{(\mathfrak{A}_0,\mathfrak{A}) }
\to \mathcal{M}_\LL(\C)
\]
defined by sending the negative line $z\in \mathcal{D}_{(\mathfrak{A}_0,\mathfrak{A})}$
to the pair $(A_0,A_z)$, where $A_0(\C) = \mathfrak{A}_{0 \R} /\mathfrak{A}_0$ and
$A_z(\C) = \mathfrak{A}_\R / \mathfrak{A}$
as real Lie groups with $\co_\kk$-actions.  The complex structure on $A_0(\C)$
is defined by the natural action of $\kk_\R \iso \C$ on $\mathfrak{A}_{0 \R}$,
but the complex structure on $A_z(\C)$ depends on $z$.   A choice of nonzero vector
$a_0\in \mathfrak{A}_0$ determines an isomorphism
$L(\mathfrak{A}_0,\mathfrak{A})_\R \to \mathfrak{A}_\R$
by $\lambda\mapsto \lambda(a_0)$.  The image of $z$ under this isomorphism is a negative line
$z\subset  \mathfrak{A}_\R$, which does not depend on the choice of $a_0$.
Of course $\mathfrak{A}_\R$ inherits a complex structure from its $\co_\kk$-action and the
isomorphism $\kk_\R \iso \C$, but this does  \emph{not} define the complex
structure on $A_z(\C)$.   Instead, define an $\R$-linear
endomorphism $I_z$ of $\mathfrak{A}_\R$ by
\[
I_z (a)= \begin{cases}
i \cdot a& \mbox{if } a\in z^\perp \\
-i\cdot a & \mbox{if } a\in z
\end{cases}
\]
and use this new complex structure $I_z$ to make $A_z(\C)$ into a complex Lie group.  The symplectic
form $\pol$ on $\mathfrak{A}$ defined by (\ref{symplectic}) defines a polarization
on $A_z(\C)$, and the subspace
\[
z^\perp \subset  \mathfrak{A}_\R \iso \Lie(A_z)
\]
satisfies Kr\"amer's signature $(n-1,1)$ condition.
From the discussion above we find the complex uniformization
\begin{equation}\label{uniformization}
\mathcal{M}_\LL (\C) \iso \bigsqcup_{( \mathfrak{A}_0,\mathfrak{A}) }
\Gamma_{(\mathfrak{A}_0 , \mathfrak{A}) } \backslash
\mathcal{D}_{(\mathfrak{A}_0, \mathfrak{A} )}.
\end{equation}


\begin{remark}\label{component count}
Assume that either  $n>2$,  or that $\LL(\infty)$ contains, everywhere locally, a
nonzero  isotropic vector.   The strong approximation theorem implies that
\[
| \mathrm{gen}(\LL(\infty))|= 2^{1-o(d_\kk)}h_\kk.
\]
For each $L\in  \mathrm{gen}(\LL(\infty))$ there are exactly $h_\kk$ pairs
$(\mathfrak{A}_0,\mathfrak{A})$  satisfying $L(\mathfrak{A}_0, \mathfrak{A}) \iso L$,
and hence $\mathcal{M}_\LL(\C)$ has  $2^{1-o(d_\kk)}h_\kk^2$ components.
\end{remark}

Now we turn to the complex uniformization of the Kudla-Rapoport divisors.
 For any $m\in \Q_{>0}$ and any  $\mathfrak{r}\mid \mathfrak{d}_\kk$, the algebraic
 stack of Definition \ref{def:KR divisors} admits a complex uniformization
\begin{equation}\label{KR uniformization}
\mathcal{Z}_\LL(m,\mathfrak{r}) (\C) \iso  \bigsqcup_{( \mathfrak{A}_0,\mathfrak{A}) }
\Big(
 \Gamma_{ (\mathfrak{A}_0, \mathfrak{A}) } \backslash
 \bigsqcup_{  \substack{
 \lambda\in \mathfrak{r}^{-1} L(\mathfrak{A}_0,\mathfrak{A}) \\
 \langle \lambda,\lambda\rangle =m
 } }
 \mathcal{D}_{(\mathfrak{A}_0,\mathfrak{A}) }(\lambda)
  \Big),
\end{equation}
in which $\mathcal{D}_{(\mathfrak{A}_0,\mathfrak{A} )} (\lambda) \subset \mathcal{D}_{(\mathfrak{A}_0,\mathfrak{A} )}$
is the space of negative lines orthogonal to $\lambda$.  The essential point is that
$\mathcal{D}_{(\mathfrak{A}_0,\mathfrak{A} )} (\lambda)$ is precisely the locus of points
$z\in \mathcal{D}_{(\mathfrak{A}_0,\mathfrak{A} )}$ for which the $\R$-linear map
$\lambda: \mathfrak{A}_{0\R} \to \mathfrak{A}_\R$
is $\C$-linear relative to the complex structure $I_z$.


\subsection{Divisors attached to harmonic Maass forms}
\label{ss:harmonic divisors}


Fix an $\LL$   as in (\ref{genus decomp}).
The hermitian form $\langle\cdot,\cdot\rangle$ on  $\LL_f$ defines a $\widehat{\Z}$-valued
quadratic form  $Q(\lambda)=\langle \lambda, \lambda\rangle$.  The dual lattice is
$\mathfrak{d}_\kk^{-1} \LL_f$, and there is an  induced
$d_\kk^{-1} \Z/\Z$-valued quadratic form $Q$ on the  discriminant group
 $\mathfrak{d}_\kk^{-1}\LL_f/\LL_f$.
 Let $\Delta$ denote the automorphism group of $\mathfrak{d}_\kk^{-1}\LL_f/\LL_f$
 with its quadratic form.  The group $\Delta$ acts on the space $S_\LL$ of complex-valued
 functions on $\mathfrak{d}_\kk^{-1}\LL_f/\LL_f$, and commutes with the Weil
 representation
 \[
\omega_\LL:  \SL_2(\Z) \to \Aut(S_\LL).
 \]
 To every $\Delta$-invariant harmonic Maass form
 $f\in H_{2-n}(\omega_\LL)$
 we will construct a divisor $\mathcal{Z}_\LL(f)$ on $\mathcal{M}_\LL$
 as a linear combination of  Kudla-Rapoport divisors.

\begin{definition}
We will say that  $\mathfrak{d}_\kk^{-1}\LL_f / \LL_f$ is \emph{isotropic} if
$\mathfrak{d}_\kk^{-1}\LL_p/\LL_p$
represents $0$ non-trivially for every prime $p$ dividing $d_\kk$. This condition is equivalent to the
existence of an isotropic element of order $|d_\kk|$ in $\mathfrak{d}_\kk^{-1}\LL_f / \LL_f$.
\end{definition}

\begin{remark}
If $n>2$ then $\mathfrak{d}_\kk^{-1}\LL_f / \LL_f$ is always isotropic.  If $n=2$ then
$\mathfrak{d}_\kk^{-1}\LL_f / \LL_f$ is isotropic if and only if $\LL_f$ represents $0$
nontrivially everywhere locally; this is equivalent to all connected components of $\mathcal{M}_\LL$
being noncompact.
\end{remark}

For every  $m\in  \Q / \Z$ and every  $\mathfrak{r} \mid \mathfrak{d}_\kk$,
define  a $\Delta$-invariant function $\varphi_{m, \mathfrak{r}} \in S_\LL$
as the characteristic function of the subset
\[
\{
\lambda\in \mathfrak{r}^{-1} \LL_f/\LL_f  :  \; Q(\lambda) =m
\} \subset   \mathfrak{d}_\kk^{-1} \LL_f/\LL_f .
\]
Using \cite[Chapter IV.1.7]{Se73}, it is easy to check that $\varphi_{m,\mathfrak{r}}\not= 0$
if and only if  $m \in \mathrm{N}(\mathfrak{r})^{-1} \Z / \Z$.
By Witt's theorem the finitely many nonzero $\varphi_{m,\mathfrak{r}}$'s form a basis of $S_\LL^\Delta$.

\begin{lemma}\label{lem:f-mr}
For any $m\in \Q_{>0}$ and any $\mathfrak{r} \mid \mathfrak{d}_\kk$, there is an
 $f_{m, \mathfrak{r} }\in H_{2-n}(\omega_{\LL})^\Delta$
with holomorphic part of the form
\[
f^+_{m, \mathfrak{r} }(\tau) =  \varphi_{m, \mathfrak{r}} \cdot q^{-m}
+ \sum_{ k\in \Q_{\ge 0} } c_{m,\mathfrak{r}}^+(k)\cdot q^k,
\]
for some $c_{m,\mathfrak{r}}^+(k) \in S_\LL$.
Furthermore
\begin{enumerate}
\item
if $n>2$, then  $f_{m,\mathfrak{r}}$ is unique;
\item
if $n=2$ and $\mathfrak{d}_\kk^{-1} \LL_f/\LL_f$
is not isotropic, then $f_{m,\mathfrak{r}}$ is again unique;
\item
if $n=2$ and $\mathfrak{d}_\kk^{-1} \LL_f/\LL_f$ is
isotropic, then any two such $f_{m,\mathfrak{r}}$ differ by a constant,
and $f_{m,\mathfrak{r}}$ is uniquely determined  if
we impose the further condition that $c_{m,\mathfrak{r}}^+(0) \in S_\LL$
vanishes at the trivial coset of  $\mathfrak{d}_\kk^{-1}\LL_f/\LL_f$.
That is to say, $c^+_{m,\mathfrak{r}}(0,0)=0$.
\end{enumerate}
\end{lemma}

\begin{remark}
In order to make the notation $f_{m,\mathfrak{r}}$ unambiguous,
when $n=2$ and $\mathfrak{d}_\kk^{-1}\LL_f/\LL_f$ is isotropic we always choose
  $f_{m,\mathfrak{r}}$  so that $c_{m,\mathfrak{r}}^+(0,0)=0$.
\end{remark}

\begin{proof}
The existence statement follows from \cite[Proposition 3.11]{BF}.
To prove the uniqueness statement when $n>2$, we note that a harmonic Maass form
$f\in H_{k}(\omega_{\LL})$
with vanishing principal part is automatically holomorphic \cite[Proposition 3.5]{BF}.
Since the weight is negative, it vanishes identically.
Now suppose that $n=2$. Using the same argument as for $n>2$,  we see that any
two   $f_{m,\mathfrak{r}}$ differ by an element of $M_0(\omega_{\LL})^\Delta$, that is, by an element of  $S_\LL$ which is invariant under the action of the group $\Sl_2(\Z)\times \Delta$.

If $\mathfrak{d}_\kk^{-1}\LL_f/\LL_f$ is not isotropic then it is
easily seen that $M_0(\omega_{\LL})^\Delta=0$.
If $\mathfrak{d}_\kk^{-1}\LL_f/\LL_f$ is isotropic then it follows from \cite[Theorem 5.4]{Sch}  that
the space of invariants  $M_0(\omega_{\LL})^\Delta$ has dimension $1$, and that
 the map  $M_{0}(\omega_{\LL})^\Delta \to \C$ given by evaluation
 of the constant term  at the trivial coset
of  $\mathfrak{d}_\kk^{-1}\LL_f/\LL_f$ is an isomorphism.
\end{proof}

Fix $f\in H_{2-n}(\omega_\LL)^\Delta$.
An argument similar to the uniqueness part of Lemma \ref{lem:f-mr}
shows that $f$ may be decomposed as a $\C$-linear combination
\begin{equation}\label{f decomp}
f (\tau)= \mathrm{const}+  \sum_{ \substack{ m\in \Q_{>0} \\ \mathfrak{r} \mid \mathfrak{d}_\kk }}
 \alpha_{m,\mathfrak{r}}   \cdot  f_{m,\mathfrak{r}}(\tau)
\end{equation}
where  $\mathrm{``const"}$ is a  constant form in $M_{2-n}(\omega_\LL)^\Delta$.
This constant form is  necessarily $0$, except
when $n=2$ and $\mathfrak{d}_\kk^{-1} \LL_f/\LL_f$ is isotropic.
Define a divisor on  $\mathcal{M}_\LL$ with complex coefficients
\begin{equation}\label{form divisor}
\mathcal{Z}_\LL(f) =
\sum_{ \substack{ m\in \Q_{>0} \\ \mathfrak{r} \mid \mathfrak{d}_\kk }}
\alpha_{m,\mathfrak{r}} \cdot \mathcal{Z}_\LL(m,\mathfrak{r}).
\end{equation}
Obviously $\mathcal{Z}_\LL(f_{m,\mathfrak{r}}) = \mathcal{Z}_\LL(m,\mathfrak{r})$.

\begin{remark}
Although the  decomposition of  (\ref{f decomp}) is not unique, the divisor
(\ref{form divisor})  does not depend on the choice of decomposition.
This amounts to  verifying that  $\mathcal{Z}_\LL(m,\mathfrak{r})=0$ whenever  $f_{m,\mathfrak{r}}=0$,
which is clear:  if $f_{m,\mathfrak{r}}=0$ then $\varphi_{m,\mathfrak{r}}=0$, which implies  that
$m \not\in \mathrm{N}(\mathfrak{r})^{-1} \Z$.
Thus  $\mathcal{Z}_\LL(m,\mathfrak{r})=0$ by Remark \ref{rem:trivially empty}.
\end{remark}


\subsection{Compactification}
\label{ss:compact}


The moduli space $\mathcal{M}_{(n-1,1)}$ defined in Section \ref{ss:moduli} admits a canonical
toroidal compactification
\[
\mathcal{M}_{(n-1,1)} \hookrightarrow \mathcal{M}^*_{(n-1,1)}.
\]
Over $\co_\kk[1/d_\kk]$ the construction
is found in \cite{Lan}; the extension to $\co_\kk$ is  in \cite{Ho3}.  The $\co_\kk$-stack
\[
\mathcal{M}^*= \mathcal{M}_{(1,0)} \times \mathcal{M}^*_{(n-1,1)}
\]
is regular,  proper and flat over $\co_\kk$ of relative dimension $n-1$, and  smooth over $\co_\kk[1/d_\kk]$.
It contains $\mathcal{M}$ as a dense open substack, and the boundary $\mathcal{M}^*\smallsetminus\mathcal{M}$,
when endowed with its reduced substack structure, is proper and smooth over $\co_\kk$ of relative
dimension $n-2$.   Exactly as in (\ref{genus decomp}), there is a decomposition
$
\mathcal{M}^* = \bigsqcup_\LL \mathcal{M}^*_\LL
$
in which  $\mathcal{M}^*_\LL$ is, by definition,
 the Zariski closure of $\mathcal{M}_\LL$ in $\mathcal{M}^*$.

Fix a $\Delta$-invariant $f\in H_{2-n}(\omega_\LL)$ with holomorphic part
\[
f^+(\tau) = \sum_{ \substack{ m\in \Q \\ m \gg -\infty}} c^+(m) q^m,
\]
so that $c^+(m)\in S_\LL$.
We will  define a divisor $\mathcal{B}_\LL( f )$ on $\mathcal{M}^*_\LL$,
 supported on the boundary $\partial\mathcal{M}_\LL = \mathcal{M}_\LL^* \smallsetminus \mathcal{M}_\LL$.
 Start with a component $\mathcal{B}$ of the geometric fiber $\partial\mathcal{M}_{\LL / \kk^\alg}$.
 This component lies on some connected component of $\mathcal{M}^*_{\LL/\kk^\alg}$,
 which, as in Section \ref{ss:complex uniformization},  is indexed by a pair $(\mathfrak{A}_0,\mathfrak{A})$.
 As in \cite{Ho3},    the component $\mathcal{B}$  corresponds to the
 $\Gamma_{(\mathfrak{A}_0,\mathfrak{A})}$-orbit of an isotropic $\co_\kk$-direct summand
 $\mathfrak{a} \subset L(\mathfrak{A}_0,\mathfrak{A})$ of rank one, and
 by \cite[Proposition 2.6.3]{Ho3} there is a decomposition
 \[
 L(\mathfrak{A}_0,\mathfrak{A}) = E \oplus \mathfrak{a} \oplus \mathfrak{b}
 \]
 in which $\mathfrak{b}$ is an isotropic $\co_\kk$-submodule of rank one, and
 $\mathfrak{a}^\perp = \mathfrak{a}\oplus E$.  Under any such decomposition,
 $E$ is a self-dual hermitian $\co_\kk$-module of signature $(n-2,0)$.

 The \emph{multiplicity of $\mathcal{B}$ with respect to $f$} is defined as follows.
 Regard $c^+(m)$ as a function on
 \[
 \mathfrak{d}_\kk^{-1} \LL_f/\LL_f \iso
 \mathfrak{d}_\kk^{-1} E / E
 \oplus  \mathfrak{d}_\kk^{-1} \mathfrak{a} /  \mathfrak{a} \oplus  \mathfrak{d}_\kk^{-1} \mathfrak{b}/\mathfrak{b}.
 \]
 If $n>2$ then
\[
 \mathrm{mult}_\mathcal{B}(f) =
 \sum_{m\in \Q_{>0}}  \frac{m}{n-2}
 \sum_{  \substack{  \lambda \in \mathfrak{d}^{-1}_\kk E  \\  \langle \lambda,\lambda\rangle =m } }
\sum_{  \substack{ \mu  \in \mathfrak{d}^{-1}_\kk \fraka /\fraka}}
c^+( -m, \lambda +\mu) .
\]
When $f=f_{m,\mathfrak{r}}$ this simplifies to
 \[
  \mathrm{mult}_\mathcal{B}(f_{m,\mathfrak{r}}) =
 \frac{m\mathrm{N}(\mathfrak{r}) }{n-2}   \cdot  |  \{  \lambda \in \mathfrak{r}^{-1} E  :  \langle \lambda,\lambda\rangle =m  \} |  .
 \]
If $n=2$ then $E=0$, and we instead define
\[
 \mathrm{mult}_\mathcal{B}(f) =
 -2\sum_{m\in \Z_{\geq 0}}
\sum_{  \substack{ \mu  \in \mathfrak{d}^{-1}_\kk \fraka /\fraka}}
c^+( -m, \mu) \sigma_1(m).
\]
 In the next section (see Remark \ref{rem:rational} and Corollary
 \ref{cor:defbf}) we will show that the above multiplicities of the
 boundary components with respect to $f$ are given by regularized
 theta lifts of $f$ to positive definite hermitian spaces of signature
 $(n-2,0)$.

 Exactly as in  \cite[Section 3.7] {Ho3},  the isomorphism class of the hermitian module $E$
 is constant on the $\Gal(\kk^\alg/\kk)$-orbit of $\mathcal{B}$,
 and so summing over all geometric components  $\mathcal{B}$ yields a divisor
 \[
 \mathcal{B}_\LL(f) = \sum_\mathcal{B} \mathrm{mult}_\mathcal{B}(f)   \cdot \mathcal{B}
 \]
 on  $\mathcal{M}^*_{ \LL / \kk^\alg }$ which   descends to
 $\mathcal{M}^*_{ \LL / \kk }$.  Denote in the same way the divisor on $\mathcal{M}_\LL^*$
 obtained by taking the Zariski closure.

 \begin{definition}\label{def:total KR}
Let  $\mathcal{Z}_\LL^*(f)$ be the Zariski closure  in $\mathcal{M}_\LL^*$
of the Kudla-Rapoport divisor  $\mathcal{Z}_\LL(f)$, and define
the \emph{total Kudla-Rapoport divisor}  on $\mathcal{M}^*_\LL$ by
 \[
 \mathcal{Z}_\LL^\mathrm{total}(f) =
 \mathcal{Z}_\LL^*(f) + \mathcal{B}_\LL(f) .
 \]
\end{definition}

Let  $\widehat{\mathrm{CH}}^1_\R(\mathcal{M}_\LL^*)$ be the arithmetic Chow group with real
coefficients and $\log$-$\log$ growth along the boundary in the sense of Burgos-Kramer-K\"uhn
\cite{BKK,BBK} (see also  \cite{Ho3} for a rapid review of the essentials), and set
\[
\widehat{\mathrm{CH}}^1_\C(\mathcal{M}_\LL^*) =
\widehat{\mathrm{CH}}^1_\R(\mathcal{M}_\LL^*)\otimes_\R \C.
\]
In the  next section (see especially Section \ref{ss:KR green}) we will construct a Green function $\Phi_\LL(f )$, which will
allow us to define an arithmetic cycle class
\[
  \widehat{\mathcal{Z}}^{\mathrm{total}}_\LL (f)=
  \big( \mathcal{Z}^\mathrm{total}_\LL(f)   ,  \Phi_\LL(f ) \big)
  \in \widehat{\mathrm{CH}}^1_\C(\mathcal{M}_\LL^*).
\]


\section{Green functions for divisors}
\label{s:green functions}


Here we consider the analytic theory of Shimura varieties
associated to hermitian spaces of signature $(n-1,1)$ over imaginary
quadratic fields. We also study their special divisors and define automorphic
Green functions   for special divisors as regularized theta lifts of harmonic Maass forms.
In Section \ref{sect:greenbd} we study these Green
functions on toroidal compactifications, and show that they are log-log Green functions in the
sense of \cite{BKK} for linear combinations of special divisors and boundary divisors.
We prove that the
multiplicities of the boundary divisors are given by regularized theta
lifts of harmonic Maass forms to hermitian
spaces of signature $(n-2,0)$.  To this end we compute
Fourier-Jacobi expansions of Green functions and analyze the different
terms at the boundary.
We use these results to construct Green functions for Kudla-Rapoport divisors on
 the complex orbifold $\mathcal{M}^*_\LL(\C)$ studied in Section \ref{s:moduli spaces}.

Since it does not cause any
extra work, and for future reference, we make no restriction on $d_\kk$
in the present section and allow it to be even.  Moreover, we work with Shimura varieties of arbitrary level structure.

Let $V$ be a hermitian space over $\kk$
equipped with a hermitian form
$\langle\cdot,\cdot\rangle$. Throughout we let $V_\R=V\otimes_\Q\R$ and assume that the signature of $V$ is $(n-1,1)$. We write
$\langle\cdot,\cdot\rangle_\Q$ for the symmetric bilinear form
$\langle x,y\rangle_\Q=\tr_{\kk/\Q}\langle x,y\rangle$.
The associated quadratic form over $\Q$ is
$Q(x) = \frac{1}2 \langle x,x\rangle_\Q=\langle x,x\rangle$. Note that the Weil representations of $\SL_2 \subset \Uni(1, 1)$ associated to the quadratic form over $\Q$ and the hermitian form are the same.


\subsection{Hermitian spaces and unitary Shimura varieties}


We realize the hermitian symmetric space
associated to the unitary group $\Uni(V)$ as the Grassmannian
$\calD$ of negative $\kk_\R$-lines in $V_\R$.  It can be
viewed as an open subset of the projective space $\P(V_\R)$
of the complex vector space $V_\R$.  The domain $\calD$  is not a
tube domain unless $n=2$, in which case it is isomorphic to the complex
upper half plane $\H$. In general, $\calD$ has a realization as a
Siegel domain as follows.

Let $\ell\in V$  be a nonzero isotropic vector and let $\widetilde\ell\in V$ be isotropic such that
$\langle \ell,\widetilde\ell\rangle=1$. The orthogonal complement
\begin{align*}
W= \ell^\perp \cap \widetilde\ell{}^\perp
\end{align*}
is a positive definite hermitian space over $\kk$ of dimension $n-2$, and we have $V=W\oplus \kk\ell\oplus \kk\widetilde\ell$.
If $z\in \calD$, then
$\langle z,\ell\rangle\neq 0$. Hence $z$
has a unique basis vector of the form
\[
\frakz+\tau \sqrt{d_\kk}\ell + \widetilde\ell
\]
with $\frakz\in W_\R$ and $\tau\in \C$. We denote this vector by
the pair $(\frakz,\tau)$. The condition that the restriction of the hermitian form to
$z$ is negative definite
is
equivalent to requiring that
\[
\norm(\frakz,\tau)=-\langle(\frakz,\tau),(\frakz,\tau) \rangle =2\sqrt{|d_\kk|}\Im (\tau)-\langle \frakz,\frakz\rangle
\]
is positive. Consequently, $\calD$ is isomorphic to
\[
\calH_{\ell,\widetilde\ell}=\{  (\frakz,\tau)\in W_\R\times \H :  2\sqrt{|d_\kk|}\Im
(\tau) >\langle \frakz,\frakz\rangle
\}.
\]

For $z\in \calD$ and $\xx\in V_\R$ we let $\xx_{z^\perp}$ and $\xx_z$ be the orthogonal projections of $\xx$ to $z^\perp$ and $z$, respectively. Then the majorant
\[
\langle \xx,\mu\rangle_z = \langle \xx_{z^\perp},\mu_{z^\perp}\rangle - \langle \xx_{z},\mu_{z}\rangle
\]
associated to $z$ defines a positive definite hermitian form on $V_\R$.
If $0\neq z_0\in z$, we have
\[
\langle \xx,\xx\rangle_z = \langle \xx,\xx\rangle +2 \frac{|\langle \xx,z_0\rangle|^2}{|\langle z_0,z_0\rangle|}.
\]

The hermitian domain $\calD$ carries over it a \emph{tautological bundle}, whose fiber
at the point $z\in \calD$ is the negative line $z$.  The hermitian form on $V_\R$ induces a hermitian metric on the tautological
bundle, whose  first Chern form $\Omega$ is $\Uni(V)(\R)$-invariant and positive.
It corresponds to an invariant K\"ahler metric on $\calD$
and gives rise to an invariant volume form $d\mu(z)=\Omega^{n-1}$.
In the coordinates of $\calH_{\ell,\widetilde\ell}$ we have
\[
\Omega=-dd^c\log \norm(\frakz,\tau).
\]

Let $L\subset V$ be an $\calO_\kk$-lattice, that is, a finitely
generated $\calO_\kk$-submodule such that $V=L\otimes_\Z \Q$ and such that the restriction of $\langle\cdot,\cdot\rangle$ to $L$ takes values in $\frakd_\kk^{-1}$.  With the quadratic form $Q(x)
=\langle x,x\rangle$, we may also view $L$ as a lattice  over $\Z$. Throughout we assume that $L$ is even as a lattice over $\Z$, that is, $\langle x,x\rangle\in \Z$ for all $x\in L$. This condition is automatically fulfilled if the hermitian form on $L$ takes values in $\calO_\kk$.
Let
\begin{align*}
L'&= \{ x\in V:  \text{$\langle x,y\rangle_\Q \in \Z$ for all $y\in L$}\},\\
L'_{\calO_\kk}&= \{ x\in V:  \text{$\langle x,y\rangle \in \calO_\kk$ for all $y\in L$}\}
\end{align*}
be the $\Z$-dual and the $\calO_\kk$-dual of $L$, respectively, so that
$L'=\frakd_\kk^{-1}L'_{\calO_\kk}\supset L$.

Let  $\Gamma$ be a finite index subgroup of the unitary group $\Uni(L)$ of $L$.  The quotient
\[
X_\Gamma =\Gamma\bs \calD.
\]
is a complex orbifold of dimension $n-1$. It is compact if and only if $V$ is anisotropic.
In particular, if $n>2$, then $X_\Gamma$ is non-compact.
We define the volume of $X_\Gamma$ by $\vol(X_\Gamma)= \int_{X_\Gamma} \Omega^{n-1}$ and the degree of a divisor $Z$ on $X_\Gamma$ by
\[
\deg(Z)= \int_{Z} \Omega^{n-2}.
\]

\subsubsection{Special divisors}

For any vector $\xx\in V$ of positive norm we put
\[
\calD(\xx)= \{ z\in \calD:  \langle z,\xx\rangle=0\}.
\]
Let $S_L$ be the complex vector space of functions $L'/L\to \C$.
In the spirit of \cite{Ku:Duke}, for $\varphi\in S_L$ and $m\in \Q_{>0}$ we define the special divisor
\[
Z(m,\varphi)= \sum_{\substack{\xx\in L'\\ \langle \xx,\xx\rangle =m}} \varphi(\xx) \calD(\xx).
\]
We write $Z(m)$ for the element of
\[
\Hom_\C(S_L,\Div_\C(\calD)) \iso \Div_\C(\calD) \otimes_\C S_L^\vee
\]
given by $\varphi\mapsto Z(m,\varphi)$.
If $\varphi$ is invariant under $\Gamma$, then $Z(m,\varphi)$ is a
$\Gamma$-invariant divisor and descends to a divisor on the quotient
$X_\Gamma$, which we will also denote by $Z(m,\varphi)$.


\subsection{Regularized theta lifts}


In this subsection we define automorphic Green functions for special
divisors as regularized theta lifts of harmonic Maass forms. These
Green functions turn out to be harmonic if the degree of the
corresponding divisor vanishes.

Let  $\tau=u+iv$ be the variable  in the upper half plane $\H$.   The
$\calO_\kk$-lattice $L$ together with the symmetric bilinear form
$\langle\cdot ,\cdot\rangle_\Q$ is an even $\Z$-lattice of signature
$(2n-2,2)$. Let $\omega_L$ be the corresponding Weil representation on
$S_L$ as in Section~\ref{sect:prelims}. For $z\in \calD$ fixed, the Siegel theta function
$\Theta_L(\tau,z)$ is a non-holomorphic
modular form of weight $n-2$ for $\Sl_2(\Z)$ with representation
$\omega_L^\vee$.

Let $f\in H_{2-n}(\omega_L)$, and denote its Fourier coefficients by $c^\pm(m)\in S_L$ as in \eqref{eq:fourierf}.
Note that $c^\pm(m,\mu)=c^\pm(m,-\mu)$ for $\mu\in L'/L$,  by \cite[Section 3]{BF}.
The pairing $\{ f,\Theta_L(\tau,z)\}$ is a function on $\H$, which is invariant under $\Sl_2(\Z)$.
Following \cite{Bo1} and \cite{BF}, we consider the regularized theta lift
\begin{align} \label{eq:AutoGreen}
\Phi(z,f)= \int_{\Sl_2(\Z)\bs \H}^{\reg}
\{ f,\Theta_L(\tau,z)\} d\mu(\tau)
\end{align}
of $f$, where $d\mu(\tau)= \frac{du\,dv}{v^2}$ is the invariant measure.
The integral is regularized by taking the constant term in the Laurent expansion at $s=0$ of the meromorphic continuation of
\[
\Phi(z,f,s)= \lim_{T\to \infty}\int_{\calF_T}
\{ f,\Theta_L(\tau,z)\} v^{-s}d\mu(\tau).
\]
Here $\calF_T$ denotes the standard fundamental domain for $\Sl_2(\Z)$ truncated at height $T$.
If $\Re(s)>0$, the limit exists and defines a smooth function in $z$ on all of $\calD$, which is invariant under the action of $\Gamma$ if $f$ is invariant under $\Gamma$.  It has a meromorphic continuation
in $s$ to  $\C$; see \cite{Bo1} or \cite{BF}.
The function $\Phi(z,f)$ is defined on all (!) of $\calD$, but it is only smooth on the complement of the divisor
\[
Z(f)= \sum_{m>0} \{c^+(-m),Z(m)\}.
\]

To describe the behavior near this divisor, we extend the incomplete Gamma function $\Gamma(0,t)=\int_{t}^{\infty}e^{-v}\frac{dv}{v}$ to a function on $\R_{\geq 0}$ by defining it as the constant term in the Laurent expansion at $s=0$ of the meromorphic continuation of $\int_{1}^{\infty}e^{-tv}v^{-s}\frac{dv}{v}$.
Hence we have
\[
\widetilde\Gamma(0,t)=\begin{cases}
\Gamma(0,t),&\text{if $t>0$,}\\
0,&\text{if $t=0$.}
\end{cases}
\]
The following result is a slight strengthening of \cite[Theorem 6.2]{Bo1} in our setting.

\begin{theorem}
\label{thm:sing}
For any $z_0\in \calD$ there exists a neighborhood $U\subset\calD$  such that the function
\[
\Phi(z,f)-\sum_{\substack{\xx\in L'\cap z_0^\perp}}
c^+(-\langle \xx,\xx\rangle ,\xx) \widetilde\Gamma(0,4\pi |\langle \xx_z,\xx_z\rangle |)
\]
is smooth on $U$. Here $c^+(m, \xx)$ stands for the value $c^+(m)(\lambda)$ of $c^+(m)$ at $\lambda+L$.
\end{theorem}

\begin{proof}
We begin by noticing that $L'\cap z_0^\perp$ is a positive definite $\calO_\kk$-module of rank $\leq n-1$. Hence the sum on the right hand side is finite.

Arguing as in the proof of \cite[Theorem 6.2]{Bo1} (see also \cite[Theorem 2.12]{Br1}), we see that there exists a small neighborhood $U\subset\calD$ of $z_0$ on which the function
\[
\Phi(z,f)-\sum_{\substack{\xx\in L'\cap z_0^\perp}}
c^+(-\langle \xx,\xx\rangle ,\xx) \CT_{s=0}\bigg[\int_{v=1}^\infty e^{4\pi \langle \xx_z,\xx_z\rangle v} v^{-s-1} \,dv\bigg]
\]
is smooth. Here $\CT_{s=0}[\cdot]$ denotes the constant term in the Laurent expansion in $s$ at $0$.
Inserting the definition of $\widetilde \Gamma(0,t)$ we obtain the assertion.
\end{proof}

\begin{corollary}
\label{cor:sing}
For any $z_0\in \calD$ we have
\[
\Phi(z_0,f)=\lim_{\substack{z\to z_0\\z\notin Z(f)}}\bigg[
\Phi(z,f)+\sum_{\substack{\xx\in L'\cap z_0^\perp\\ \xx\neq 0}}
c^+(-\langle \xx,\xx\rangle ,\xx) (\log(4\pi |\langle \xx_z,\xx_z\rangle |)-\Gamma'(1))\bigg].
\]
\end{corollary}

\begin{proof}
Using the fact that
$\Gamma(0,t)=-\log(t)+\Gamma'(1)+o(t)$ as $t\to 0$, the corollary follows from Theorem \ref{thm:sing}.
\end{proof}

By a {\em Green function} for a divisor $D$ on a complex manifold $X$ we mean a smooth function $G$ on $X\smallsetminus D$ with the property that for every point $z_0\in X$ there is a neighborhood $U$ and a local equation $\phi=0$ for $D$ on $U$ such that $G+\log|\phi|^2$ extends to a smooth function on all of $U$.
Using this definition, we may rephrase Theorem \ref{thm:sing} and the corollary by saying that $\Phi(z,f)$ is a Green function for $Z(f)$.  In fact, the difference of $\log | \langle  \xx_z,\xx_z \rangle|$ and  $\log |\phi_\xx|^2$ for any local equation $\phi_\xx=0$ of $\mathcal{D}(\xx)$ extends to a smooth function.
In the next subsection we will study the growth of $\Phi(z,f)$ at the boundary of a toroidal compactifaction of $X_\Gamma$  and show that it can also be considered as a Green function for a suitably `compactified' divisor
there.

\begin{remark}
Corollary \ref{cor:sing}, together with Theorem \ref{thm:CM value},
will be used in Section \ref{sect:compladj}  to compute the height pairing of a hermitian line bundle corresponding to an arithmetic  Kudla-Rapoport divisor with a CM cycle.
\end{remark}


\begin{proposition}
Let $\Delta_\calD$ be the $\Uni(V)(\R)$-invariant Laplacian on $\calD$. There exists a non-zero real constant $c$ (which only depends on the normalization of $\Delta_\calD$ and which is independent of $f$), such that
\[
\Delta_\calD \Phi(z,f) = c\cdot \deg Z(f)
\]
on the complement of the divisor $Z(f)$.
\end{proposition}

\begin{proof}
This can be proved in the same way as \cite[Theorem 4.7]{Br1}.
\end{proof}

\begin{proposition}
\label{prop:lp}
Let $f\in H_{2-n}(\omega_L)$ be $\Gamma$-invariant.
\begin{enumerate}
\item
If $n>2$, then the Green function $\Phi(z,f)$ belongs to $L^{p}(X_\Gamma,\Omega^{n-1})$ for every $p<2$.
\item
If $n>3$, then $\Phi(z,f)$ belongs to $L^{2}(X_\Gamma,\Omega^{n-1})$.
\end{enumerate}
\end{proposition}

We will prove this Proposition at the end of Section \ref{sect:greenbd}.

\begin{theorem}
Assume that $n>2$ and that $f\in H_{2-n}(\omega_L)$ is $\Gamma$-invariant.
Let $G$ be a smooth real valued function on $X_\Gamma\smallsetminus Z(f)$ with the properties:
\begin{enumerate}
\item $G$ is a Green function for $Z(f)$,
\item $\Delta_{\calD} G = \text{constant}$,
\item $G\in L^{1+\eps}(X_\Gamma,\Omega^{n-1})$ for some $\eps>0$.
\end{enumerate}
Then $G(z)$ differs from $\Phi(z,f)$ by a constant.
\end{theorem}

\begin{proof}[Proof]
The difference $G(z)-\Phi(z,f)$ is a smooth subharmonic function on the complete Riemann manifold $X_\Gamma$
which is contained in $L^{1+\eps}(X_\Gamma,\Omega^{n-1})$.
By a result of Yau, such a function must be constant (see e.g. \cite[Corollary 4.22]{Br1}).
\end{proof}

For $n=2$ one can obtain a similar characterization by also requiring growth conditions at the cusps of $X_\Gamma$ (if there are any).


\subsection{The toroidal compactification}


\label{ss:torcomp}

The orbifold  $X_\Gamma=\Gamma\bs \calD$ can be compactified as follows.
Let $\Iso(V)$ be the set of
isotropic one-dimensional subspaces $I\subset V$. The group $\Gamma$
acts on $\Iso(V)$ with finitely many orbits. The rational boundary
point corresponding to $I\in \Iso(V)$  is the point $I_\R=I\otimes_\Q\R\in \P(V_\R)$. It lies in the closure of $\calD$ in $\P(V_\R)$.
The Baily-Borel compactification of $X_\Gamma$ is obtained by equipping the quotient
\[
\Gamma\bs \big(\calD\cup \{I_\R:  I\in \Iso(V)\}\big)
\]
with the Baily-Borel topology and complex structure.  The boundary points of this compactification are usually singular.
In contrast, here we work
with a canonical toroidal compactification of $X_\Gamma$, which we now
describe; see also \cite[Chapter 1.1.5]{Hof} and \cite[Section 3.3]{Ho3}.
It can be viewed as a resolution of the singularities at
the boundary points of the Baily-Borel compactification.

Let $I\in \Iso(V)$ be a one-dimensional isotropic subspace. Let $\ell\in I$ be a generator, and let
 $\widetilde\ell\in V$ be isotropic such that $\langle \ell,\widetilde\ell\rangle=1$. For $\eps>0$ we put
 \[
 U_\eps(\ell)= \left\{z\in \calD: -\frac{\langle z,z\rangle}{|\langle z,\ell\rangle|^2}>\frac{1}{\eps} \right\}.
 \]
In the coordinates of $\calH_{\ell,\widetilde\ell}$ we have
\[
 U_\eps(\ell)\cong \{(\frakz,\tau)\in  \calH_{\ell,\widetilde\ell}:  \norm(\frakz,\tau)>1/\eps\}.
\]
The stabilizer $\Uni(V)_\ell$ of $\ell$ acts on this subset.
Let $\Gamma_\ell=\Gamma\cap\Uni(V)_\ell$.
If $\eps $ is sufficiently small,
then
\begin{align}
\label{eq:oi}
\Gamma_\ell \bs U_\eps(\ell) \longrightarrow X_\Gamma
\end{align}
is an open immersion.
The center of $\Uni(V)_\ell$ is given by the subgroup of translations $T_a$ for $a\in \Q$, where
\[
T_a(\xx)=\xx +a\langle \xx,\ell\rangle \sqrt{d_\kk}\ell
\]
for $\xx\in V$.
It is isomorphic to the additive group over $\Q$.
The action of the translations on $\calH_{\ell,\widetilde\ell}$ is given by
$T_a(\frakz,\tau)= (\frakz,\tau+a).
$
The center of $\Gamma_\ell$
is of the form
\[
\Gamma_{\ell,T}= \{ T_a: a\in r\Z\}
\]
for a unique $r\in \Q_{>0}$, which is sometimes called the {\em width} of the cusp $I_\R$. If we put $q_r=e^{2\pi i \tau/r}$, then $(\frakz,\tau)\mapsto (\frakz,q_r)$ defines an isomorphism from $\Gamma_{\ell,T}\bs U_\eps(\ell)$ to
\[
V_\eps(\ell)=\left\{ (\frakz,q_r)\in \C^{n-2}\times \C:  0<|q_r|<\exp\left(-\tfrac{\pi}{r\sqrt{|d_\kk|}}(\langle\frakz,\frakz\rangle+1/\eps)\right)\right\}.
\]
Hence $\Gamma_{\ell,T}\bs U_\eps(\ell)$ can be viewed as a punctured disc bundle over $\C^{n-2}$. Adding the origin to every disc gives the disc bundle
\[
\widetilde V_\eps(\ell)= \left\{ (\frakz,q_r)\in \C^{n-2}\times \C :  |q_r|<\exp\left(-\tfrac{\pi}{r\sqrt{|d_\kk|}}(\langle\frakz,\frakz\rangle+1/\eps)\right)\right\}.
\]

The action of $\Gamma_\ell$ on $V_\eps(\ell)$ extends to an action on $\widetilde V_\eps(\ell)$, which leaves the boundary divisor $q_r=0$ invariant, and which is free if $\Gamma$ is sufficiently small.
We obtain an open immersion of orbifolds
\begin{align}
\label{eq:partcomp}
\Gamma_\ell\bs U_\eps(\ell) \longrightarrow \left(\Gamma_\ell/\Gamma_{\ell,T}\right)\bs \widetilde V_\eps(\ell).
\end{align}
It can be used to glue the right hand side to $X_\Gamma$ to obtain a partial compactification, which is smooth if $\Gamma$ is sufficiently small.
For a point $(\frakz_0,0)\in \widetilde V_{\eps}(\ell)$ and $\delta>0$, we put
\begin{align}
\label{boundary-nbhd}
B_{\delta}(\frakz_0,0)= \left\{ (\frakz,q_r)\in \widetilde V_\eps(\ell):  \langle \frakz-\frakz_0,  \frakz-\frakz_0\rangle <\delta,\;|q_r|<\delta\right\}.
\end{align}
The images of the $B_{\delta}(\frakz_0,0)$ for $\delta>0$ under the natural map to  $\left(\Gamma_\ell/\Gamma_{\ell,T}\right)\bs \widetilde V_\eps(\ell)$
define a basis of open neighborhoods of the boundary point given by $(\frakz_0,0)$.

We let $X_\Gamma^*$ be the compactification of $X_\Gamma$ obtained by gluing
the right hand side of \eqref{eq:partcomp} to $X_\Gamma$ for every $\Gamma$-class of $\Iso(V)$. We denote by $B_I$ the boundary divisor of $X_\Gamma^*$ corresponding to $I\in\Iso(V)$.

The behavior of the special divisor $Z(m,\varphi)$ near the boundary can be described as follows.
Let $I\in \Iso(V)$ and let $\ell\in I$ be a generator.
Let $0<\eps <\frac{1}{2m}$ be small enough so that \eqref{eq:oi} defines
an open immersion. Then Lemma \ref{maj-est} below implies that the pullback
of $Z(m,\varphi)$ to $U_\eps(\ell)$ is given by the local special
divisor
\[
Z_\ell(m,\varphi)= \sum_{\substack{\xx\in L'\cap \ell^\perp\\ \langle \xx,\xx\rangle =m}} \varphi(\xx) \calD(\xx).
\]

\begin{lemma}
\label{maj-est}
If $z_0$ is a generator of $z\in \calD$ and $\xx\in V\otimes_\Q\R$, we have
\[
\langle \xx,\xx\rangle_z \geq  \frac{|\langle \xx,\ell\rangle|^2 |\langle z_0, z_0\rangle|}{2  |\langle z_0,\ell\rangle|^2}.
\]
\end{lemma}

\begin{proof}
The right hand side is independent of the choice of the generator $z_0$, and so  we may assume $\langle z_0,\ell\rangle=1$.
Moreover, both sides of the inequality remain unchanged if we act on $\xx$ and $z_0$ by elements of the stabilizer of $\ell$ in $\Uni(V)(\R)$.
Using this observation, one may reduce to the case $z_0= \tau \sqrt{d_\kk}\ell + \widetilde\ell$.
The remaining computation we leave to the reader.
\end{proof}


\subsection{Regularized integrals}


Let $k\in \Z_{\geq 0}$, and let $(M,Q)$ be an even integral lattice as in Section \ref{sect:prelims}.
Following \cite{Bo1}, for $f\in H_{-k}(\omega_M)$ and $g\in M_{k}(\omega_M^\vee)$
we define a regularized Petersson pairing by
\begin{align}
\label{eq:thetadef}
(f,g)^{\reg}&=\int_{\Sl_2(\Z)\bs \H}^{\reg} \{ f(\tau), g(\tau)\} d\mu(\tau)\\
\nonumber
&=\lim_{T\to \infty}\int_{\calF_T} \{ f(\tau), g(\tau)\} d\mu(\tau).
\end{align}
In Section \ref{sect:greenbd} such integrals will occur as multiplicities of the boundary components,
where $g$ will be the theta function of a positive definite hermitian lattice given by a quotient of $L$.

In the special case when $k=0$ and $g$ is constant, this integral is evaluated in \cite[Theorem 9.1]{Bo1}.
Here we describe how the integral can be computed when $k>0$.
We denote the Fourier expansion of $g$ by
\[
g(\tau)= \sum_{m\geq 0 } b(m) q^m,
\]
with coefficients $b(m)\in S_M^\vee$.
We let $\vartheta=q\frac{d}{dq}$ be the Ramanujan theta operator on $q$-series.
Recall that the image under $\vartheta$ of a holomorphic modular form
$g$ of weight $k$ is in general not a modular form. However, the function
\[
\widetilde\vartheta(g)=\vartheta(g)-\frac{k}{12} g E_2
\]
is a holomorphic modular form of weight $k+2$.
Here \[ E_2(\tau)=-24\sum_{m\geq 0} \sigma_1(m)q^m\] denotes the non-modular Eisenstein series of weight $2$ for $\Sl_2(\Z)$. If $R_k=2i\frac{\partial}{\partial\tau}+\frac{k}{v} $ denotes the Maass raising operator and $E_2^*(\tau)=E_2(\tau)-\frac{3}{\pi v}$ the non-holomorphic (but modular) Eisenstein series of weight $2$, we also have
\begin{align}
\label{eq:serre}
\widetilde\vartheta(g)=-\frac{1}{4\pi }R_k(g)-\frac{k}{12} g E_2^*.
\end{align}
If $h(q)\in \C((q))$ is a (formal) Laurent series in $q$, we denote by $\CT[h]$ its constant term.

\begin{theorem}\label{thm:regint}
Let $f\in H_{-k}(\omega_M)$ and $g\in M_{k}(\omega_M^\vee)$ be as above.
\begin{enumerate}
\item
If $k>0$, then
\[
(f,g)^{\reg} = \frac{4\pi}{k}\CT[ \{f^+,\vartheta(g)\}]= \frac{4\pi}{k}\sum_{m>0} m\cdot \{c^+(-m),b(m)\}.
\]

\item
 If $k=0$  (so that $g$ is constant), then
\[
(f,g)^{\reg} = \frac{\pi}{3}\CT[ \{f^+,g E_2)\}]= -8\pi\sum_{m\geq 0} \sigma_1(m)\cdot \{c^+(-m),g\}.
\]
\end{enumerate}
\end{theorem}

\begin{proof}
(1) We use the identity
$\overline{\partial} (E^*_2 d\tau) = -\frac{3}{\pi}d\mu(\tau)$ to obtain
\begin{align}
\label{eq:thetadef1}
(f,g)^{\reg}=-\frac{\pi}{3}\int_{\Sl_2(\Z)\bs \H}^{\reg} \{ f(\tau), \overline{\partial}(g E^*_2 d\tau)\} .
\end{align}
In view of
\eqref{eq:serre}, we have
\[
 \overline{\partial}(g E^*_2 d\tau) = -\frac{3}{k\pi} \overline{\partial}( R_k(g)d\tau).
\]
Putting this into \eqref{eq:thetadef1}, we get
\begin{align*}
(f,g)^{\reg}&=\frac{1}{k}\int_{\Sl_2(\Z)\bs \H}^{\reg} \{ f(\tau),\overline{\partial}( R_k(g)d\tau)\} \\
&=\frac{1}{k}\lim_{T\to \infty} \int_{\calF_T}d \{ f(\tau), R_k(g)d\tau\}
-\frac{1}{k}\int_{\Sl_2(\Z)\bs \H}\{ (\overline{\partial} f), R_k(g)d\tau\}\\
&=-\frac{1}{k}\lim_{T\to \infty} \int_{0}^1\{ f(u+Ti), R_k(g)(u+Ti)\}du\\
&\phantom{=}{}+\frac{1}{k}
\int_{\Sl_2(\Z)\bs \H}\{ \overline{\xi_{-k}( f)}, R_k(g)\} v^{k+2}d\mu(\tau).
\end{align*}
The second summand on the right hand side is a Petersson scalar product which is easily seen to vanish. The first summand is equal to
$\frac{4\pi}{k}\CT[\{ f^+, \vartheta(g) \}]$.
This concludes the proof of the $k>0$ case.

(2) If $k=0$, and $f\in M^!_{0}(\omega_L)$, the assertion follows from \cite[Theorem 9.2]{Bo1}. If  $f\in H_{0}(\omega_L)$ it can be proved in the same way.
\end{proof}


\subsection{Automorphic Green functions at the boundary}
\label{sect:greenbd}


Let $I\in \Iso(V)$ be an isotropic $\kk$-line. Then $\fraka=I\cap L$ is a
projective $\calO_\kk$-module of rank $1$.
The $\calO_\kk$-module
\[
D= (L\cap \fraka^\perp) / \fraka
\]
is positive definite of rank $n-2$. Let $\ell\in \fraka$ be a primitive (that is, $\Q\ell\cap \fraka=\Z\ell$) isotropic vector. We write $\fraka=\fraka_0\ell$ with a fractional ideal $\fraka_0\subset \kk$, and we let $\widetilde \ell\in V$ be isotropic such that $\langle \widetilde \ell, \ell\rangle=1$.

The lattice $D$ can be realized as a sublattice of $L$ as follows.
The lattice $\fraka^*=L_{\calO_\kk}'\cap I^\perp$ is a projective
$\calO_\kk$-module of rank $n-1$. The quotient $L_{\calO_\kk}'/\fraka^*$
is an $\calO_\kk$-module of rank $1$, which is projective since it is
torsion free.  Hence there is a projective $\calO_\kk$-module
$\frakb\subset L_{\calO_\kk}'$ of rank $1$ such that
$L_{\calO_\kk}'=\fraka^*\oplus \frakb$.
We have $\langle \fraka,L_{\calO_\kk}'\rangle =\langle \fraka,\frakb\rangle =\calO_\kk$
and $\langle \frakb,L\rangle =\calO_\kk$.  We put
\[
E=L\cap\fraka^\perp\cap\frakb^\perp.
\]

\begin{lemma}
With $\mathfrak{a}$ and  $\mathfrak{b}$ defined as above,
\label{lem:lat1}
\begin{enumerate}
\item
 $L\cap \fraka^\perp = E\oplus \fraka$ and $D\cong E$;
\item
if $L$ is $\calO_\kk$-self-dual then $L=E\oplus \fraka\oplus \frakb$;
\item
if $L$ is $\calO_\kk$-self-dual and $d_\kk$ is odd then in (2)  we may chose $\frakb$ to be isotropic.
\end{enumerate}
\end{lemma}

Let $f\in H_{2-n}(\omega_L)$. By analogy with \cite[Theorem 5.3]{Bo1}, the harmonic Maass form $f$ induces an $S_D$-valued harmonic Maass form $f_D\in H_{2-n}(\omega_D)$. It is characterized by its values on $\nu\in D'/D$ as follows:
\begin{align}
\label{eq:deffd}
f_D(\tau)(\nu)=  \sum_{\substack{\mu\in L'/L\\\mu \mid L\cap \fraka^\perp=\nu}}f(\tau)(\mu).
\end{align}
Here $\mu\mid L\cap \fraka^{\perp}$ denotes the restriction of $\mu\in\Hom(L,\Z)$ to $L\cap \fraka^{\perp}$, and we consider $\nu\in D'$ as an element of $\Hom(L\cap\fraka^{\perp},\Z)$ via the quotient map $ L\cap\fraka^{\perp}\to D$.

Let $\eps>0$ such that \eqref{eq:oi} is an open immersion. For a
boundary point $(\frakz_0,0)\in \widetilde V_{\eps}(\ell)$ and $\delta>0$,
we consider the Green function $\Phi(z,f)$ in the open
neighborhood $B_{\delta}(\frakz_0,0)$ defined in
\eqref{boundary-nbhd}. The pullback of the special divisor $Z(f)$ to
$B_{\delta}(\frakz_0,0)$ is given by the linear combination of local
special divisors
\[
Z_\ell(f)=
\sum_{m>0} \left\{ c^+(-m), Z_\ell(m)\right\}.
\]
Note that $Z_\ell(m)$ is invariant under the subgroup of translations $\Gamma_{\ell,T}\subset \Gamma_\ell$. The  support of $Z_\ell(m)$ on $\widetilde V_{\eps}(\ell)$ is the union of the sets
$\{(\frakz,q_r):  \langle \frakz+\widetilde\ell,\xx\rangle=0\}$ for $\xx\in ( L'\cap \ell^\perp)/\Gamma_{\ell,T}$ with $\langle\xx,\xx\rangle=m$.

\begin{theorem}
\label{thm:bndgrowth}
Let $f\in H_{2-n}(\omega_L)$
and denote its Fourier coefficients by $c^\pm(m)$.
Let $(\frakz_0,0)\in \widetilde V_{\eps}(\ell)$ be a boundary point.
The set
\[
S_f=\{\xx\in L'\cap\ell^\perp :  \text{$\langle \xx,\xx\rangle>0$, $c^+(-\langle \xx,\xx\rangle,\xx)\neq 0$
and $\langle \frakz_0+\widetilde \ell,\xx\rangle =0$}\}
\]
is finite.
If $\delta>0$ is sufficiently small, then the function
\[
\Phi(z,f) +\frac{r\Phi^D(f_D)}{2\pi\norm(\fraka_0)}\log|q_r|+c^+(0,0)\log \left|\log|q_r|\right|
+2\sum_{\xx\in S_f}c^+(-\langle\xx,\xx\rangle,\xx) \log|\langle \frakz+\widetilde\ell,\xx\rangle|
\]
has a continuation to a continuous function on $B_{\delta}(\frakz_0,0)$. It is smooth on the complement
 of the boundary divisor $q_r=0$, and its images under the differentials $\partial$, $\overline{\partial}$, $\partial\overline{\partial}$
have log-log growth along the divisor $q_r=0$ in the sense of  \cite[Definition 1.2]{BBK}.
Here
\[
\Phi^D(f_D)=(f_D,\Theta_D)^{\reg}
\]
 is the regularized Petersson pairing of $f_D$ and the theta function $\Theta_D$ as defined in \eqref{eq:thetadef}.
\end{theorem}

We postpone the proof of the theorem to Section  \ref{ss:Fourier-Jacobi}.

\begin{remark}
\label{rem:rational}
Let $c^\pm_D(m)\in S_D$ be the coefficients of $f_D$, and write
$\Theta_D(\tau)= \sum_{m\geq 0 } R_D(m) q^m$,
where the representation numbers $R_D(m)\in S_D^\vee$ are given by
\[
R_D(m , \varphi)= \sum_{\substack{\xx\in D'\\ Q(\xx)=m}} \varphi(\xx)
\]
for $\varphi\in S_D$. If $n>2$, then according to Theorem \ref{thm:regint} we have
\[
\Phi^D(f_D) = \frac{4\pi}{n-2}\CT[ \{f_D^+,\vartheta(\Theta_D)\}]= \frac{4\pi}{n-2}\sum_{m>0} m\cdot \{c_D^+(-m),R_D(m)\}.
\]
If $n=2$, then $D$ is trivial, and we have
\[
\Phi^D(f_D) = \frac{\pi}{3}\CT[ f_D^+\cdot E_2]= -8\pi\sum_{m\geq 0}  c_D^+(-m)\sigma_1(m).
\]
\end{remark}

We now associate a boundary divisor to the harmonic Maass form $f\in H_{2-n}(\omega_L)$.
We define the multiplicity of the  boundary divisor  $B_I$ with respect to $f$ by
\[
\mathrm{mult}_{B_I}(f)= \frac{r\Phi^{D}(f_{D})}{4\pi \norm(\fraka_0)}.
\]
If the principal part of $f$ has rational coefficients, then according to Remark \ref{rem:rational}, this multiplicity is rational.
In the special case that $d_\kk$ is odd, $L$ is $\calO_\kk$-self-dual, and $\Gamma=\Uni(L)$, we have in  view of Lemma \ref{lem:lat1}
that $r=\norm(\fraka_0)$, and therefore
$\mathrm{mult}_{B_I}(f)= \frac{1}{4\pi}\Phi^{D}(f_{D})$.
We define the boundary divisor associated with $f$ by
\[
B(f)= \sum_{I\in \Iso(V)/\Gamma} \mathrm{mult}_{B_I}(f)\cdot B_I.
\]
Theorem \ref{thm:bndgrowth} implies the following corollary.

\begin{corollary}
\label{cor:defbf}
The function $\Phi(z,f)$ is a logarithmic Green function on $X_\Gamma^*$ for the divisor $Z(f)+B(f)$ with possible additional log-log growth along the boundary divisors $B_I$.
\end{corollary}

\subsection{The Fourier-Jacobi expansion}
\label{ss:Fourier-Jacobi}

Here we compute the Fourier-Jacobi expansion of the automorphic Green
function $\Phi(z,f)$ using \cite{Hof}, \cite{Bo1} and \cite{Br1}, and
we provide the proofs of Theorem \ref{thm:bndgrowth} and  Proposition \ref{prop:lp}.

The natural embedding of $\calD$ into the Grassmannian of negative
definite $2$-dimensional oriented real subspaces of $V_\R$ is
compatible with the actions of the unitary group
$\Uni(V,\langle\cdot,\cdot\rangle)$ and the orthogonal group
$\Orth(V,\langle\cdot,\cdot\rangle_\Q)$. We may calculate the theta
lift of $f\in H_{2-n}(\omega_L)$ to $X_\Gamma$ by lifting to the
orthogonal group $\Orth(V,\langle\cdot,\cdot\rangle_\Q)$ and then pulling
back to the unitary group.

We continue to use the setup of Section \ref{sect:greenbd}. In
addition we introduce the following notation.  We fix $\ell'\in L'$
such that $\langle\ell',\ell\rangle_\Q=1$.  We denote by $N$ the
positive integer which generates the ideal $\langle L,\ell\rangle_\Q
\subset \Z$.  We write $\ell^{\perp,\Q}$ for the orthogonal complement
of $\ell$ with respect to the bilinear form
$\langle\cdot,\cdot\rangle_\Q$, and put
$K=(L\cap\ell^{\perp,\Q})/\Z\ell$. Then $K$ is an even lattice over
$\Z$ of signature $(2n-3,1)$.

For $x\in V_\R$ we put $x^2=\langle x,x\rangle_\Q$
and $|x|=\sqrt{|x^2|}$.
Let $(\frakz,\tau)\in \calH_{\ell,\widetilde\ell}$ and let $z$ be the corresponding point in $\calD$.
We have
\[
\ell_{z}^2
= -\frac{2}{\norm(\frakz,\tau)},
\]
where the quantity $\norm(\frakz,\tau)= -\langle (\frakz,\tau),(\frakz,\tau)\rangle$ is positive.
We also view $z$ as a two-dimensional (oriented) real subspace of $V_\R$.
The vector $\ell_z$  spans a one-dimensional real subspace of $z$,
whose orthogonal complement in $z$ with respect to $\langle \cdot, \cdot\rangle_\Q$ we denote by $w$, so that $z=w\oplus \R\ell_z$.
The real line $w$ is generated by the vector
$w_0(z)=-i(\frakz,\tau)=-i(\frakz+\tau\sqrt{d_\kk}\ell+\widetilde\ell)$, which we use to define an orientation on $w$.
Hence we obtain a map
\begin{align}
\label{eq:mapdk}
\calD\longrightarrow \Gr^+(K)
\end{align}
to the Grassmannian $\Gr^+(K)$ of oriented negative lines in $K\otimes_\Z\R$.
If $\xx\in K\otimes_\Z\R$, we have $\langle -i(\frakz,\tau),\xx\rangle_\Q=2\Im\langle (\frakz,\tau),\xx\rangle$.
The orthogonal projection of $\lambda$ to $w$ is given by
\[
\frac{|\xx_w|}{|\ell_{z}|}= |\Im\langle (\frakz,\tau),\xx\rangle|.
\]
We also define the vector
\[
\mu=-\ell'+\frac{\ell_z}{2\ell_z^2}+\frac{\ell_{z^\perp}}{2\ell_{z^\perp}^2}
\]
in $L\cap\ell^{\perp,\Q}$. It is easily checked that
\[
\langle\mu,\xx\rangle_\Q = \Re \langle (\frakz,\tau),\xx\rangle.
\]
For $w\in \Gr^+(K)$ and $\xx\in K\otimes_\Z\R$, we write $\langle w,\xx\rangle_\Q>0$ if $\langle w_0,\xx\rangle_\Q>0$ for a vector $w_0\in w$ defining the orientation.

Let $f\in H_{2-n}(\omega_L)$.
Similarly as in \eqref{eq:deffd}, according to \cite[Theorem 5.3]{Bo1}, the harmonic Maass form $f$ induces an
$S_K$-valued harmonic Maass form $f_K\in H_{2-n}(\omega_K)$. It is characterized by its values on $\nu\in K'/K$ as follows:
\[
f_K(\tau)(\nu)=  \sum_{\substack{\mu\in L'/L\\\mu \mid L\cap \ell^{\perp,\Q}=\nu}}f(\tau)(\mu).
\]
Here $\mu\mid L\cap \ell^{\perp,\Q}$ denotes the restriction of $\mu\in\Hom(L,\Z)$ to $L\cap \ell^{\perp,\Q}$,
and we consider $\nu\in K'$ as an element of $\Hom(L\cap\ell^{\perp,\Q},\Z)$ via the quotient map $ L\cap\ell^{\perp,\Q}\to K$.

Finally, following \cite[(3.25)]{Br1}, we define a special function for $A,B\in \R$ by
\[
\calV_n(A,B) = \int\limits_0^\infty \Gamma(n-1, A^2 y) e^{-B^2 y -1/y}y^{-3/2}\,dy.
\]
According to \cite[p.~74]{Br1} we have
\[
\calV_n(A,B)=2(n-2)! \sum_{r=0}^{n-2} \frac{A^{2r}}{r!}(A^2+B^2)^{1/4-r/2} K_{r-1/2} (2\sqrt{A^2+B^2}).
\]
The following result is now an immediate consequence of \cite[Theorem 3.9]{Br1}.

\begin{theorem}
\label{fourierphi2}
Let $f\in H_{2-n}(\omega_L)$ and denote its Fourier coefficients by $c^\pm(m)\in S_L$.
Let  $z\in \calD\smallsetminus Z(f)$ with $|\ell_z^2|<\frac{1}{2m_0}$,
where $m_0=\max\{m\in \Q: \; \text{$c^+(-m)\neq 0$}\}$.
Then the Green function $\Phi(z,f)$ is equal to
\begin{align*}
& \frac{1}{\sqrt{2} |\ell_z|} \Phi^K(w,f_K) +C_{f}+c^+(0,0)\log|\ell_z^2|\\
& {}-2 \sum_{\substack{\xx\in K'\smallsetminus \{0\} }} \sum_{\substack{\nu\in L'/L \\ \nu\mid L\cap \ell^{\perp,\Q}=\xx}
} c^+(-\langle\xx,\xx\rangle,\nu)\\
 & \hskip .4in   {}\times\Log \left( 1-e\big(\langle\nu,\ell'\rangle_\Q+\langle\xx,\mu\rangle_\Q+i|\xx_w|/|\ell_z|\big)\right)\\
& {}+\frac{2}{\sqrt{\pi}}  \sum_{\substack{\xx\in K'\\ \langle\xx,\xx\rangle >0}}\, \sum_{\substack{\nu\in L'/L \\ \nu\mid L\cap \ell^{\perp,\Q}=\xx}} c^-(-\langle\xx,\xx\rangle,\nu)\\
&{}\qquad \times \sum_{j\geq 1}\frac{1}{j} e\big(j\langle\nu,\ell'\rangle_\Q+j\langle\xx,\mu\rangle_\Q\big)
\calV_{n} \left( \frac{\pi j|\xx|}{|\ell_z|},\, \frac{\pi j |\xx_w|}{|\ell_z|}\right),
\end{align*}
where
\begin{align*}
C_{f} = -c^+(0,0)\left( \log(2\pi ) +\Gamma'(1)\right) -2 \sum_{\substack{a\in\Z/N\Z\\ a\not\neq 0}} c^+(0, a \ell/N) \log|1-e(a/N)|.
\end{align*}
Here $\Phi^K(w,f_K)$ denotes the function on $\Gr^+(K)$ given by the regularized theta lift of $f_K$
for the orthogonal group of $K$ as in \cite[Chapter 3.1]{Br1}. We view it as a function on $\calD$ via
the map \eqref{eq:mapdk}.
Finally, $\Log(z)$ stands for the principle branch of the complex logarithm.
\hfill$\square$
\end{theorem}

\begin{remark}
If $f\in M^!_{2-n}(\omega_L)$ is weakly holomorphic and has integral principal part, then according to \cite[Theorem 4.2.1]{Hof}
there exists a meromorphic modular form $\Psi(z,f)$ of weight $c^+(0,0)/2$ for the group $\Gamma$ (with a multiplier system of finite order) such that $-2\log\|\Psi(z,f)\|^2=\Phi(z,f)$ and $\dv(\Psi(z,f))=\frac{1}{2}Z(f)$.
Here $\|\cdot\|$ denotes the suitably normalized Petersson metric.
The above Fourier expansion of $\Phi(z,f)$ leads to the Borcherds product expansion
\begin{align*}
\Psi(z,f)&=e\big(\langle(\frakz,\tau),\varrho_W\rangle\big)\\
&\phantom{=}{}\times
\prod_{\substack{\xx\in K'\\ \langle W,\xx\rangle_\Q >0}} \prod_{\substack{\nu\in L'/L \\ \nu\mid L\cap \ell^{\perp,\Q}=\xx}}
\big( 1-e(\langle\nu,\ell'\rangle_\Q+\langle(\frakz,\tau),\xx\rangle)
\big)^{c^+(-\langle\xx,\xx\rangle,\nu)},
\end{align*}
which converges for $\norm(\frakz,\tau)>4m_0$.
Here $W\subset \Gr^+(K)$ denotes a Weyl chamber corresponding to $f$
(that is, a connected component of the complement of the singular locus of $\Phi^K(w,f)$),
and $\varrho_W\subset K\otimes_\Z\Q$ denotes the corresponding Weyl vector. Moreover $\langle W,\xx\rangle_\Q >0$
means that $\langle w,\xx\rangle_\Q >0$ for $w\in W$, see \cite[Section 4.1.2]{Hof}.
\end{remark}

We now turn to the proof of Theorem \ref{thm:bndgrowth}.
We begin with two technical lemmas.
The first one gives an estimate for the majorant of
the lattice $K$.  For $0<C<1$ we define
\[
\calS_C= \{ (\frakz,\tau)\in \calH_{\ell,\widetilde\ell}:  C\cdot 2\sqrt{|d_\kk|}\Im(\tau)>\langle \frakz,\frakz\rangle \}.
\]
For $\xx\in K\otimes_\Z\R$ and $(\frakz,\tau)\in \calH_{\ell,\widetilde\ell}$ we define
\[
h((\frakz,\tau),\xx) = \norm(\frakz,\tau)\langle \xx,\xx\rangle + 2 (\Im\langle (\frakz,\tau),\xx\rangle)^2.
\]

\begin{lemma}
\label{hyp-maj}
Let $0<C<1$. There exists an $\eps>0$ such that for any $(\frakz,\tau)\in \calS_C$ and any
$\xx = \xx_D -a\sqrt{d_\kk} \ell - \frac{b}{\sqrt{d_\kk}}\widetilde\ell\in K\otimes_\Z\R$ (where $\xx_D\in D\otimes_\Z\R$ and $a,b\in \R$), we have
\[
h((\frakz,\tau),\xx) \geq  \eps \left(a^2 |d_\kk| + b^2\Im(\tau)^2+ \norm(\frakz,\tau) \langle \xx_D,\xx_D\rangle\right).
 \]
\end{lemma}

\begin{proof}
This result can be viewed as a lower bound for the majorant $\langle \xx_{w^\perp},\xx_{w^\perp}\rangle_\Q -\langle \xx_{w},\xx_{w}\rangle_\Q$
associated to the negative line $w=\R w_0(z)\in \Gr^+(K)$.
It directly follows from \cite[Lemma 4.13]{Br1}. Note that in the proof of this lemma, of the equalities defining $\calR_t$ we only need that $|q(Y_D)|<B y_1 y_2$ with $B=\frac{t^4}{t^4+1}$ and $t>0$.
 \end{proof}

\begin{corollary}
\label{cor:growth}
Let $0<C<1$. There exists an $\eps>0$ such that for any $(\frakz,\tau)\in \calS_C$ and any
$\xx = \xx_D - a\sqrt{d_\kk} \ell - \frac{b}{\sqrt{d_\kk}}\widetilde\ell\in K\otimes_\Z\R$ (where $\xx_D\in D\otimes_\Z\R$ and $a,b\in \R$) with $\langle\xx,\xx\rangle\leq 0$, we have
\[
(\Im\langle (\frakz,\tau),\xx\rangle)^2
\geq  \eps \left(a^2 |d_\kk| + b^2\Im(\tau)^2+ \norm(\frakz,\tau) \langle \xx_D,\xx_D\rangle\right).
 \]
\end{corollary}

The following lemma is a useful variant of the corollary.

\begin{lemma}
\label{lem:growth}
Let $A\geq 0$ and $0\leq B<1$. Assume that $\Im(\tau)> \frac{(|\frakz|+2A)^2}{4|d_\kk|(1-B)^2}$. Then we have
\[
\Im\langle (\frakz,\tau),\xx\rangle -A |\xx|\geq B \left(a |d_\kk| + b\Im(\tau)\right)
\]
for all $\xx = \xx_D - a\sqrt{d_\kk} \ell - \frac{b}{\sqrt{d_\kk}}\widetilde\ell\in K\otimes_\Z\R$ with $b\geq 0$ and $\langle \xx,\xx\rangle\leq 0$.
\end{lemma}

\begin{proof}
We have
\begin{align*}
\Im\langle (\frakz,\tau),\xx\rangle&= \Im\langle \frakz,\xx_D\rangle + a\sqrt{|d_\kk|} +b\Im(\tau)\\
&\geq  {}-\frac{1}{2}|\xx_D|\cdot |\frakz|+ a\sqrt{|d_\kk|} +b\Im(\tau).
\end{align*}
Since $0\geq \langle\xx,\xx\rangle=\langle\xx_D,\xx_D\rangle-2ab$, we also have $|\xx|^2\leq 4ab$ and $ |\xx_D|^2\leq 4ab$. Consequently,
\begin{align*}
\Im\langle (\frakz,\tau),\xx \rangle -A |\xx|&\geq a\sqrt{|d_\kk|} +b\Im(\tau)-\sqrt{ab}\cdot (|\frakz|+2A)\\
&\geq B\left(a |d_\kk| + b\Im(\tau)\right)\\
&\phantom{\geq}{}+(1-B)\left(a |d_\kk| + b\Im(\tau)\right)-\sqrt{ab}\cdot (|\frakz|+2A).
\end{align*}
The quantity in the latter line can be interpreted as a binary quadratic form in $\sqrt{a}$ and $\sqrt{b}$, which is positive definite if $\Im(\tau)> \frac{(|\frakz|+2A)^2}{4|d_\kk|(1-B)^2}$. This implies the assertion.
\end{proof}

\begin{proof}[Proof of Theorem \ref{thm:bndgrowth}.]
It is easily seen that $S_f$ is finite. To obtain the claimed analytic properties of $\Phi(z,f)$ on  $B_{\delta}(\frakz_0,0)$, we consider the different terms of the Fourier expansion given in Theorem \ref{fourierphi2}.

\emph{Step 1.} We begin with the term
\begin{align*}
&\sum_{\substack{\xx\in K'\\ \langle\xx,\xx\rangle >0}}\, \sum_{\substack{\nu\in L'/L \\ \nu\mid L\cap \ell^{\perp,\Q}=\xx}} c^-(-\langle\xx,\xx\rangle,\nu)\\
&{}\times \sum_{j\geq 1}\frac{1}{j} e\big(j\langle\nu,\ell'\rangle_\Q+j\langle\xx,\mu\rangle_\Q\big) \calV_{n} \left( \frac{\pi j|\xx|}{|\ell_z|},\, \frac{\pi j |\xx_w|}{|\ell_z|}\right).
\end{align*}
According to \cite[equality (3.26)]{Br1}, the function $\calV_n(A,B)$ is bounded by a constant multiple of $e^{-\sqrt{A^2+B^2}}$. Moreover, for $\xx\in K'$ with
$\langle\xx, \xx\rangle>0$ we have
\[
2\frac{\xx^2}{|\ell_z^2|}+2\frac{|\xx_w^2|}{|\ell_z^2|}>\frac{\xx^2}{|\ell_z^2|}+2\frac{|\xx_w^2|}{|\ell_z^2|}=h((\frakz,\tau),\xx).
\]
If we write $\xx = \xx_D -a\sqrt{d_\kk} \ell - \frac{b}{\sqrt{d_\kk}}\widetilde\ell$ (where $\xx_D\in D\otimes_\Z\Q$ and $a,b\in \Q$), then in view of Lemma \ref{hyp-maj}
there exists an $\eps'>0$ such that
\[
\calV_{n} \left( \frac{\pi j|\xx|}{|\ell_z|},\, \frac{\pi j |\xx_w|}{|\ell_z|}\right)
\ll \exp\left(-\eps' j\sqrt{a^2|d_\kk|+b^2\Im(\tau)^2+\norm(\frakz,\tau)\langle\xx_D,\xx_D\rangle}\right).
\]
Since the coefficients $c^-(m,\mu)$ have only polynomial growth as $m\to -\infty$, we find that the above sum over $\xx\in K'$ converges uniformly on $B_\delta(\frakz_0,0)$ to a function which is bounded by
$O(\exp(-\eps''\sqrt{-\log|q_r|}))$ as $q_r\to 0$ for some $\eps''>0$.
Hence this sum converges to a continuous function on $B_\delta(\frakz_0,0)$ which vanishes along the divisor $q_r=0$.
Analogous estimates hold for all iterated partial derivatives with respect to $(\frakz,\tau)$. Using the fact that $d\tau =\frac{r}{2\pi i}\frac{dq_r}{q_r}$, we obtain that the differentials $\partial$, $\overline{\partial}$, $\partial\overline{\partial}$ of this function have log-log growth along $q_r=0$.

\emph{Step 2.} For the term  $c^+(0,0)\log|\ell_z^2|$, we notice that
\begin{align*}
\log|\ell_z^2|&=
-\log(\norm(\frakz,\tau)/2)\\
&=-\log\left(\sqrt{|d_\kk|}\Im(\tau)-\langle\frakz,\frakz\rangle/2\right)\\
&=-\log\left(-\log|q_r|\right)
-\log\left(\frac{r\sqrt{|d_\kk|}}{2\pi}+\frac{\langle\frakz,\frakz\rangle}{2\log|q_r|}\right).
\end{align*}
The second summand on the right hand side extends to a continuous function on $B_\delta(\frakz_0,0)$ whose differentials  have log-log growth along the boundary divisor $q_r=0$.

\emph{Step 3.} Next, we consider the term
\begin{align*}
&\sum_{\substack{\xx\in K'\smallsetminus \{0\} }} \sum_{\substack{\nu\in L'/L \\ \nu\mid L\cap \ell^{\perp,\Q}=\xx}
} c^+(-\langle\xx,\xx\rangle,\nu) \\
&\phantom{=}{}\times \log \left( 1-e\big(\langle\nu,\ell'\rangle_\Q+\langle\xx,\mu\rangle_\Q+i|\xx_w|/|\ell_z|\big)\right)\\
\nonumber
&= \sum_{\substack{\xx\in K'\smallsetminus \{0\} }} \sum_{\substack{\nu\in L'/L \\ \nu\mid L\cap \ell^{\perp,\Q}=\xx}} c^+(-\langle\xx,\xx\rangle,\nu) \\
&\phantom{=}{}\times \log \left( 1-e\big(\langle\nu,\ell'\rangle_\Q+\Re\langle (\frakz,\tau),\xx\rangle+i|\Im\langle(\frakz,\tau),\xx\rangle|\big)\right).
\end{align*}
There exists a constant $C>0$ such that $c^+(m,\nu) =O(e^{C\sqrt{m}})$ for $m\to \infty$. Hence it follows
from Corollary \ref{cor:growth} and Lemma \ref{lem:growth}, that the sum over $\xx\in K'$ with $\langle\xx,\xx\rangle< 0$ converges uniformly on $B_\delta(\frakz_0,0)$ to a function which is bounded by
$O(\exp(-\eps''\sqrt{-\log|q_r|}))$ as $q_r\to 0$ for some $\eps''>0$.
Observe that $\langle\xx,\xx\rangle< 0$ implies that $\langle\xx,\ell\rangle\neq 0$.

Moreover, Lemma \ref{hyp-maj}
implies that, if $\delta$ is sufficiently small, the sum over $\xx\in K'$ with $\langle \xx,\xx \rangle\geq 0$ and $\langle \xx,\ell\rangle\neq 0$ converges uniformly on $B_\delta(\frakz_0,0)$ to a function which is bounded by
$O(\exp(-\eps''\sqrt{-\log|q_r|}))$ as $q_r\to 0$ for some $\eps''>0$.
Analogous estimates hold for all iterated partial derivatives with respect to $(\frakz,\tau)$.
Hence, up to a continuous function with log-log growth differentials, the above sum is equal to
\begin{align}
\label{term2b}
&\sum_{\substack{\xx\in K'\smallsetminus \{0\}\\ \langle \xx,\ell\rangle=0 }} \sum_{\substack{\nu\in L'/L \\ \nu\mid L\cap \ell^{\perp,\Q}=\xx}
} c^+(-\langle\xx,\xx\rangle,\nu) \\
\nonumber
&\phantom{=}{}\times \log \left( 1-e\big(\langle\nu,\ell'\rangle_\Q+\langle\xx,\mu\rangle_\Q+i|\xx_w|/|\ell_z|\big)\right)\\
\nonumber
&=\sum_{\substack{\xx\in L'\cap \ell^\perp/\Z\ell\\\xx\neq 0}}
 c^+(-\langle\xx,\xx\rangle,\xx) \Log \left( 1-e\big(\Re\langle \frakz+\widetilde\ell,\xx\rangle+i|\Im\langle \frakz+\widetilde \ell,\xx\rangle|\big)\right).
\end{align}
Notice that the this sum does not depend on $\tau$.
We let $T_f$ be the finite set
\[
T_f=\{\xx\in L'\cap\ell^\perp/\Z\ell :  \text{$\langle \xx,\xx\rangle>0$, $c^+(-\langle \xx,\xx\rangle,\xx)\neq 0$
and $\Im\langle \frakz_0+\widetilde \ell,\xx\rangle =0$}\}.
\]
It is an analogue for the integral lattice $K$
of the set $S_f$ defined in Theorem \ref{thm:bndgrowth}.
Let $\widetilde T_f$ be a fixed system of representatives for $T_f/\{\pm 1\}$.
If $\delta$ is sufficiently small, then on the the right hand side of \eqref{term2b},
the sum over those $\xx$ which do not belong to $T_f$ defines a smooth function on
$B_\delta(\frakz_0,0)$. Hence, up to a smooth function, \eqref{term2b} is equal to
\begin{align*}
&\sum_{\substack{\xx\in T_f}}
 c^+(-\langle\xx,\xx\rangle,\xx) \Log \left( 1-e\big(\Re\langle \frakz+\widetilde\ell,\xx\rangle+i|\Im\langle \frakz+\widetilde \ell,\xx\rangle|\big)\right)\\
&= \sum_{\substack{\xx\in \widetilde T_f}}
 c^+(-\langle\xx,\xx\rangle,\xx) \log \left|1-e\big(\langle \frakz+\widetilde\ell,\xx\rangle\big)\right|^2
 +4\pi \sum_{\substack{\xx\in \widetilde T_f\\ \Im\langle \frakz+\widetilde\ell,\xx\rangle<0}}
 c^+(-\langle\xx,\xx\rangle,\xx) \Im\langle \frakz+\widetilde\ell,\xx\rangle.
\end{align*}
We find that \eqref{term2b} is the sum of a smooth function on $B_\delta(\frakz_0,0)$ and
\begin{align}
\label{term2bb}
\sum_{\substack{\xx\in S_f}}
 c^+(-\langle\xx,\xx\rangle,\xx) \log| \langle \frakz+\widetilde \ell,\xx\rangle|
 +4\pi \sum_{\substack{\xx\in \widetilde T_f\\ \Im\langle \frakz+\widetilde\ell,\xx\rangle<0}}
 c^+(-\langle\xx,\xx\rangle,\xx) \Im\langle \frakz+\widetilde\ell,\xx\rangle.
\end{align}

\emph{Step 4.} It remains to consider the quantity $\frac{1}{\sqrt{2} |\ell_z|} \Phi^K(w,f_K)$. Let
$\ell_K\in (\kk\ell\cap L)/\Z\ell= \fraka/\Z\ell$ be a primitive vector. Then $\fraka= \Z\ell_K+\Z\ell$. If we write $\ell_K=a\ell$ with $a\in \kk$, we have $\fraka=\fraka_0\ell$ with $\fraka_0= \Z a +\Z\subset \kk$ and
\begin{align}
\label{norma0}
\frac{2\Im(a)}{\sqrt{|d_\kk|}}= \norm(\fraka_0).
\end{align}
The positive definite lattice $(K\cap \ell_K^{\perp,\Q} ) / \Z\ell_K$ is isomorphic to the $\calO_\kk$-lattice $D=L\cap\fraka^\perp/\fraka$.
We use the Fourier expansion given in \cite[Chapter 3.1]{Br1}  with respect to the primitive isotropic vector $\ell_K$, to describe the behavior on $B_\delta(\frakz_0,0)$. The vector
\[
w_1=-i\frac{(\frakz,\tau)}{\sqrt{2\norm(\frakz,\tau)}}
\]
is the unique positively oriented vector in the real line $w$ of length $-1$. For $\xx\in D\otimes_\Z\R$ we have
\begin{align*}
\langle w_1,\xx\rangle_\Q &= \frac{\sqrt{2}\Im\langle \frakz,\xx\rangle}{\sqrt{\norm(\frakz,\tau)}},\\
\langle w_1,\ell_K \rangle_\Q &= \frac{\sqrt{2}\Im(a)}{\sqrt{\norm(\frakz,\tau)}}.
\end{align*}

Let $\ell_K'\in K'$ such that $\langle \ell_K',\ell_K\rangle_\Q=1$. According to \cite[p.~68]{Br1}, we have in our present notation that
\begin{align}\nonumber
\Phi^K(w_1,f_K) & = \frac{1}{\sqrt{2} \langle w_1,\ell_K\rangle_\Q} \Phi^D(f_D)\\
\nonumber
&\phantom{=}{}
+4\sqrt{2}\pi \langle w_1,\ell_K\rangle_\Q \sum_{\xx\in  D'} \sum_{\substack{\nu\in L'/L \\ \nu\mid L\cap \ell^{\perp}=\xx}}
 c^+(-\langle \xx,\xx\rangle, \nu)  \B_2\left( \frac{\langle w_1,\xx\rangle_\Q}{\langle w_1,\ell_K\rangle_\Q} + \langle \nu,\ell_K'\rangle_\Q\right) \\
\nonumber
&\phantom{=}{}+4\sqrt{2} \left(\frac{\pi}{\langle w_1,\ell_K\rangle_\Q}\right)^{n-2}\sum_{\xx\in  D'\smallsetminus\{0\}} \sum_{\substack{\nu\in L'/L \\ \nu\mid L\cap \ell^{\perp}=\xx}}
 c^-(-\langle \xx,\xx\rangle, \nu)
 |\xx|^{n-1}\\
\nonumber
&\phantom{=} \times \sum_{j\geq 1}  j^{n-3} e\left(j \frac{\langle w_1,\xx\rangle_\Q}{\langle w_1,\ell_K\rangle_\Q} + j\langle \nu,\ell_K'\rangle_\Q \right) K_{n-1} \left(\frac{2\pi j |\xx|}{\langle w_1,\ell_K\rangle_\Q}\right).
\end{align}
Here $\B_2(x)$ denotes the $1$-periodic function on $\R$ which agrees on $0\leq x< 1$ with the second Bernoulli polynomial $B_2(x)=x^2-x+1/6$, and $K_\nu(x)$ denotes the $K$-Bessel function. Because of the exponential decay of the $K$-Bessel function, we find that
\begin{align}
\nonumber
&\frac{1}{\sqrt{2} |\ell_z|} \Phi^K(w,f_K)
= \frac{\sqrt{\norm(\frakz,\tau)}}{2}\Phi^K(w,f_K) \\
\nonumber
& = \frac{\norm(\frakz,\tau)}{4\Im(a)} \Phi^D(f_D)\\
\nonumber
&\phantom{=}{}+4\pi \Im(a)
\sum_{\substack{\xx\in L'\cap \fraka^{\perp}/\fraka}} c^+(-\langle \xx,\xx\rangle, \xx)  \B_2\left( \frac{\Im\langle \frakz+\widetilde\ell,\xx\rangle}{\Im(a)}\right)+s(\frakz,\tau),
\end{align}
where $s(\frakz,\tau)$ is a continuous function on $B_\delta(\frakz_0,0)$ with log-log growth differentials.
If $\delta$ is sufficiently small, then the second summand on the right hand side is the sum of a smooth function on $B_\delta(\frakz_0,0)$ and
\[
8\pi \sum_{\substack{\xx\in \widetilde T_f\\ \Im\langle \frakz+\widetilde\ell,\xx\rangle<0}}
 c^+(-\langle\xx,\xx\rangle,\xx) \Im\langle \frakz+\widetilde\ell,\xx\rangle.
\]
Note that this term is the negative of the contribution coming from the second quantity in \eqref{term2bb}.
We obtain that up to a continuous function on $B_\delta(\frakz_0,0)$ with log-log growth differentials, the term
$\frac{1}{\sqrt{2} |\ell_z|} \Phi^K(w,f_K)$ is equal to
\[
-\frac{r\Phi^D(f_D)}{2\pi\norm(\fraka_0)}\log|q_r|+8\pi \sum_{\substack{\xx\in \widetilde T_f\\ \Im\langle \frakz+\widetilde\ell,\xx\rangle<0}}
 c^+(-\langle\xx,\xx\rangle,\xx) \Im\langle \frakz+\widetilde\ell,\xx\rangle.
\]
Here we have also used \eqref{norma0}.

\emph{Step 5.} Adding together all the contributions, we find that if $\delta$ is sufficiently small,
then
\[
\Phi(z,f) +\frac{r\Phi^D(f_D)}{2\pi\norm(\fraka_0)}\log|q_r|+c^+(0,0)\log\left|\log|q_r|
\right|+2\sum_{\xx\in S_f} c^+(-\langle\xx,\xx\rangle,\xx)\log|\langle \frakz+\widetilde\ell,\xx\rangle|
\]
has a continuation to a continuous function on $B_{\delta}(\frakz_0,0)$. It is smooth on the complement of the boundary divisor $q_r=0$, and its images under the differentials $\partial$, $\overline{\partial}$, $\partial\overline{\partial}$
have log-log growth along the divisor $q_r=0$.
\end{proof}

\begin{proof}[Proof of Proposition \ref{prop:lp}]
We only prove that for $n>3$ the Green function $\Phi(z,f)$ belongs to $L^{2}(X_\Gamma^*,\Omega^{n-1})= L^{2}(X_\Gamma,\Omega^{n-1})$. The other
assertion can be proved analogously.
Since $X_\Gamma^*$ is compact, it suffices to show this locally for a small neighborhood of any point of $X_\Gamma^*$.
Since $\Phi(z,f)$ has only logarithmic singularities outside the boundary, and since $\Omega^{n-1}$ is smooth outside the boundary, this is clear outside the boundary points.

Therefore it suffices to show that for any primitive isotropic vector $\ell\in L$ and any boundary point $(\frakz_0,0)\in \widetilde V_{\eps}(\ell)$ the function $\Phi(z,f)$ is square integrable with respect to the measure $\Omega^{n-1}$ in a small neighborhood $B_\delta(\frakz_0,0)$.

It is easily seen that there exists a non-zero constant $c$ such that
\begin{align*}
\Omega^{n-1}&= c\cdot\norm(\frakz,\tau)^{-n}d\frakz\,d\overline{\frakz}\, d\tau \, d\overline{\tau}\\
&= -\frac{r^2c}{4\pi^2} \cdot \left(\frac{\sqrt{|d_\kk|} r}{\pi}\log |q_r|-\langle \frakz,\frakz\rangle\right)^{-n}d\frakz\,d\overline{\frakz}\, \frac{dq_r\, d\overline{ q}_r}{|q_r|^2}.
\end{align*}
Here we have put $d\frakz = d\frakz_1\cdots d\frakz_{n-2}$.
Hence, according to Theorem \ref{thm:bndgrowth}, it suffices to show that
$\log|q_r|$ is square integrable on $B_\delta(\frakz_0,0)$ with respect to the measure $\Omega^{n-1}$.
Since $n>3$, this is now easily seen.
\end{proof}


\subsection{Automorphic Green functions for Kudla-Rapoport divisors}
\label{ss:KR green}


Fix $\LL$ and $f\in H_{2-n}(\omega_\LL)^\Delta$ as in Section
\ref{ss:compact}.  We will construct a Green function for the total Kudla-Rapoport divisor
of Definition \ref{def:total KR}.

Using the uniformization (\ref{uniformization}), fix a connected component
$\Gamma_{(\mathfrak{A}_0,\mathfrak{A})} \backslash \mathcal{D}_{(\mathfrak{A}_0,\mathfrak{A})}$ of  $\mathcal{M}_\LL(\C)$.
In particular,
\[
\widehat{L} (\mathfrak{A}_0,\mathfrak{A}) \iso \LL_f
\]
as hermitian $\widehat{\co}_\kk$-modules. Exactly as with $\LL_f$, the $\Z$-module
$\mathfrak{d}_\kk^{-1} L(\mathfrak{A}_0 , \mathfrak{A}) /  L(\mathfrak{A}_0 , \mathfrak{A})$
is equipped with a $d_\kk^{-1} \Z/\Z$-valued quadratic form whose automorphism group
we again denote by $\Delta$, and there is an isomorphism of quadratic spaces
\begin{equation}\label{quadratic iso}
\mathfrak{d}_\kk^{-1} L(\mathfrak{A}_0 , \mathfrak{A}) /L(\mathfrak{A}_0 , \mathfrak{A})
\iso \mathfrak{d}_\kk^{-1} \LL_f / \LL_f .
\end{equation}
Such an  isomorphism identifies  $S_\LL$  with the space
$S_{L(\mathfrak{A}_0,\mathfrak{A} )}$ of complex valued functions on the left hand side of
(\ref{quadratic iso}). This identification depends on the choice of (\ref{quadratic iso}), but the restriction
\begin{equation}\label{globalized schwartz}
S^\Delta_{L(\mathfrak{A}_0,\mathfrak{A} )} \iso S_\LL^\Delta
\end{equation}
to $\Delta$-invariants is independent of the choice.  This allows us to  view the function $f$
as a $\Delta$-invariant  $S_{L(\mathfrak{A}_0,\mathfrak{A})}$-valued harmonic Maass form.
The construction (\ref{eq:AutoGreen}) defines a  function $\Phi_{L(\mathfrak{A}_0, \mathfrak A)}( f)$ on
$\Gamma_{(\mathfrak{A}_0,\mathfrak{A})}\backslash  \mathcal{D}_{(\mathfrak{A}_0,\mathfrak{A})}$
with  logarithmic singularities along the divisor $\mathcal{Z}_{\LL}(f)(\C)$.

By  repeating the above construction on every connected component of $\mathcal{M}_\LL(\C)$
we obtain a Green function $\Phi_\LL( f)$ for the divisor  $\mathcal{Z}_{\LL}(f)$ on
$\mathcal{M}_\LL$.   By Corollary \ref{cor:defbf},  the pair
\begin{equation}\label{total cycle}
  \widehat{\mathcal{Z}}^{\mathrm{total}}_\LL (f)=
  \big( \mathcal{Z}^\mathrm{total}_\LL(f)   ,  \Phi_\LL(f ) \big)
\end{equation}
defines a class in $\widehat{\mathrm{CH}}^1_\C(\mathcal{M}_\LL^*)$.


\section{Complex multiplication cycles}
\label{s:cm cycles}


In this section we study a  $1$-dimensional cycle $\mathcal{Y}\to \mathcal{M}$  of  complex multiplication points,
and begin the  calculation  of its intersection with the Kudla-Rapoport divisors.


\subsection{Definition of the CM cycle}


For an $\co_\kk$-scheme $S$, an $S$-valued point
\[
(A_1,B) \in  (\mathcal{M}_{(0,1)} \times_{\co_\kk} \mathcal{M}_{(n-1,0)} ) (S)
\]
determines an $S$-valued point $A_1\times B \in \mathcal{M}_{(n-1,1)}(S)$,
where $A_1\times B$ is implicitly endowed with the product polarization, the product action of $\co_\kk$,
and the $\co_\kk$-stable $\co_S$-submodule $\Lie(B)\subset \Lie(A_1\times B)$  satisfying
Kr\"amer's signature $(n-1,1)$ condition.  In other words, the construction
$(A_1, B) \mapsto A_1\times B$ defines a morphism
\[
\mathcal{M}_{(0,1)}  \times_{\co_\kk} \mathcal{M}_{(n-1,0) }  \to \mathcal{M}_{(n-1,1)}.
\]
The algebraic stack
\[
\mathcal{Y} = \mathcal{M}_{(1,0)} \times_{\co_\kk} \mathcal{M}_{(0,1)} \times_{\co_\kk} \mathcal{M}_{(n-1,0) }
\]
is smooth and proper of relative dimension $0$ over $\co_\kk$, and admits a
finite and unramified morphism $\mathcal{Y} \to \mathcal{M}$ defined by
$(A_0,A_1,B) \mapsto (A_0, A_1\times B)$.    The algebraic stack $\mathcal{Y}$
is a \emph{CM cycle}, in the sense that for any triple
$(A_0,A_1,B) \in\mathcal{Y}(S)$ the entries $A_0$ and $A_1$ are elliptic curves with
complex multiplication, while $B$ is isogenous to a product of elliptic curves with complex multiplication.

For any $S$-valued point
$(A_0,A_1,B)\in \mathcal{Y}(S)$ there is an orthogonal  decomposition
\begin{equation}\label{ortho sum}
L(A_0, A_1\times B) \iso L(A_0,A_1) \oplus L(A_0,B),
\end{equation}
where $L(A_0,A_1)  = \Hom_{\co_\kk}(A_0,A_1)$ and $L(A_0,B)  = \Hom_{\co_\kk}(A_0,B)$.

\begin{theorem}[Canonical lifting theorem]\label{thm:superrigid}
Let $\widetilde{S}$ be an $\co_\kk$-scheme, and let $S\hookrightarrow \widetilde{S}$
be a closed subscheme defined by a nilpotent ideal sheaf.
Suppose $k$ and $\ell$ are positive integers. Every pair
\[
(B_1,B_2) \in \left(   \mathcal{M}_{(k,0)}  \times_{\co_\kk}   \mathcal{M}_{(\ell,0)} \right) (S)
\]
admits a unique deformation to an $\widetilde{S}$-valued point
\[
( \widetilde{B}_1, \widetilde{B}_2) \in \left(   \mathcal{M}_{(k,0)}  \times_{\co_\kk}   \mathcal{M}_{(\ell,0)} \right) (\widetilde{S}),
\]
and the restriction map
$
\Hom_{\co_\kk}(  \widetilde{B}_1, \widetilde{B}_2 ) \to \Hom_{\co_\kk}(B_1,B_2)
$
is an isomorphism.
\end{theorem}

\begin{proof}
The analogous statement for $p$-divisible groups, proved using Grothendieck-Messing theory
and assuming that $p$ is locally nilpotent
on $S$, is \cite[Proposition 2.4.1]{Ho2}.  To prove the lemma, combine the
argument of [\emph{loc.~cit.}] with the proof of \cite[Proposition 2.1.2]{Ho3},
which is based instead on algebraic de Rham cohomology, and so is valid for
abelian schemes over an arbitrary base.
\end{proof}

Proposition \ref{prop:genus type} has the following analogue, whose proof we again leave to the reader.

\begin{proposition}\label{prop:genus type 2}
Let $S=\Spec(\F)$ be the spectrum of an algebraically closed field, and suppose
$(A_0,A_1,B) \in \mathcal{Y}(\F)$.
\begin{enumerate}
\item
There is a unique incoherent   self-dual hermitian $(\kk_\R,\widehat{\co}_\kk)$-module
$\LL_0(A_0,A_1)$ of signature $(1,0)$  satisfying
 \[
 \LL_0(A_0,A_1)_\ell \iso   \Hom_{\co_{\kk,\ell}} (  T_\ell ( A_0) , T_\ell( A_1) )
 \]
for every prime $\ell\not=\mathrm{char}(\F)$.
\item
The hermitian $\co_\kk$-module $L(A_0,B)$ is self-dual of signature $(n-1,0)$.
\end{enumerate}
Moreover, the  modules   $\LL_0(A_0,A_1)$ and $L(A_0,B)$
depend only the connected component of $\mathcal{Y}$ containing $(A_0,A_1,B)$, and not on
the point $(A_0,A_1,B)$ itself.
\end{proposition}

From Proposition \ref{prop:genus type 2} we have a decomposition
\begin{equation}\label{CM genus}
\mathcal{Y} = \bigsqcup_{ (\LL_0, \Lambda ) } \mathcal{Y}_{(\LL_0, \Lambda )},
\end{equation}
 where the disjoint union is over the isomorphism classes of
pairs $(\LL_0,\Lambda)$ consisting of
\begin{itemize}
\item
an incoherent self-dual hermitian $(\kk_\R, \widehat{\co}_\kk)$-module $\LL_0$
of signature $(1,0)$,
\item
 a  self-dual hermitian $\co_\kk$-module $\Lambda$ of signature $(n-1,0)$.
\end{itemize}
The stack $\mathcal{Y}_{(\LL_0, \Lambda )}$ is the union of those connected components
of $\mathcal{Y}$ along which  $\LL_0(A_0,A_1) \iso \LL_0$ and $ L(A_0,B) \iso \Lambda$.

\begin{remark}\label{rem:sum}
Each pair $(\LL_0,\Lambda)$ as above determines an incoherent self-dual
$(\kk_\R, \widehat{\co}_\kk)$-module $\LL_0\oplus \Lambda$
of signature $(n,0)$, whose archimedean and finite parts are, by definition,
\begin{align*}
( \LL_0\oplus \Lambda)_\infty
& =  \LL_{0 , \infty} \oplus ( \Lambda \otimes_{\Z} \R)  \\
( \LL_0 \oplus \Lambda )_f
&=    \LL_{0 , f} \oplus (\Lambda \otimes_{\Z} \widehat{\Z}  ) . \nonumber
\end{align*}
\end{remark}

For the rest of Section \ref{s:cm cycles}, fix one pair $(\LL_0, \Lambda)$
as in (\ref{CM genus}), and set $\LL = \LL_0 \oplus \Lambda$.
The  morphism $\mathcal{Y} \to \mathcal{M}$ restricts to a morphism
$\mathcal{Y}_{(\LL_0 , \Lambda ) } \to \mathcal{M}_{\LL}$.


\subsection{Decomposition of the  intersection}
\label{ss:zero cycles}


There is a  cartesian diagram (this is the definition of the upper left corner)
\[
\xymatrix{
{ \mathcal{Z}_\LL(m,\mathfrak{r}) \cap \mathcal{Y}_{(\LL_0,\Lambda)}  } \ar[r]  \ar[d]
&  {  \mathcal{Y}_{(\LL_0,\Lambda)}   } \ar[d] \\
{   \mathcal{Z}_\LL(m,\mathfrak{r})   }  \ar[r]  & { \mathcal{M}_\LL ,}
}
\]
and our goal is to decompose the intersection
$\mathcal{Z}_\LL(m,\mathfrak{r}) \cap \mathcal{Y}_{(\LL_0,\Lambda)}$
 into smaller,  more manageable substacks.

Given $m_1,m_2\in \Q_{\ge 0}$ and
$\mathfrak{r} \mid \mathfrak{d}_\kk$, denote by
$\mathcal{X}_{(\LL_0,\Lambda)} (m_1, m_2 , \mathfrak{r}  )$
the algebraic stack over $\co_\kk$ whose functor of points assigns to a connected
$\co_\kk$-scheme $S$ the groupoid of tuples $(A_0,A_1, B ,\lambda_1, \lambda_2)$ in which
\begin{itemize}
\item
$(A_0,A_1,B) \in \mathcal{Y}_{(\LL_0,\Lambda)} (S)$,
\item
$\lambda_1\in \mathfrak{r}^{-1} L(A_0,A_1)$ satisfies $\langle \lambda_1 , \lambda_1 \rangle =m_1$,
\item
$\lambda_2\in \mathfrak{r}^{-1} L(A_0,B)$ satisfies $\langle \lambda_2 , \lambda_2 \rangle =m_2$,
\end{itemize}
and the map $\delta_\kk \lambda_1 : A_0 \to  A_1$ induces the trivial map
\begin{equation}\label{small vanishing}
\delta_\kk \lambda_1 : \Lie(A_0) \to \Lie(A_1)
\end{equation}
for any generator $\delta_\kk \in \mathfrak{d}_\kk$.
As in Remark \ref{rem:mostly vanishing}, vanishing of (\ref{small vanishing}) is automatic  if
$\mathrm{N}(\mathfrak{r}) \in \co_S^\times$.

\begin{proposition}
For every $m\in \Q_{>0}$ and every  $\mathfrak{r} \mid \mathfrak{d}_\kk$,
there is an isomorphism of $\co_\kk$-stacks
\begin{equation}\label{scheme-theoretic decomp}
\mathcal{Z}_\LL(m, \mathfrak{r} ) \cap \mathcal{Y}_{(\LL_0,\Lambda)}
\iso
\bigsqcup_{   \substack{  m_1, m_2 \in \Q_{\ge 0} \\ m_1+m_2 =m  }   }
 \mathcal{X}_{(\LL_0,\Lambda)} (m_1, m_2 , \mathfrak{r} ).
\end{equation}
\end{proposition}

\begin{proof}
Suppose $S$ is a connected $\co_\kk$-scheme.  An $S$-valued point on the
left hand side of (\ref{scheme-theoretic decomp}) consists of a pair of triples
\[
(A_0, A,\lambda) \in \mathcal{Z}_\LL(m, \mathfrak{r} )(S)
\qquad
(A_0,A_1,B) \in \mathcal{Y}_{(\LL_0,\Lambda)} (S)
\]
together with an isomorphism $A\iso A_1\times B$  identifying $\Lie(B)$ with the subsheaf
$\mathcal{F}\subset\Lie(A)$.   Under the orthogonal  decomposition
(\ref{ortho sum}),   $\lambda \in \mathfrak{r}^{-1} L(A_0 , A)$ decomposes  as
\[
\lambda = \lambda_1+ \lambda_2\in  \mathfrak{r}^{-1}L(A_0 , A_1) \oplus  \mathfrak{r}^{-1}L(A_0 , B)
\]
 in such a way that
 $\langle \lambda, \lambda \rangle = \langle \lambda_1 ,\lambda_1\rangle +  \langle \lambda_2, \lambda_2\rangle$.
If we set $m_1= \langle \lambda_1,\lambda_1\rangle$ and $m_2=\langle \lambda_2,\lambda_2\rangle$ then the
quintuple $(A_0,A_1,B,\lambda_1,\lambda_2)$ defines an $S$-valued point of
$\mathcal{X}_{(\LL_0,\Lambda)} (m_1, m_2 , \mathfrak{r})$.
This defines the desired isomorphism.
\end{proof}

Now we completely determine the structure of the stacks appearing in the
right hand side of (\ref{scheme-theoretic decomp}).  We will see momentarily
that each has dimension $0$ or $1$, depending on whether $m_1>0$ or $m_1=0$.
For any $m\in \Q_{\ge 0}$ and any $\mathfrak{r} \mid \mathfrak{d}_\kk$, define the
\emph{representation number}
\begin{equation}\label{lambda rep}
R_\Lambda(m,\mathfrak{r}) = \big|
\{ \lambda\in \mathfrak{r}^{-1} \Lambda :  \langle \lambda,\lambda\rangle=m  \}
\big|.
\end{equation}
For $m\in \Q_{>0}$ define a finite set of odd cardinality
\begin{equation}\label{diff set}
\mathrm{Diff}_{\LL_0}(m)
=\{ \mbox{primes $p$ of $\Q :  m$ is not represented by }\LL_{0,p} \otimes _{\Z_p}\Q_p \}.
\end{equation}
  Note that  every  $p\in \mathrm{Diff}_{\LL_0}( m )$ is  nonsplit in $\kk$.

\begin{theorem}\label{thm:zero cycles}
Fix $m_1,m_2\in \Q_{\ge 0}$ with $m_1>0$, and   $\mathfrak{r} \mid \mathfrak{d}_\kk$.
Abbreviate \[ \mathcal{X} = \mathcal{X}_{(\LL_0,\Lambda)} ( m_1, m_2 , \mathfrak{r} ). \]
\begin{enumerate}
\item
If $| \mathrm{Diff}_{\LL_0}(m_1) |> 1$, then  $\mathcal{X} =\emptyset$.

\item
If $ \mathrm{Diff}_{ \LL_0} (m_1  ) = \{p\}$, then
$\mathcal{X}$ has dimension $0$ and is supported in characteristic $p$.  Furthermore,
the \'etale local ring of every geometric point of $\mathcal{X}$ has length
\[
\nu_p(m_1) =
\ord_p(pm_1) \cdot
\begin{cases}
1/2 & \mbox{if $p$ is inert in $\kk$,}\\
1& \mbox{ if $p$ is ramified in $\kk$,}
\end{cases}
\]
and the number of geometric points of $\mathcal{X}$ (counted with multiplicities) is
\begin{equation}\label{geometric counting}
\sum_{ z\in \mathcal{X}(\F_\mathfrak{p}^\alg) } \frac{1}{|\Aut(z) |}
=    \frac{h_\kk}{w_\kk}  \cdot  \frac{  R_\Lambda(m_2, \mathfrak{r} ) }{  |\Aut(\Lambda)| }
 \cdot    \rho \left(   \frac{ m_1 \mathrm{N}(\mathfrak{s}) }{ p^\epsilon } \right)
\end{equation}
where $\mathfrak{p}$ is the unique prime of $\kk$ above $p$,  $\F_\mathfrak{p}^\alg$
is an algebraic closure of its residue field, $\rho$ is defined by (\ref{rho}),
$\mathfrak{s} = \mathfrak{r}/ (\mathfrak{r} + \mathfrak{p})$ is the prime-to-$\mathfrak{p}$ part of $\mathfrak{r}$,
and
\begin{equation}\label{epsilon}
\epsilon = \begin{cases}
1 & \mbox{if $p$ is inert in $\kk$,} \\
0 &\mbox{if $p$ is ramified in $\kk$}.
\end{cases}
\end{equation}
\end{enumerate}
\end{theorem}

\begin{proof}
If $\mathcal{X}\not=\emptyset$ then there is some point $(A_0,A_1,B,\lambda_1,\lambda_2) \in \mathcal{X}(\F)$,
where $\F$ is either $\C$ or $\F_p^\alg$ for some prime $p$.   Let $\overline{A}_1$
be the elliptic curve $A_1$, but with the action of $\co_\kk$ replaced by its complex conjugate.
Thus $\lambda_1:A_0 \to \overline{A}_1$ is an  $\co_\kk$-\emph{conjugate}-linear degree $m_1$
quasi-isogeny between elliptic curves with complex multiplication, and $\co_\kk$ acts
on the Lie algebras of $A_0$ and $\overline{A}_1$ through the \emph{same}
homomorphism $\co_\kk \to \F$. The only way such a conjugate linear quasi-isogeny can exist
is if $\F$ has nonzero characteristic, $p$ is nonsplit in $\kk$, and $A_0$ and $A_1$
are supersingular elliptic curves.  In particular
\[
\Hom_{\Z_\ell}(T_\ell(A_0) , T_\ell(A_1) )  \iso \Hom(A_0,A_1)\otimes_\Z\Z_\ell
\]
for every prime $\ell\not=p$, and hence also
\[
\LL_{0,\ell} \iso \LL_0(A_0,A_1)_\ell  \iso \Hom_{\co_{\kk,\ell}} (T_\ell(A_0) , T_\ell(A_1) )
\iso L(A_0,A_1) \otimes_{\Z}\Z_\ell
\]
as hermitian  $\co_{\kk,\ell}$-modules.  As $\langle \lambda_1,\lambda_1\rangle=m_1$ by definition
of the moduli space $\mathcal{X}$, we have now shown that $\LL_{0,\ell}$
represents $m_1$ for all finite primes $\ell\not=p$.  Therefore $\mathrm{Diff}_{\LL_0}(m_1)$
contains at most one prime, $p$.  We have already remarked that this set has
odd cardinality, and therefore $\mathrm{Diff}_{\LL_0}(m_1)= \{p\}$.

Next we compute the lengths of the local rings.

\begin{lemma}\label{lem:local ring}
The \'etale local ring of $\mathcal{X}$ at every point
\[
(A_0,A_1,B,\lambda_1,\lambda_2) \in \mathcal{X}(\F_\mathfrak{p}^\alg)
\]
is Artinian of length $\nu_p(m_1)$.
\end{lemma}

\begin{proof}
We reduce the proof to calculations of Gross \cite{Gr}.
The  tuple $(A_0,A_1,B,\lambda_1,\lambda_2)$ corresponds to a morphism
$z:\Spec(\F_\mathfrak{p}^\alg) \to \mathcal{X}$,
and by composing with the structure morphism we obtain a geometric point
$\Spec(\F_\mathfrak{p}^\alg) \to \Spec(\co_\kk)$.
Let $W$ be the completion of the \'etale local ring of $\Spec(\co_\kk)$ at this point.
 Let $R$ be the completed
\'etale local ring of $\mathcal{X}$ at $(A_0,A_1,B,\lambda_1,\lambda_2)$.  This ring pro-represents the
deformation functor of the tuple $(A_0,A_1,B,\lambda_1,\lambda_2)$ to Artinian local $W$-algebras
with residue field $\F_\mathfrak{p}^\alg$.   Theorem \ref{thm:superrigid}
implies that $(A_0,B,\lambda_2)$ admits a unique lift to any such $W$-algebra.
Thus the data of $B$ and $\lambda_2$
can be  ignored in the deformation problem, and $R$ pro-represents
the deformation functor of $(A_0,A_1,\lambda_1)$.  Equivalently, $R$ pro-represents
the deformation functor of $(A_0, \overline{A}_1,\lambda_1)$.

By the Serre-Tate theorem we may replace $A_0$ and $\overline{A}_1$ in the above deformation problem
by their $p$-divisible groups, which are $\co_{\kk,p}$-linearly isomorphic.
Call the common $p$-divisible group $G$, so that
\[
\lambda_1 \in \mathfrak{r}^{-1}   \End(G)
\]
is  $\co_{\kk,p}$-conjugate-linear,   $\delta_\kk \lambda_1 :G \to G$ induces the
trivial map on Lie algebras, and
\[
\ord_p(\mathrm{Nrd}(\lambda_1)) =
\ord_p(m_1)
\]
where $\mathrm{Nrd}$ is the reduced norm on the quaternion order $\End(G)$.
The ring $R$ pro-represents the  functor of deformations $(\widetilde{G},\widetilde{\lambda}_1)$ of $(G,\lambda_1)$
 with $\widetilde{\lambda}_1 \in \mathfrak{r}^{-1} \End(\widetilde{G})$ and
\begin{equation}\label{p-div vanishing}
\delta_\kk \widetilde{\lambda}_1: \Lie(\widetilde{G})\to \Lie(\widetilde{G})
\end{equation}
equal to zero.

Suppose first that  $\mathrm{N}(\mathfrak{r}) \in \Z_p^\times$.  Then  $\lambda_1\in \End(G)$,
and Remark \ref{rem:mostly vanishing} implies that the vanishing of (\ref{p-div vanishing}) is
automatically satisfied for any deformation.
In this case, Gross's results immediately imply that $R$ is Artinian of length $\nu_p(m_1)$.

Now suppose  $\mathrm{N}(\mathfrak{r}) \notin \Z_p^\times$, so that
$\mathfrak{r} \co_{\kk,\mathfrak{p}} = \mathfrak{d}_\kk \co_{\kk,\mathfrak{p}}$.
  If we set $y=\delta_\kk \lambda_1\in \End(G)$, then the  ring $R$ pro-represents the functor of defomations
  $(\widetilde{G},\widetilde{y})$ of $(G,y)$ with $\widetilde{y}\in \End(\widetilde{G})$ and
  $\widetilde{y} : \Lie(\widetilde{G}) \to \Lie(\widetilde{G})$ equal to zero.
Let $R'$ be the ring pro-representing the same deformation problem, but without
the condition that $\widetilde{y}: \Lie(\widetilde{G})\to \Lie(\widetilde{G})$ vanish.
By Gross's results $R'\iso W/\mathfrak{p}^{k+1}$, where
\[
k =   \ord_p (\mathrm{Nrd}(y))   = \nu_p(m_1).
\]
Let $(G_{k+1},y_{k+1})$ be the universal deformation of $(G,y)$ to $W/\mathfrak{p}^{k+1}$,
and let $(G_k,y_k)$ be its reduction to $W/\mathfrak{p}^k$.  To show that $R\iso W/\mathfrak{p}^k$
it suffices to prove that
\begin{equation}\label{length lie I}
y_{k+1}:\Lie(G_{k+1}) \to \Lie(G_{k+1})
\end{equation}
is nonzero, but  that
\begin{equation}\label{length lie II}
y_k:\Lie(G_k) \to \Lie(G_k)
\end{equation}
vanishes.

For any $\ell$, let $G_\ell$ denote the canonical
lift\footnote{in the sense of \cite{Gr}, so $G_\ell$ is the unique defomation of $G$, with its action
 of $\co_{\kk,\mathfrak{p}}$, to $R_\ell$}
 of $G$ to $R_\ell = W/\mathfrak{p}^\ell$,
and let $D(G_\ell)$ be the Grothendieck-Messing crystal of $G_\ell$ evaluated at
$R_\ell$.  Thus $D(G_\ell)$ is a free $\co_\kk\otimes_\Z R_\ell$-module of rank one,  and sits in
an exact sequence free $R_\ell$-modules
\[
0 \to \mathrm{Fil} ( D(G_\ell) ) \to D(G_\ell) \to \Lie(G_\ell) \to 0.
\]
Fix any $\Pi\in \co_\kk$ such that $\co_\kk=\Z[\Pi]$.  The proof of \cite[Proposition 2.1.2]{Ho2} shows that
\[
\mathrm{Fil} ( D(G_\ell) ) = J D(G_\ell),
\]
where
\[
J = \Pi \otimes 1 - 1\otimes \Pi \in \co_\kk\otimes_\Z R_\ell
\]
generates (as an $R_\ell$-module), the kernel of the natural map $\co_\kk \otimes_\Z R_\ell \to R_\ell$.
Note that the image of
\[
\overline{J} =  \overline{\Pi} \otimes 1 - 1\otimes \Pi \in \co_\kk\otimes_\Z R_\ell
\]
in $R_\ell$  is $\overline{\Pi} -\Pi$, which  generates the maximal ideal $\mathfrak{p} R_\ell$.

Suppose we are given an $\co_\kk$-conjugate-linear endomorphism
$y_{\ell-1} \in \End(G_{\ell-1})$.  By Grothendieck-Messing theory, such an endomorphism
induces an endomorphism $\widetilde{y}_{\ell-1}$ of $D(G_\ell)$, and $y_{\ell-1}$ lifts to $\End(G_\ell)$ if and only if
the composition
\[
\mathrm{Fil} (   D(G_\ell)  ) \to D(G_\ell) \map{ \widetilde{y}_{\ell-1}} D(G_\ell) \to \Lie(G_\ell)
\]
is trivial. It is now easy to see that each of the following statements is equivalent to the next one:
\begin{enumerate}
\item
$y_{\ell-1}$ lifts to $\End(G_\ell)$,
\item
the image of  $\widetilde{y}_{\ell-1}( J D(G_\ell)) = \overline{J} y_\ell(D(G_\ell))$
in $\Lie(G_\ell)$ is trivial,
\item
the image of $\widetilde{y}_{\ell-1} (D(G_\ell))$ in $\Lie(G_\ell)$ lies in $\mathfrak{p}^{\ell-1}\Lie(G_\ell)$,
\item
the composition
\[
D(G_\ell) \map{ \widetilde{y}_{\ell -1 }} D(G_\ell) \to \Lie(G_\ell) \to \Lie(G_{\ell-1})
\]
vanishes,
\item
the composition
\[
D(G_{\ell-1}) \map{y_{\ell-1}} D(G_{\ell-1})  \to \Lie(G_{\ell-1})
\]
vanishes,
\item
$y_{\ell-1}: \Lie(G_{\ell-1}) \to \Lie(G_{\ell-1})$ is trivial.
\end{enumerate}
Thus $y_{\ell-1}$ lifts to $\End(G_\ell)$ if and only if it induces the zero endomorphism of  $\Lie(G_{\ell-1})$,
and the nonvanishing of (\ref{length lie I}) and vanishing of (\ref{length lie II}) follow immediately.
\end{proof}

To complete the proof of Theorem \ref{thm:zero cycles}, it only remains to prove (\ref{geometric counting}).
We do this through a sequence of lemmas.

\begin{lemma}
Abbreviating $\mathcal{Y}=\mathcal{Y}_{(\LL_0,\Lambda)}$, we have
\begin{equation}\label{twisty count}
\sum_{z\in \mathcal{X}(\F_\mathfrak{p}^\alg) } \frac{1}{ |\Aut(z)| }
=
\sum_{  L_0  }
\sum_{  \substack{ (A_0,A_1,B) \in \mathcal{Y}(\F_\mathfrak{p}^\alg)   \\
L(A_0,A_1) \iso L_0   }      }
\frac{ R_{L_0}(m_1, \mathfrak{s} )  R_\Lambda(m_2, \mathfrak{r} ) }{| \Aut(A_0,A_1,B)| } ,
\end{equation}
where the outer sum on the right is over all hermitian $\co_\kk$-modules $L_0$
of rank one, and the representation number $R_{L_0}(m_1, \mathfrak{s} )$ is defined in the same
way as (\ref{lambda rep}).
\end{lemma}

\begin{proof}
Directly from the definitions, we have
\begin{align*}
&\sum_{z\in \mathcal{X}(\F_\mathfrak{p}^\alg) } \frac{1}{ |\Aut(z)| }\\
&=
\sum_{  (A_0,A_1,B) \in \mathcal{Y}(\F_\mathfrak{p}^\alg)       }
\sum_{   \substack{  \lambda_1\in \mathfrak{r}^{-1} L(A_0,A_1)  \\ \langle \lambda_1,\lambda_1\rangle =m_1 \\  \Lie(\delta_\kk \lambda_1) =0 } }
\sum_{   \substack{  \lambda_2\in \mathfrak{r}^{-1} L(A_0,B)  \\ \langle \lambda_2 , \lambda_2 \rangle =m_2  }   }
\frac{1}{| \Aut(A_0,A_1,B)| }  ,
\end{align*}
where the condition $\Lie(\delta_\kk \lambda_1)=0$ refers to the vanishing of (\ref{small vanishing}).

We claim that
\[
\{ \lambda\in \mathfrak{r}^{-1} L(A_0,A_1) : \Lie(\delta_\kk \lambda )=0 \} = \mathfrak{s}^{-1} L(A_0,A_1)
\]
for all $A_0 \in \mathcal{M}_{(1,0)}(\F_\mathfrak{p}^\alg)$ and  $A_1\in \mathcal{M}_{(0,1)}(\F_\mathfrak{p}^\alg)$.
If  $\lambda\in  \mathfrak{s}^{-1} L(A_0,A_1)$ then, as $\mathfrak{s}$ is prime to $\mathfrak{p}$, $\lambda$ induces
a morphism of Lie algebras $\lambda:\Lie(A_0) \to \Lie(A_1)$.  By the argument of
Remark \ref{rem:mostly vanishing}, the image of this
map is annihilated by $\delta_\kk$, and so $\Lie(\delta_\kk \lambda)=0$.
Conversely, suppose we start with $\lambda\in \mathfrak{r}^{-1} L(A_0,A_1)$ satisfying $\Lie(\delta_\kk \lambda )=0$.
Let $G$ be the connected $p$-divisible group over $\F^\alg_\mathfrak{p}$ of height $2$ and dimension $1$, and set $\co_B=\End(G)$.
Thus $\co_B$ is the maximal order in a quaternion division algebra over $\Q_p$.
We may fix an embedding $\co_{\kk,p} \to \End(G)$ and isomorphisms $A_0[p^\infty] \iso G \iso A_1[p^\infty]$
in such  a way that the first is $\co_{\kk,p}$-linear, and the second is $\co_{\kk,p}$-conjugate-linear.
The hypothesis $\lambda\in \mathfrak{r}^{-1} L(A_0,A_1)$ implies that  $\delta_\kk \lambda\in  \co_B$,
but we cannot have $\delta_\kk \lambda\in \co_B^\times$ (for then $\delta_\kk \lambda$, and also $\Lie(\delta_\kk \lambda)$,
would be an isomorphism).  Therefore $\delta_\kk \lambda$ lies in the unique maximal ideal of $\co_B$, and hence
\[
\lambda\in \co_B \iso \Hom (A_0[p^\infty] , A_1[p^\infty]).
\]
 This implies that
\[
\lambda\in \mathfrak{r}^{-1} L(A_0,A_1) \cap \Hom(A_0[p^\infty] , A_1[p^\infty]) = \mathfrak{s}^{-1} L(A_0,A_1)
\]
as desired.

We have now shown that
\begin{align*}
&\sum_{z\in \mathcal{X}(\F_\mathfrak{p}^\alg) } \frac{1}{ |\Aut(z)| }
\\
&=
\sum_{  (A_0,A_1,B) \in \mathcal{Y}(\F_\mathfrak{p}^\alg)       }
\sum_{   \substack{  \lambda_1\in \mathfrak{s}^{-1} L(A_0,A_1)  \\ \langle \lambda_1,\lambda_1\rangle =m_1  } }
\sum_{   \substack{  \lambda_2\in \mathfrak{r}^{-1} L(A_0,B)  \\ \langle \lambda_2 , \lambda_2 \rangle =m_2  }   }
\frac{1}{| \Aut(A_0,A_1,B)| }  .
\end{align*}
On the right hand side, each $L(A_0,A_1)$ is a hermitian $\co_\kk$-module of rank one,
while  $L(A_0,B) \iso \Lambda$.  The lemma follows immediately.
\end{proof}

Let $\VV_0$ be the incoherent hermitian space over $\A_\kk$ determined by
$\LL_0$, and recall from Remark \ref{rem:nearby} that for
every prime $p$ nonsplit in $\kk$ there is a unique coherent hermitian space
$\VV_0(p)$ that is isomorphic to $\VV_0$ everywhere locally away from
$p$.   We now repeat this construction on the level of $(\kk_\R,\widehat{\co}_\kk)$-modules.
Define a new  hermitian $(\kk_\R, \widehat{\co}_\kk)$-module
$\LL_0(p)$ by setting  $\LL_0(p)_\ell = \LL_{0,\ell}$
for every place $\ell\not=p$.  For  the $p$-component $\LL_0(p)_p$,
take the same underlying $\co_{\kk,p}$-module as $\LL_{0,p}$,
but replace the hermitian form $\langle\cdot,\cdot\rangle_{\LL_{0,p}}$ on $\LL_{0,p}$
with the hermitian form
\[
\langle\cdot,\cdot\rangle_{\LL_0(p)_p} = c_p \langle\cdot,\cdot\rangle_{\LL_{0,p}},
\]
where
\[
c_p = \begin{cases}
\mbox{any uniformizing parameter of $\Z_p$,} & \mbox{ if  $p$ is inert in $\kk$, } \\
\mbox{any element of  $ \Z_p^\times$   that is  not a norm from  $\co_{\kk,p}^\times$,}
&  \mbox{ if  $p$ is ramified in $\kk$. }
\end{cases}
\]
The resulting coherent $(\kk_\R, \widehat{\co}_\kk)$-module $\LL_0(p)$
has $\VV_0(p)$ as its associated hermitian $\A_\kk$-module.
Note that if $p$ is inert in $\kk$ then $\LL_0(p)$ is not self-dual.

\begin{lemma}
Any triple $(A_0,A_1,B)$ appearing in the final sum of (\ref{twisty count}) satisfies
\[
L(A_0,A_1) \in \mathrm{gen}(\LL_0(p)).
\]
\end{lemma}

\begin{proof}
It is easy to see that
\[
\LL_0(p)_\ell   \iso  \LL_{0,\ell} \iso
\LL_0(A_0,A_1)_\ell \iso \Hom_{\co_{\kk,\ell}} ( T_\ell(A_0), T_\ell(A_1) )
\iso L(A_0,A_1) \otimes_{\Z}\Z_\ell
\]
for all primes $\ell\not=p$, and that
\[
\LL_0(p)_\infty \iso  L(A_0,A_1) \otimes_{\Z}\R,
\]
as both sides are positive definite.  In particular  the coherent $\A_\kk$-hermitian spaces
$\VV_0(p)$ and  $L(A_0,A_1) \otimes_{\Z}\A$ are isomorphic at all places
away from $p$, and by  comparing invariants  we see that
\begin{equation}\label{nearby quaternion space}
\VV_0(p)_p \iso L(A_0,A_1) \otimes_{\Z}\Q_p
\end{equation}
as $\kk_p$-hermitian spaces.  In order to strengthen (\ref{nearby quaternion space}) to an isomorphism
\begin{equation}\label{nearby quaternion module}
\LL_0(p)_p \iso L(A_0,A_1) \otimes_{\Z}\Z_p,
\end{equation}
fix $\co_{\kk,p}$-module generators $x$ and $y$ of the left hand side and right hand side,
respectively,  of (\ref{nearby quaternion module}),
and define $p$-adic integers $\alpha=\langle x,x\rangle$ and $\beta=\langle y,y\rangle$.

Let $G$ be the unique connected $p$-divisible group of
height $2$ and dimension $1$ over $\F_\mathfrak{p}^\alg$, and fix an action of $\co_{\kk,\mathfrak{p}}$ on $G$
in such a way that the induced action on $\Lie(G)$ is through the structure map
$\co_{\kk,\mathfrak{p}} \to \F_\mathfrak{p}^\alg$. There is an
$\co_{\kk,\mathfrak{p}}$-linear isomorphism $G\iso A_0[p^\infty]$,
and an $\co_{\kk,\mathfrak{p}}$-\emph{conjugate}-linear isomorphism
$G\iso A_1[p^\infty]$.  These choices identify $L(A_0,A_1) \otimes_{\Z}\Z_p$ with the
submodule of $\co_{\kk,\mathfrak{p}}$-conjugate-linear endomorphisms
\[
\End_{\overline{\co}_{\kk,\mathfrak{p}}}(G) \subset \End(G),
\]
and identify the quadratic form $\langle\cdot,\cdot\rangle$ on $L(A_0,A_1) \otimes_{\Z}\Z_p$
with a $\Z_p^\times$-multiple of the restriction to $\End_{\overline{\co}_{\kk,\mathfrak{p}}}(G)$
of the  reduced norm on the quaternionic order $\End(G)$.
A routine calculation with quaternion algebras, as in \cite[pp.~376--378]{KRY1},  now implies that
 \[
\ord_p(\beta) = \begin{cases}
1 & \mbox{if $p$ is inert in $\kk$,} \\
0 & \mbox{if $p$ is ramified in $\kk$.}
\end{cases}
\]
Comparing with the definition of $\LL_0(p)$ then shows that
$\ord_p(\alpha)=\ord_p(\beta)$. The isomorphism (\ref{nearby quaternion space}) implies that
$\chi_{\kk,p}(\alpha)=\chi_{\kk,p}(\beta)$, and this information is enough to
guarantee that $\alpha/\beta$ is a norm from $\co_{\kk,\mathfrak{p}}^\times$.
This proves (\ref{nearby quaternion module}), and completes the proof of the lemma.
\end{proof}

\begin{lemma}
For each  $L_0\in\mathrm{gen}(\LL_0(p) )$ there are $h_\kk$ isomorphism classes of triples
$(A_0,A_1,B)\in \mathcal{Y}(\F_\mathfrak{p}^\alg)$ such that
$L(A_0,A_1) \iso L_0$.  Any such triple  satisfies
\[
| \Aut(A_0,A_1,B) |  = w^2_\kk  \cdot  |\Aut(\Lambda)| .
\]
\end{lemma}

\begin{proof}
Let $R$ be any complete local Noetherian ring with residue  $\F_\mathfrak{p}^\alg$.
Using Theorem \ref{thm:superrigid} and Grothendieck's formal existence theorem
\cite[Section 8.4.4]{FGA}, the triple $(A_0,A_1,B)$ lifts uniquely to $R$,
as do all of its automorphisms.  Using this, we are easily reduced to the corresponding
counting problem in characteristic $0$, which is easily solved using the
linear algebraic description of  $\mathcal{Y}(\C)$ found in  Section \ref{ss:CM values} below.
\end{proof}

\begin{lemma}
Still assuming that $\mathrm{Diff}_{\LL_0}(m_1)=\{p\}$, we have
\[
\frac{1}{w_\kk}  \sum_{  L_0  \in \mathrm{gen}(\LL_0(p))}
 R_{L_0}(m_1, \mathfrak{s} ) = \rho\left( \frac{ m_1\mathrm{N}(\mathfrak{s} ) }{ p^\epsilon} \right).
\]
\end{lemma}

\begin{proof}
As $\VV_0(p)$ is coherent, we may fix  a hermitian space $V_0$ over $\kk$ such that
\[
V_0\otimes_\Q \A\iso \VV_0(p).
\]
Note that $\mathrm{Diff}_{\LL_0}(m_1) =\{p\}$ implies that  $V_0$ represents $m_1$.
Pick one vector $\lambda_0\in V_0$ such that $\langle \lambda_0,\lambda_0\rangle =m_1$, and an
$\co_\kk$-lattice $L_0\subset V_0$ such that $L_0\in\mathrm{gen}(\LL_0(p))$.

Let $\widehat{\kk}^1$ denote the group of norm one elements in $\widehat{\kk}^\times$,
and define $\widehat{\co}_\kk^1$ in the same way.  As $h$ varies over
$\widehat{\kk}^1/\widehat{\co}_\kk^1$, the lattices $h\cdot L_0\subset V_0$, with the
hermitian forms restricted from $V_0$, vary over $\mathrm{gen}(\LL_0(p))$.  Thus
\[
\frac{1}{w_\kk}  \sum_{  L_0  \in \mathrm{gen}(\LL_0(p))}
 R_{L_0}(m_1, \mathfrak{s} )  =
 \sum_{  h \in \widehat{\kk}^1/\widehat{\co}_\kk^1 }
 \bm{1}_{ h\mathfrak{s}^{-1}L_0 } (\lambda_0),
\]
where $\bm{1}$ denotes characteristic function.
If we fix any $\widehat{\co}_\kk$-linear isomorphism $\widehat{\co}_\kk \iso \widehat{L}_0$,
 the hermitian form on $L_0$ is identified with
$\langle x,y\rangle =     x\overline{y} p^\epsilon u$ for some $u\in\widehat{\co}_\kk^\times$,
and now
\[
\frac{1}{w_\kk}  \sum_{  L_0  \in \mathrm{gen}(\LL_0(p))}
 R_{L_0}(m_1, \mathfrak{s} )  =
 \sum_{  h \in \widehat{\kk}^1/\widehat{\co}_\kk^1 }
 \bm{1}_{ \widehat{\co}_\kk } ( h^{-1}  s \lambda_0)
\]
where $\lambda_0\in \widehat{\kk}^\times$ satisfies  $u \mathrm{N}( \lambda_0 ) = m_1/p^\epsilon$,
and $s \in \widehat{\kk}^\times$ satisfies $s \widehat{\co}_\kk = \widehat{\mathfrak{s}}$.
The equality
\[
 \sum_{  h \in \widehat{\kk}^1/\widehat{\co}_\kk^1 }
 \bm{1}_{ \widehat{\co}_\kk } ( h^{-1}  s \lambda_0)
=
\rho\left( \frac{ m_1\mathrm{N}(\mathfrak{s} ) }{ p^\epsilon} \right)
\]
is easily checked, as both sides admit a factorization over the prime numbers, and the
prime-by-prime comparison is elementary.
\end{proof}

Combining (\ref{twisty count}) and the four lemmas  shows that
\begin{align*}
\sum_{z \in \mathcal{X}(\F_\mathfrak{p}^\alg) } \frac{1}{ |\Aut(z)| }
 & =
\sum_{  L_0  \in \mathrm{gen}(\LL_0(p))}
\sum_{  \substack{ (A_0,A_1,B) \in \mathcal{Y}(\F_\mathfrak{p}^\alg)    \\
L(A_0,A_1) \iso L_0   }      }
 \frac{  R_{L_0}(m_1, \mathfrak{s} ) R_\Lambda(m_2, \mathfrak{r} )  }{| \Aut(A_0,A_1,B)| }  \\
& =
\frac{h_\kk}{w^2_\kk}
\sum_{  L_0  \in \mathrm{gen}(\LL_0(p))}
\frac{  R_{L_0}(m_1, \mathfrak{s} ) R_\Lambda(m_2, \mathfrak{r} )  }{    |\Aut(\Lambda)|  } \\
& =
\frac{h_\kk}{w_\kk}  \cdot
\frac{ R_\Lambda(m_2, \mathfrak{r} )  }{    |\Aut(\Lambda)|  }
\cdot  \rho\left( \frac{ m_1 \mathrm{N}(\mathfrak{s} ) }{ p^\epsilon} \right),
\end{align*}
and completes the proof of Theorem \ref{thm:zero cycles}.
\end{proof}

Theorem  \ref{thm:zero cycles} implies that
$\mathcal{X}_{(\LL_0,\Lambda)} ( m_1 , m_2 , \mathfrak{r} )$ has dimension $0$
whenever $m_1>0$.  Now we turn to the case of $m_1=0$.

\begin{proposition}\label{prop:bad parts}
Fix a positive $m\in \Q$  and   $\mathfrak{r} \mid \mathfrak{d}_\kk$.
\begin{enumerate}
\item
If $R_\Lambda(m, \mathfrak{r})=0$ then
$\mathcal{X}_{(\LL_0,\Lambda)} ( 0 , m , \mathfrak{r} ) = \emptyset.$
\item
If $R_\Lambda(m,\mathfrak{r})\not=0$ then  $\mathcal{X}_{(\LL_0,\Lambda)} ( 0 , m , \mathfrak{r} )$
is nonempty, and  is smooth of relative dimension $0$ over $\co_\kk$.
In particular, it is a regular stack of dimension $1$.
\end{enumerate}
\end{proposition}

\begin{proof}
The morphism  $\mathcal{Y}_{(\LL_0,\Lambda)} \to \Spec(\co_\kk)$ is smooth of relative dimension
$0$, and Theorem \ref{thm:superrigid} implies that  the map
$\mathcal{X}_{(\LL_0,\Lambda)} ( 0 , m , \mathfrak{r} ) \to \mathcal{Y}_{(\LL_0,\Lambda)}$
defined by
\[
(A_0,A_1,B,0,\lambda_2) \mapsto (A_0,A_1,B)
\]
is formally \'etale.  Hence the composition
$\mathcal{X}_{(\LL_0,\Lambda)} ( 0 , m , \mathfrak{r} )  \to \Spec(\co_\kk)$
is smooth of relative dimension $0$.

It only remains to show that $\mathcal{X}_{(\LL_0,\Lambda)} ( 0 , m , \mathfrak{r} )$
is nonempty if and only if $R_\Lambda(m,\mathfrak{r})\not=0$.
If $\mathcal{X}_{(\LL_0,\Lambda)} ( 0 , m , \mathfrak{r} )$
is nonempty then we may pick any geometric point
\[
(A_0,A_1,B,\lambda_1,\lambda_2) \in \mathcal{X}_{(\LL_0,\Lambda)} ( 0 , m , \mathfrak{r} )(\F).
\]
Using $L(A_0 , B ) \iso \Lambda$, the homomorphism $\lambda_2$  defines an element of
$\Lambda$ satisfying $\langle \lambda_2,\lambda_2\rangle=m$, and in particular
 $R_\Lambda(m,\mathfrak{r})\not=0$.
Conversely,  if $R_\Lambda(m,\mathfrak{r})\not=0$
then pick some $\lambda_2\in \Lambda$ satisfying $\langle \lambda_2,\lambda_2\rangle=m$.
It follows from the uniformization (\ref{cm uniformization}) below that
$\mathcal{Y}_{(\LL_0,\Lambda)} (\C) \not=\emptyset$, and for any choice of
 $(A_0,A_1,B) \in \mathcal{Y}_{(\LL_0,\Lambda)} (\C)$ the vector
 $\lambda_2$ defines an element of $\Lambda\iso L(A_0,B)$.  Setting $\lambda_1=0$, the tuple
$(A_0,A_1,B,\lambda_1,\lambda_2)$ defines a complex point of
$\mathcal{X}_{(\LL_0,\Lambda)} ( 0 , m , \mathfrak{r} )$.
 \end{proof}

\begin{remark}\label{rem:decomposition}
It follows from (\ref{scheme-theoretic decomp}) and Theorem \ref{thm:zero cycles}
that if $R_\Lambda(m, \mathfrak{r} )=0$, the intersection
$\mathcal{Z}(m, \mathfrak{r} ) \cap \mathcal{Y}_{(\LL_0,\Lambda)}$ is isomorphic
to the  zero dimensional stack
\begin{equation}\label{proper part}
\bigsqcup_{   \substack{  m_1 \in \Q_{>0} \\  m_2 \in \Q_{\ge 0} \\ m_1+m_2 =m  }   }
 \mathcal{X}_{(\LL_0,\Lambda)} (m_1, m_2 , \mathfrak{r} ).
 \end{equation}
On the other hand, if $R_\Lambda(m, \mathfrak{r} )\not=0$ then
$\mathcal{Z}(m , \mathfrak{r} ) \cap \mathcal{Y}_{(\LL_0,\Lambda)}$ is the disjoint union of
the zero dimensional stack (\ref{proper part})  with  the one dimensional stack
$\mathcal{X}_{(\LL_0,\Lambda)} (0,m, \mathfrak{r})$.
\end{remark}


\subsection{The CM value formula}
\label{ss:CM values}


 Let $\LL(\infty)$ be  obtained from $\LL$ by
changing the signature at the archimedean place from $(n,0)$ to $(n-1,1)$.
Similarly, let $\LL_0(\infty)$ be obtained from $\LL_0$ by switching the
signature at the archimedean place from $(1,0)$ to $(0,1)$.

 As in Section \ref{ss:harmonic divisors}, the finite $\Z$-module $\mathfrak{d}_\kk^{-1} \LL_f/\LL_f$ is
equipped with a $d_\kk^{-1} \Z/\Z$-valued quadratic form, and we denote by $\Delta$  its automorphism
group as a finite quadratic space. The space $S_\LL$ of   complex valued functions on
$\mathfrak{d}_\kk^{-1}\LL_f/\LL_f$
is equipped with an action of $\Delta$ and a commuting action $\omega_\LL$
of $\Sl_2(\Z)$ defined by the Weil representation.
In exactly the  same way, the finite $\Z$-modules  $\mathfrak{d}_\kk^{-1}\LL_{0,f}/\LL_{0,f}$
and $\mathfrak{d}_\kk^{-1}\Lambda/\Lambda$ are equipped with
quadratic forms (still denoted $Q$),  and the spaces $S_{\LL_0}$ and $S_\Lambda$ are equipped
with actions $\omega_{\LL_0}$ and $\omega_\Lambda$ of $\Sl_2(\Z)$.
Moreover, the obvious isomorphism
\[
S_\LL \iso S_{\LL_0} \otimes_\C S_\Lambda
\]
is $\SL_2(\Z)$-equivariant.
Fix a $\Delta$-invariant harmonic form $f\in H_{2-n}(\omega_\LL)$.
In this subsection we compute the  value of the Green function $\Phi_\LL(f)$ at the points of
$\mathcal{Y}_{(\LL_0,\Lambda)}(\C)$.

First we must describe the complex uniformization of the
CM cycle $\mathcal{Y}_{(\LL_0,\Lambda)}$.  Fix a triple $(\mathfrak{A}_0,\mathfrak{A}_1,\mathfrak{B})$ in which
\begin{itemize}
\item
$\mathfrak{A}_0$ and $\mathfrak{A}_1$ are self-dual hermitian $\co_\kk$-modules of
signatures $(1,0)$ and $(0,1)$, respectively, satisfying
$L(\mathfrak{A}_0,\mathfrak{A}_1) \in \mathrm{gen}(\LL_0(\infty) )$,
\item
$\mathfrak{B}$ is a self-dual hermitian $\co_\kk$-module of signature $(n-1,0)$ satisfying
$L(\mathfrak{A}_0, \mathfrak{B}) \iso \Lambda$.
\end{itemize}
We attach to this  triple the point
$(A_0,A_1,B) \in \mathcal{Y}_{(\LL_0,\Lambda)}(\C),$
where
\begin{align*}
A_0(\C) &= \mathfrak{A}_{0\R} / \mathfrak{A}_0 \\
A_1(\C) &= \mathfrak{A}_{1 \R} / \mathfrak{A}_1 \\
B(\C) &= \mathfrak{B}_\R / \mathfrak{B}
\end{align*}
as real Lie groups with $\co_\kk$-actions.  The complex structure on $A_0(\C)$ is given by the natural action of
$\kk_\R\iso \C$ on $\mathfrak{A}_{0 \R}$, and similarly for the complex structure on $B(\C)$.
The complex structure on $A_1(\C)$  is given by the complex  \emph{conjugate}
of the natural action of  $\kk_\R\iso \C$ on $\mathfrak{A}_{1 \R}$.
The elliptic curves $A_0$ and $A_1$ are endowed with their unique principal
polarizations, while $B$ is endowed with the polarization determined by the
symplectic form $\pol_B$ on $\mathfrak{B}\iso H_1(B(\C),\Z)$ defined  as in (\ref{symplectic}).

The construction $(\mathfrak{A}_0,\mathfrak{A}_1,\mathfrak{B}) \mapsto (A_0,A_1,B)$
establishes a bijection   from the
set of isomorphism classes of all such triples to the set of isomorphism classes
of the category $\mathcal{Y}_{(\LL_0,\Lambda)}(\C)$, and defines an
isomorphism of $0$-dimensional complex orbifolds
\begin{equation}\label{cm uniformization}
\mathcal{Y}_{(\LL_0,\Lambda)}(\C) \iso \bigsqcup_{ (\mathfrak{A}_0,\mathfrak{A}_1,\mathfrak{B})  }
\Gamma_{(\mathfrak{A}_0,\mathfrak{A}_1,\mathfrak{B})}
\backslash \{ y_{(\mathfrak{A}_0,\mathfrak{A}_1,\mathfrak{B})} \},
\end{equation}
where $y_{(\mathfrak{A}_0,\mathfrak{A}_1,\mathfrak{B})}$ is a single point on which
\[
\Gamma_{(\mathfrak{A}_0,\mathfrak{A}_1,\mathfrak{B})} = \Aut(\mathfrak{A}_0,\mathfrak{A}_1,\mathfrak{B} )
\]
acts trivially.
The morphism $\mathcal{Y}_{(\LL_0,\Lambda)}(\C) \to \mathcal{M}_\LL(\C)$ is
easy to describe in terms of (\ref{cm uniformization}) and  (\ref{uniformization}).
For each triple $(\mathfrak{A}_0,\mathfrak{A}_1,\mathfrak{B})$ we set
$\mathfrak{A}=\mathfrak{A}_1 \oplus \mathfrak{B}$, and send the point
$y_{(\mathfrak{A}_0,\mathfrak{A}_1,\mathfrak{B})}$ to the point of $\mathcal{D}_{(\mathfrak{A}_0,\mathfrak{A})}$
defined by the negative $\kk_\R$-line
$
L(\mathfrak{A}_0,\mathfrak{A}_1)_\R \subset L(\mathfrak{A}_0,\mathfrak{A})_\R.
$

\begin{remark}\label{rem:degree}
The uniformization (\ref{cm uniformization}) implies that
$\mathcal{Y}_{(\LL_0,\Lambda)}(\C)$
has $2^{1-o(d_\kk) }h_\kk^2$ points, each with $w_\kk^2 \cdot |\Aut(\Lambda)|$ automorphisms.
Thus the rational number
\[
\deg_\C \mathcal{Y}_{(\LL_0,\Lambda)} =
\sum_{ y\in \mathcal{Y}_{(\LL_0,\Lambda)}(\C)} \frac{1}{|\Aut(y)|}
\]
is given by the explicit formula
\[
\deg_\C \mathcal{Y}_{(\LL_0,\Lambda)}  =    \frac { h_\kk^2  } {  w_\kk^2   }  \cdot
   \frac { 2^{1-o(d_\kk) }    }{  |\Aut(\Lambda) |  } .
\]
\end{remark}
Moreover, we have the following proposition.

\begin{proposition}
Assume that either $n>2$, or $n=2$ and $\LL(\infty)$ contains, everywhere locally,  a nonzero isotropic vector.
\begin{enumerate}
\item
The rule  $L_0 \mapsto L_0  \oplus \Lambda$ establishes a bijection
\[
\mathrm{gen}(\LL_0(\infty)) \rightarrow \mathrm{gen}(\LL(\infty)).
\]
\item
The set $\mathcal{Y}_{(\LL_0, \Lambda)}(\C)$ has exactly one point on every connected component of $\mathcal{M}_{\LL}(\C)$.
\end{enumerate}
\end{proposition}

\begin{proof}
(1) Since both sets have $2^{1-o(d_\kk)}h_\kk$ elements (see Remark \ref{component count}), it suffices to show that the map is injective.
For any hermitian $\co_\kk$-modules  $L_0, L_0' \in \mathrm{gen}(\LL_0( \infty) )$
there are fractional ideals $\mathfrak{b}$ and $\mathfrak{b}'$ such that $\mathfrak{b} \iso L_0$  and $\mathfrak{b}'\iso L_0'$,
where the hermitian forms on $\mathfrak{b}$ and $\mathfrak{b}'$ are defined by   $-x\overline{y}$.
If  $L_0\oplus \Lambda \iso L_0' \oplus \Lambda$, then taking top exterior powers (in the category of $\co_\kk$-modules)
shows that $L_0 \iso L_0'$ as $\co_\kk$-modules.  But this implies that $\mathfrak{b}$ and $\mathfrak{b}'$ lie in the same ideal class,
and hence are isomorphic as hermitian $\co_\kk$-modules.  Therefore $L_0 \iso L_0'$ as hermitian $\co_\kk$-modules.
Claim (2) is an easy consequence of (1).
\end{proof}

Let $P\subset\Sl_2(\Z)$ be the subgroup of upper triangular matrices.
For each $\varphi \in S_{\LL_0}$,  define an \emph{incoherent}
Eisenstein series of weight $1$
\begin{equation}\label{incoherent Eisenstein}
E_{ \LL_0}(\tau, s, \varphi)=
\sum_{\gamma \in P \backslash \Sl_2(\Z)}
 \omega_{\LL_0}(\gamma)  \varphi  (0)  \cdot (c\tau+d)^{-1} \cdot
\mathrm{Im}(\gamma \tau)^{\frac{s}2}.
\end{equation}
Here  $\gamma =  \kabcd$, $\tau=u+iv\in \H$, and $s$ is a complex variable
with $\mathrm{Re}(s) \gg 0$.
The Eisenstein series  has meromorphic continuation to all $s$, and
is holomorphic at $s=0$.  As (\ref{incoherent Eisenstein}) is linear in $\varphi$,
it may be viewed as a function $E_{\LL_0}(\tau, s)$
taking values in the dual space $S_{\LL_0}^\vee$.
This particular Eisenstein series was studied in \cite{Scho} and  \cite[Section 2]{BY1}. Indeed,
if we pick any $L_0\in\mathrm{gen}(\LL_0(\infty))$ then
(\ref{incoherent Eisenstein}) is precisely the Eisenstein series denoted
$E_{L_0}(\tau,s,1)$ in \cite{BY1}, and depends only on the genus of $L_0$, not
on $L_0$ itself.

By \cite[Proposition 2.5]{BY1}, the completed Eisenstein series
\[
E^*_{ \LL_0}(\tau, s, \varphi) = \Lambda(\chi_\kk, s+1) \cdot E_{ \LL_0}(\tau, s, \varphi)
\]
satisfies the functional equation $E^*_{ \LL_0}(\tau, - s,\varphi) = - E^*_{ \LL_0}(\tau, s,\varphi)$, where
\[
\Lambda(\chi_\kk, s) = |d_\kk|^{\frac{s}2} \pi^{-\frac{s+1}2} \Gamma\big(\frac{s+1}{2}\big) L(\chi_\kk, s).
\]
In particular $E_{\LL_0}(\tau, 0) =0$.
The central derivative $E_{\LL_0}'(\tau, 0)$ at $s=0$ is a harmonic Maass form
of weight $1$ with representation $\omega_{\LL_0}^\vee$,
whose holomorphic part we denote (as in \cite[(2.26)]{BY1}) by
\begin{equation} \label{eq:HolPart}
 \mathcal E_{\LL_0}(\tau) =\sum_{m \gg -\infty} a_{\LL_0}^+(m) \cdot q^m.
\end{equation}

Up to a change of notation, the following proposition is due to Schofer \cite{Scho}; see also
\cite[Theorem 2.6]{BY1}.  Be warned that both references contain minor misstatements.
The formula of part (4) is misstated in \cite{Scho}, but the error is corrected in \cite{BY1}. The
formula of (2) is correct in \cite{Scho}, but is misstated in \cite{BY1}.

 \begin{proposition}  \label{prop:eisenstein coeff}
 Recall the finite set $\mathrm{Diff}_{\LL_0}(m)$ of odd cardinality from (\ref{diff set}),
and the function $\rho$ of (\ref{rho}).
 The coefficients $a^+_{\LL_0}(m) \in S_{\LL_0}^\vee$ are given by the following formulas.
\begin{enumerate}
\item
If $m<0$  then $a_{\LL_0}^+(m)=0$.
\item
The constant term  is
\[
 a_{\LL_0}^+(0,\varphi)   =
\varphi(0) \cdot \left(   \gamma +  \log\left|\frac{ 4 \pi }{d_\kk } \right|
-  2 \frac{ L ' (\chi_\kk,0)}{L(\chi_\kk,0)} \right)
\]
for every $\varphi\in S_{\LL_0}$. Here $\gamma=-\Gamma'(1)$ is Euler's constant.
\item
If $m>0$ and $|\mathrm{Diff}_{\LL_0}(m) | >1$, then $a_{\LL_0}^+(m)=0$.
\item
If  $m>0$ and   $\mathrm{Diff}_{\LL_0}(m)=\{p\}$ for a single prime $p$, then
\[
a^+_{\LL_0}(m, \varphi )
=
- \frac{   w_\kk}{ 2 h_\kk}\cdot  \rho \left( \frac{m  |d_\kk| }{p^\epsilon}  \right) \cdot  \ord_p(pm)\cdot \log(p)
\sum_{
\substack{  \mu \in \mathfrak{d}_\kk^{-1} \LL_{0,f} / \LL_{0,f}
\\ Q(\mu) =m } } 2^{s(\mu) } \varphi(\mu) .
\]
On the right hand side, $s(\mu)$ is the number of  primes  $q\mid d_\kk$ such that  $\mu_q= 0$,
$\epsilon$ is defined by (\ref{epsilon}), and $Q(\mu)=m$ is understood as an equality  in $\Q/\Z$.
\end{enumerate}
\end{proposition}

Define a \emph{coherent}  Eisenstein series of weight $-1$ associated to $\LL_0(\infty)$ by
\begin{equation}\label{coherent Eisenstein}
E_{\LL_0(\infty)}(\tau, s, \varphi )=
\sum_{\gamma \in P \backslash \Sl_2(\Z)}
\omega_{\LL_0}(\gamma) \varphi  (0) \cdot
 (c\tau +d) \cdot \mathrm{Im}(\gamma \tau)^{\frac{s}2+1 }.
\end{equation}
This is the Eisenstein series denoted $E_{L_0}(\tau,s,-1)$ in \cite[Section 2]{BY1},
for any choice of  $L_0\in \mathrm{gen}(\LL_0(\infty))$.
The following  relationship between the coherent and incoherent Eisenstein series
was first observed by Kudla \cite[(2.17)]{Ku:Integrals}, and is a special case of \cite[Lemma 2.3]{BY1}.

\begin{proposition} \label{prop:EisensteinRelation}
For any $\varphi \in  S_{\LL_0}$ the Eisenstein series (\ref{incoherent Eisenstein})
and (\ref{incoherent Eisenstein}) are related by the equality
\[
- 2\cdot  \overline{\partial} (E_{\LL_0}'(\tau, 0,\varphi )d\tau )
= E_{\LL_0(\infty)}(\tau, 0, \varphi) \cdot v^{-2} du\wedge dv
\]
of smooth $2$-forms on $\H$.
\end{proposition}

Suppose $L_0\in \mathrm{gen}(\LL_0(\infty))$.   Exactly as in
(\ref{globalized schwartz}), there is an $\SL_2(\Z)$-equivariant  isomorphism
$S_{L_0}^{\Delta_0} \iso S_{\LL_0}^{\Delta_0}$,
which allows us to define a non-holomorphic theta series
$\theta_{L_0} :\H \to (S^\vee_{\LL_0})^{\Delta_0}$
by
\[
\theta_{L_0}(\tau,\varphi ) = v
\sum_{   \lambda \in  \mathfrak{d}_\kk^{-1} L_0  } \varphi(\lambda) e^{ 2\pi i \langle \lambda,\lambda \rangle \overline{\tau}}.
\]
 for any $\varphi \in S_{\LL_0}^{\Delta_0}$. The following
 proposition follows  from \cite[Proposition 2.2]{BY1}.

\begin{proposition}[Siegel-Weil formula] \label{Siegel-Weil}
The coherent Eisenstein series (\ref{coherent Eisenstein}) is related to the above
theta series by
\[
\frac{2^{o(d_\kk)}}{h_\kk}
\sum_{L_0 \in \gen(\LL_0(\infty))} \theta_{L_0}(\tau)
=  E_{\LL_0(\infty)}(\tau, 0 ).
\]
\end{proposition}

For each $\varphi\in S_\Lambda$ define
\[
 R_\Lambda(m,\varphi)  =
  \sum_{  \substack{  \lambda\in \mathfrak{d}_\kk^{-1} \Lambda  \\ \langle \lambda,\lambda\rangle=m  } } \varphi(\lambda).
\]
These representation numbers are the Fourier coefficients of a holomorphic
$S_\Lambda^\vee$-valued modular form
\[
\theta_\Lambda(\tau,\varphi) = \sum_{m\in\Q} R_\Lambda(m,\varphi) \cdot q^m
\in M_{n-1}(\omega_\Lambda^\vee).
\]

Given any  $F\in S_n(\overline{\omega}_\LL)$ with Fourier expansion
\[F(\tau) = \sum_{m\in \Q_{\ge 0}} b(m)  \cdot  q^m,\]
define the \emph{Rankin-Selberg convolution $L$-function}
\begin{equation} \label{eq:vectorRankin-Selberg}
L(F, \theta_\Lambda , s) =
 \Gamma \big(\frac{s}{2} +n-1 \big)
\sum_{m \in \Q_{>0}}  \frac{ \big\{\overline{b(m)},  R_\Lambda(m) \big\}  }
{ (4\pi m)^{\frac{s}{2}+n-1 }}.
\end{equation}
On the right hand side the pairing  is the tautological pairing between $S_\LL$ and its dual.
 The  inclusion   $\mathfrak{d}_\kk^{-1}\Lambda/\Lambda \to \mathfrak{d}_\kk^{-1} \LL_f / \LL_f$
determines a canonical surjection
$S_\LL \to S_\Lambda$, and hence an injection on dual spaces.
In particular, this allows us to view  $R_\Lambda(m)$ as an element of $S_\LL^\vee$.
The usual unfolding method shows that
\[
L(F, \theta_\Lambda , s)
=\int_{\Sl_2(\Z) \backslash \mathbb H}
 \left\{\overline{F (\tau) } ,  E_{\LL_0}(\tau,s) \otimes \theta_\Lambda(\tau)    \right\} v^{n-2} \, du\, dv,
\]
where $\tau=u+iv$  and $E_{\LL_0}(\tau,s)$
is the incoherent Eisenstein series of (\ref{incoherent Eisenstein}).
On the right hand side we are using the canonical isomorphism
$S_\LL^\vee \iso S_{\LL_0}^\vee \otimes S_\Lambda^\vee$
to view $E_{\LL_0}(\tau,s) \otimes \theta_\Lambda(\tau)$ as an
$S_\LL^\vee$-valued function.
Of course $L(F, \theta_\Lambda , 0)=0$, as the Eisenstein series vanishes at $s=0$.

\begin{theorem} \label{thm:CM value}
For every $f\in H_{2-n}(\omega_{\LL})^\Delta$ the CM value
\[
\Phi_\LL( \mathcal{Y}_{(\LL_0,\Lambda)}  ,f ) =
 \sum_{y\in \mathcal{Y}_{(\LL_0,\Lambda)} (\C) }
\frac{ \Phi_\LL(y,f)}{ |  \Aut(y) | }
\]
satisfies
\[
\frac{1} { \deg_\C \mathcal{Y}_{(\LL_0, \Lambda)} }\cdot
\Phi_\LL( \mathcal{Y}_{(\LL_0,\Lambda)}  ,f )
=   - L' \big(\xi(f), \theta_\Lambda, 0 \big) +
  \CT \big[\{ f^+ ,   \mathcal E_{\LL_0}\otimes  \theta_\Lambda \} \big] .
\]
Here  the differential operator
\[
\xi : H_{2-n}(\omega_{\LL} ) \to S_n( \overline{\omega}_\LL)
\]
is defined by   (\ref{defxi}),  and
 $
 \CT [ \{ f^+ ,  \mathcal {E}_{\LL_0} \otimes \theta_\Lambda  \} ]
 $
is the constant term of the $q$-expansion of
 $\{ f^+ ,  \mathcal {E}_{\LL_0} \otimes \theta_\Lambda  \}$.
\end{theorem}

\begin{proof}
This is really a special case of \cite[Theorem 4.7]{BY1}, but beware that the
statement of  [\emph{loc.~cit.}]  contains a sign error.
We sketch the main ideas for the convenience of the reader.
Recall from Section \ref{ss:complex uniformization}
that the complex points of $\mathcal{Y}_{(\LL_0,\Lambda)}$ are indexed by
triples $(\mathfrak{A}_0,\mathfrak{A}_1,\mathfrak{B})$.  If we fix such a point
$y_{(\mathfrak A_0, \mathfrak A_1, \mathfrak B)} \in \mathcal{Y}_{(\LL_0,\Lambda)}(\C)$
and abbreviate
$L_0=L(\mathfrak A_0, \mathfrak A_1)$ and $L=\Lambda \oplus L_0$, then
the theta function $\Theta_L ( \tau,z)$ appearing in (\ref{eq:AutoGreen}) admits a factorization
\[
\Theta_L(\tau, y_{(\mathfrak A_0, \mathfrak A_1, \mathfrak B)})=\theta_{\Lambda}(\tau) \otimes \theta_{L_0}(\tau)
\]
at  $z=y_{(\mathfrak A_0, \mathfrak A_1, \mathfrak B)}$.  Hence
 \[
\Phi_{\LL}(y_{(\mathfrak A_0, \mathfrak A_1, \mathfrak B)}, f)=
 \int_{\Sl_2(\Z)\backslash \mathbb H}^{\reg} \{ f, \theta_{\Lambda} \otimes \theta_{L_0}\} \, d\mu(\tau).
\]
Summing over all  complex points of $\mathcal{Y}_{(\LL_0,\Lambda)}$ yields
\begin{align*}
\Phi_{\LL}(\mathcal{Y}_{(\LL_0,\Lambda)} , f)
 &=
 \sum_{\substack{  (\mathfrak{A}_0,\mathfrak{A}_1,\mathfrak{B}) \\
 L(\mathfrak A_0, \mathfrak A_1) \in \LL_0(\infty)
 \\ L(\mathfrak A_0, \mathfrak B) \iso \Lambda}}
 \frac{ 1 }{ |\Aut(\mathfrak{A}_0,\mathfrak{A}_1,\mathfrak{B} )| }
 \int_{\Sl_2(\Z)\backslash \mathbb H}^{\reg} \left\{  f,   \theta_{\Lambda} \otimes
 \theta_{L(\mathfrak A_0, \mathfrak A_1) }
 \right\}
\,  d\mu(\tau)
 \\
 &= \frac{ h_\kk} { w_\kk^2 |\Aut(\Lambda)| }
 \sum_{L_0 \in \gen(\LL_0(\infty))}
 \int_{\Sl_2(\Z)\backslash \mathbb H}^{\reg} \left\{ f,  \theta_{\Lambda}\otimes  \theta_{L_0}\right\} \, d\mu(\tau)
 \\
 &=\frac{\deg_\C \mathcal{Y}_{(\LL_0, \Lambda)}}{2}  \int_{\Sl_2(\Z)\backslash \mathbb H}^{\reg}\{ f(\tau), \theta_{\Lambda}(\tau) \otimes E_{\LL_0(\infty)} (\tau, 0) \}  \, d\mu(\tau)
\end{align*}
by the Siegel-Weil formula (Proposition \ref{Siegel-Weil}) and Remark \ref{rem:degree}.
Applying Proposition  \ref{prop:EisensteinRelation} and Stokes' theorem, a simple calculation,
as in   the proof of \cite[Theorem 4.7]{BY1},  shows that
\begin{eqnarray*}\lefteqn{
\frac{1}{\deg_\C \mathcal{Y}_{(\LL_0, \Lambda)} } \cdot
\Phi_{\LL}( \mathcal{Y}_{(\LL_0,\Lambda)} , f) } \\
&= &
- \int_{\Sl_2(\Z)\backslash \mathbb H}^{\reg}
\{ f, \theta_{\Lambda} \otimes  \overline{\partial} E_{\LL_0}'(\tau, 0) \, d\tau\} \\
&=& -   \int_{\Sl_2(\Z)\backslash \mathbb H}^{\reg} d\{f, \theta_{\Lambda}\otimes E_{\LL_0}'(\tau, 0) \, d\tau\}
 -  \int_{\Sl_2(\Z)\backslash \mathbb H}
 \{\overline{ \xi(f)} , \theta_{\Lambda} \otimes E_{\LL_0}'(\tau, 0) \} v^{n}\, d\mu(\tau)
  \\
  &=&
   \CT[\{ f^+(\tau) ,  \theta_\Lambda(\tau) \mathcal E_{\LL_0}(\tau) \}] - L'(\xi(f), \theta_\Lambda, 0) ,
\end{eqnarray*}
as claimed.
\end{proof}

\begin{remark}\label{rem:lots of pairings}
Let $\varphi_\mathfrak{r} \in S_\LL$ be the characteristic function of
\[
\mathfrak{r}^{-1}\LL_f/\LL_f \subset \mathfrak{d}_\kk^{-1}\LL_f/\LL_f.
\]
By abuse of notation we denote again by $\varphi_\mathfrak{r}$ the similarly defined
elements of $S_{\LL_0}$ and $S_\Lambda$.
If $A$ is an element of  $S_\LL^\vee$, $S_{\LL_0}^\vee$, or $S_\Lambda^\vee$
abbreviate $A(\mathfrak{r}) = A(\varphi_\mathfrak{r})$.  For example
\[
R_\Lambda(m,\mathfrak{r}) = R_\Lambda(m,\varphi_\mathfrak{r})
= \sum_{  \substack{  \lambda\in \mathfrak{r}^{-1} \Lambda  \\ \langle \lambda,\lambda\rangle=m  } } 1.
\]
\end{remark}

The following is a restatement of  Theorem \ref{thm:CM value} in the case $f=f_{m,\mathfrak{r}}$.

\begin{corollary}\label{cor:cm values}
The harmonic form  $f_{m,\mathfrak{r}}$   of  Lemma \ref{lem:f-mr} satisfies
\begin{align*}
\frac{1}{ \deg_\C \mathcal{Y}_{(\LL_0, \Lambda)} }\cdot
\Phi_\LL( \mathcal{Y}_{(\LL_0,\Lambda)}  ,f_{m,\mathfrak{r}} )
 & =  - L'( \xi( f_{m,\mathfrak{r}}  ) , \theta_\Lambda , 0)
+
c^+_{m,\mathfrak{r}}( 0 , 0)   \cdot a_{\LL_0}^+(0 ,\mathfrak{r}  )   \cdot R_\Lambda( 0, \mathfrak{r} ) \\
 & \phantom{=} +
\sum_{\substack{  m_1 , m_2 \in \Q_{\ge 0} \\ m_1 + m_2 =m  }}
 a_{\LL_0}^+(m_1 ,\mathfrak{r}  )  \cdot R_\Lambda( m_2, \mathfrak{r} ) .
\end{align*}
\end{corollary}


\section{The metrized cotautological bundle}


In this section we recall some generalities on metrized line bundles, and introduce
the metrized cotautological bundle.  Fix a pair $(\LL_0,\Lambda)$ as in
(\ref{CM genus}), and let $\LL=\LL_0\oplus\Lambda$ as in Remark
\ref{rem:sum}.


\subsection{Metrized line bundles}
\label{ss:line bundles}


As in \cite{Ho3}, the  canonical map $\mathcal{Y}_{(\LL_0,\Lambda)} \to \mathcal{M}_\LL$
induces a  linear functional
\begin{equation}\label{first degree}
\widehat{\mathrm{CH}}^1_\C(\mathcal{M}_\LL^*) \to \C
\end{equation}
called the \emph{arithmetic degree along $\mathcal{Y}_{  (\LL_0,\Lambda)  }$}, and  denoted
$\widehat{\mathcal{Z}} \mapsto [\widehat{\mathcal{Z}}    : \mathcal{Y}_{  (\LL_0,\Lambda)  } ]$.
Because $\mathcal{Y}_{(\LL_0,\Lambda)}$
is proper over $\co_\kk$, it does not meet the boundary  of $\mathcal{M}_\LL^*$.  Thus
the arithmetic degree along  $\mathcal{Y}_{  (\LL_0,\Lambda)  }$ can be defined and computed
entirely on the open Shimura variety $\mathcal{M}_\LL$.
This is most easily done  using the  language of metrized line bundles, which
provides a rudimentary  intersection theory on  $\mathcal{M}_\LL$.

A \emph{metrized line bundle}  $\widehat{\mathcal{L}}=(\mathcal{L} , || \cdot||)$
on $\mathcal{M}_\LL$  consists of a line bundle $\mathcal{L}$  and a hermitian metric $||\cdot||$
on the complex points $\mathcal{L}(\C)$. The isomorphism classes of metrized line bundles form a group
$\widehat{\Pic}(\mathcal{M}_\LL)$ under tensor product.
An \emph{arithmetic divisor} $\widehat{\mathcal{Z} }=(\mathcal{Z},\Phi)$ is a pair consisting
of a divisor (with integral coefficients) $\mathcal{Z}$ on $\mathcal{M}_\LL$ and a
Green function on $\mathcal{M}_\LL(\C)$  for the complex fiber $\mathcal{Z}(\C)$.
The arithmetic divisors form a group $\widehat{\mathrm{Div}}(\mathcal{M}_\LL)$ under addition.
If we start with an arithmetic divisor  $\widehat{\mathcal{Z}}$, the
constant function $1$ on $\mathcal{M}_\LL$ defines a rational section $\bm{s}$ of
the line bundle $\mathcal{L}=\co(\mathcal{Z})$ associated to $\mathcal{Z}$,
and there is a unique metric $||\cdot||$ on $\mathcal{L}$ satisfying
$-\log || \bm{s} ||_z^2 =  \Phi(z).$ This establishes a surjection
\[
\widehat{\mathrm{Div}}(\mathcal{M}_\LL) \to \widehat{\Pic}(\mathcal{M}_\LL).
\]
A similar discussion holds with $\mathcal{M}_\LL$ replaced by $\mathcal{Y}_{( \LL_0,\Lambda)}$,
and the morphism $\mathcal{Y}_{( \LL_0,\Lambda)} \to \mathcal{M}_\LL$ induces a pullback homomorphism
\[
\widehat{\Pic}(\mathcal{M}_\LL) \to \widehat{\Pic}(\mathcal{Y}_{( \LL_0,\Lambda )}).
\]

As in \cite[Chapter 2.1]{KRY2} there is a linear functional, called the \emph{arithmetic degree},
\[
\widehat{\deg} : \widehat{\mathrm{Pic}}(\mathcal{Y}_{( \LL_0,\Lambda )}) \to \R.
\]
The composition
\begin{equation}\label{second degree}
\widehat{\Pic} (\mathcal{M}_\LL) \to
 \widehat{\Pic}(\mathcal{Y}_{(\LL_0,\Lambda)} )  \map{\widehat{\deg}}\R,
\end{equation}
is again denoted $[\; \cdot : \mathcal{Y}_{(\LL_0,\Lambda)}  ]$.
Taking $f=f_{m,\mathfrak{r}}$ in  (\ref{total cycle})  and restricting to the open Shimura variety
$\mathcal{M}_\LL$ defines an arithmetic divisor
\[
\widehat{\mathcal{Z}}_\LL( f_{m,\mathfrak{r}} )
= \big( \mathcal{Z}_\LL( f_{m,\mathfrak{r}} )  , \Phi_\LL(  f_{m,\mathfrak{r}} )   \big)
 \in \widehat{\Div} (\mathcal{M}_\LL)
\]
satisfying
\[
[ \widehat{\mathcal{Z}}_\LL( f_{m,\mathfrak{r }}  ) : \mathcal{Y} _{(\LL_0,\Lambda)} ]
= [ \widehat{\mathcal{Z}}^\mathrm{total}_\LL(f_{m,\mathfrak{r}}) : \mathcal{Y} _{(\LL_0,\Lambda)} ].
\]
The pairing on the left is (\ref{second degree}), while the pairing on the right is (\ref{first degree}).
If   the intersection $\mathcal{X}=\mathcal{Z}_\LL( m,\mathfrak{r} ) \cap \mathcal{Y}_{(\LL_0,\Lambda)}$
has dimension $0$,  we have the explicit formula
\[
[ \widehat{\mathcal{Z}}_\LL( f_{m,\mathfrak{r} }  ) : \mathcal{Y} _{(\LL_0,\Lambda)} ]
= I( \mathcal{Z}_\LL( m,\mathfrak{r}): \mathcal{Y}_{(\LL_0,\Lambda)})
+ \Phi_\LL( \mathcal{Y}_{(\LL_0,\Lambda)} ,  f_{m,\mathfrak{r}}  )
\]
where the first term is the
\emph{finite intersection multiplicity}\footnote{As both $\mathcal{Z}$ and $\mathcal{Y}$ are
Cohen-Macaulay, the finite intersection multiplicity agrees with the more natural
\emph{Serre intersection multiplicity} defined as  in \cite[p.~11]{SABK}  or \cite[Section 3.1]{Ho3}.  This follows from
 \cite[p.~111]{Se00}.}
\[
I( \mathcal{Z}_\LL(m,\mathfrak{r}): \mathcal{Y}_{(\LL_0,\Lambda)})
= \sum_{ \mathfrak{p} \subset\co_\kk }   \log(\mathrm{N}(\mathfrak{p}))
\sum_{y \in  \mathcal{X}  (\F_\mathfrak{p}^\alg)}
 \frac{    \mathrm{length}_{ \co_{\mathcal{X}   ,y}}(\co_{ \mathcal{X}  , y })}{|\Aut(y) |} ,
\]
and the second term is defined as in Theorem \ref{thm:CM value}.


\subsection{The cotautological bundle}
\label{ss:cotaut bundle}


 Let $S$ be a  scheme and $\pi:A\to S$  an abelian scheme.
 As in \cite[Section 2.1.6]{Lan}, define  a coherent $\co_S$-module,
 the \emph{algebraic de Rham cohomology} of $A$,
 as  the hypercohomology  $H_{dR}^i(A) = \mathbb{R}^i\pi_*( \Omega^\bullet_{A/S} )$
  of the de Rham complex
  \[
  0\to \co_{A} \to \Omega_{A/S}^1 \to \Omega_{A/S}^2 \to \cdots .
  \]
The \emph{algebraic de Rham homology} $H^{dR}_1(A) = \Hom_{\co_S}( H^1_{dR}(A) ,\co_S)$
sits in an exact sequence
\[
0 \to \mathrm{Fil}(A) \to H^{dR}_1(A) \to \Lie(A) \to 0,
\]
and  $\mathrm{Fil}(A)$ is canonically isomorphic to the $\co_S$-dual of $\Lie(A^\vee)$.
Let  $(A^\univ_0,A^\univ)$  denote the universal object over $\mathcal{M}$,
and recall that  $A^\univ$ is endowed with an $\co_\kk$-stable $\co_\mathcal{M}$-submodule
$\mathcal{F}^\univ \subset \Lie(A^\univ)$ such that the quotient
$\Lie(A^\univ)/\mathcal{F}^\univ$ is locally free of rank one.

\begin{definition}
The \emph{cotautological bundle}  on $\mathcal{M}$ is the line bundle
 \[
 \mathcal{T} =
 \Hom_{\co_\mathcal{M}} ( \mathrm{Fil}  ( A_0^\univ ) ,   \Lie(A^\univ) / \mathcal{F}^\univ ).
 \]
Denote by $\mathcal{T}_\LL$ the restriction of $\mathcal{T}$ to $\mathcal{M}_\LL$.
 \end{definition}

The universal abelian
scheme $A^\univ \to \mathcal{M}_{(n-1,1)}$ extends to a semi-abelian scheme
$A^* \to \mathcal{M}^*_{(n-1,1)}$, and the universal subsheaf
$\mathcal{F}^\univ \subset \Lie(A^\univ)$ extends canonically  to a subsheaf
$\mathcal{F}^*\subset \Lie(A^*)$ by \cite[Theorem 2.5.2]{Ho3}.
The cotautological bundle $\mathcal{T}$   therefore extends to a line bundle
 \[
\mathcal{T}^* = \Hom_{\co_{\mathcal{M}^*}} (\mathrm{Fil}(A_0^\univ) , \Lie(A^*) / \mathcal{F}^*)
\]
 on  $\mathcal{M}^*$, and  the restriction of $\mathcal{T}^*$ to $\mathcal{M}^*_\LL$
is denoted $\mathcal{T}^*_\LL$.

Recall from Section \ref{ss:complex uniformization} that each connected component of
$\mathcal{M}_\LL(\C)$
admits a uniformization $\mathcal{D}_{(\mathfrak{A}_0,\mathfrak{A})} \to \mathcal{M}_\LL(\C)$,
where $\mathcal{D}_{(\mathfrak{A}_0,\mathfrak{A})}$ is the space of negative $\kk_\R$-lines
in $L(\mathfrak{A}_0,\mathfrak{A})$.  The hermitian symmetric domain
$\mathcal{D}_{(\mathfrak{A}_0,\mathfrak{A})}$  carries a tautological bundle
whose fiber at a point $z$ is the line $z$, made into a complex vector space using the fixed
isomorphism $\kk_\R\iso \C$.  The following proposition  explains the connection between
this bundle and the cotautological bundle.

\begin{proposition}\label{prop:analytic taut}
At any point $z\in \mathcal{D}_{(\mathfrak{A}_0,\mathfrak{A})}$  there is a canonical complex-linear
isomorphism
 \[
 \fiber^\vee:\mathcal{T}_{\LL,z} \iso \Hom_\C( z ,\C).
 \]
\end{proposition}

\begin{proof}
Identify $\kk\otimes_\Q\C\iso \C\times\C$ in such a way that
$x\otimes 1\mapsto (x,\overline{x})$, and define idempotents $e=(1,0)$
and $\overline{e}=(0,1)$.   If $X$ is any complex vector space with a commuting
action of $\kk$, then $\kk$ acts on $e X$ through the fixed embedding $\kk\to \C$,
and acts on $\overline{e} X$ through the conjugate embedding.

 Let $(A_0 , A_z) \in \mathcal{M}_\LL(\C)$ be the image of $z$ under
 $\mathcal{D}_{(\mathfrak{A}_0,\mathfrak{A}) }\to \mathcal{M}_\LL(\C)$.
 The map
 \[
 \mathfrak{A}_{0\C} \iso H_1^{dR}(A_0) \map{\overline{e}}
 \overline{e} H_1^{dR}(A_0)
 \iso \mathrm{Fil}(A_0)
  \]
 restricts to a $\kk_\R$-linear isomorphism $\mathfrak{A}_{0\R} \iso \mathrm{Fil}(A_0)$,
 while the quotient map
 \[
 \mathfrak{A}_\C \iso H_1^{dR}(A_z) \to \Lie(A_z)
 \]
 restricts to a $\kk_\R$-linear
 isomorphism $\mathfrak{A}_\R \iso \Lie(A_z)$.  Thus we obtain isomorphisms
 \[
 L(\mathfrak{A}_0,\mathfrak{A})_\R \iso \Hom_{\kk_\R} (\mathfrak{A}_{0\R} , \mathfrak{A}_\R)
 \iso \Hom_{\kk_\R} (\mathrm{Fil}(A_0)  , \Lie(A)_z ),
 \]
 and tracing through the constructions of Section \ref{ss:complex uniformization} shows that
 their composition identifies $z^\perp$ with $\Hom_{\kk_\R} (\mathrm{Fil}(A_0)  , \mathcal{F}_z )$.
 The surjection
 \[
 L(\mathfrak{A}_0,\mathfrak{A})_\R \to \Hom_{\kk_\R} (\mathrm{Fil}(A_0)  , \Lie(A_z)/\mathcal{F}_z )
 \]
 therefore has kernel $z^\perp$, and identifies
 $z\iso \Hom_{\kk_\R} (\mathrm{Fil}(A_0)  , \Lie(A_z)/\mathcal{F}_z )$.  The inverse of this map
is a $\kk_\R$-linear isomorphism
 \begin{equation}\label{conj taut}
 \fiber:\mathcal{T}_{\LL,z} \to z
 \end{equation}
which is \emph{not} $\C$-linear.  The point is that the signature conditions imposed on
$\Lie(A_0^\univ)$ and $\Lie(A^\univ)/\mathcal{F}^\univ$
 imply that the action of $\kk_\R$ on the fiber  $\mathcal{T}_{\LL,z}$ is through the
 complex conjugate of the fixed isomorphism $\kk_\R\iso \C$.  Thus   $\fiber$  complex-\emph{conjugate}-linear.   The map
 \[
 \fiber^\vee:\mathcal{T}_{\LL,z} \iso \Hom_\C( z ,\C)
 \]
 defined by $\fiber^\vee( s ) = \langle \cdot, \fiber(s) \rangle$ defines the desired complex-linear isomorphism.
 \end{proof}

We use (\ref{conj taut}) to  metrize the cotautological bundle $\mathcal{T}_\LL$:
the norm of a section $\bm{s}$ is
 \begin{equation}\label{taut metric}
 || \bm{s} ||_z^2 = - 4\pi e^\gamma \cdot    \langle \fiber(\bm{s}) ,\fiber(\bm{s})  \rangle,
 \end{equation}
 where $\gamma=-\Gamma'(1)$ is Euler's constant.
The cotautological bundle endowed with the above metric is denoted
\[
\widehat{\mathcal{T}}_\LL   \in \widehat{\Pic}(\mathcal{M}_\LL).
\]

\begin{proposition}
For any nonzero rational section $\bm{s}$ of $\mathcal{T}^*_\LL$, the arithmetic divisor
\[
\widehat{\mathrm{div}}(\bm{s}) = ( \mathrm{div}(\bm{s}) , - \log||\bm{s}||^2 )
\]
defines a  class
\begin{equation}\label{cotaut def}
\widehat{\mathcal{T}}^*_\LL \in \widehat{\mathrm{CH}}^1_\R(\mathcal{M}_\LL^*),
\end{equation}
which does not depend on the choice of $\bm{s}$.
\end{proposition}

\begin{proof}
The only thing to check is that the Green function $- \log||\bm{s}||^2$ has at worst the $\log$-$\log$ error terms at the boundary allowed in the   Burgos-Kramer-K\"uhn theory \cite{BKK,BBK}.

Using the complex coordinates of Section \ref{ss:torcomp}, it suffices to show that for any primitive isotropic vector $\ell\in L(\mathfrak{A}_0,\mathfrak{A})$ and any boundary point
$(\frakz_0,0)\in \widetilde V_\eps(\ell)$ the function
\[
\log|\ell_z^2| = -\log (\norm(\frakz,\tau)/2)
\]
determining the metric and its differentials have log-log
 growth along $q_r=0$ on the neighborhood $B_\delta(\frakz_0,0)$ for some $\delta>0$.
This  is done in Step 2 of the proof of Theorem \ref{thm:bndgrowth}.
\end{proof}


\subsection{The Chowla-Selberg formula}


This section is devoted to  studying the image of $\widehat{\mathcal{T}}_\LL$
under the arithmetic intersection
\[
[ \; \cdot :  \mathcal{Y}_{(\LL_0,\Lambda)} ]  : \widehat{\Pic} (\mathcal{M}_\LL) \to \R
\]
of  Section \ref{ss:line bundles}.   The main result is the following theorem, whose proof ultimately rests on the
Chowla-Selberg formula.

\begin{theorem}\label{thm:taut degree}
For any $\varphi\in S_{\LL_0}$, the metrized cotautological bundle satisfies
\[
\varphi(0) \cdot  [  \widehat{\mathcal{T}}_\LL : \mathcal{Y}_{(\LL_0,\Lambda)}  ]
= -   \deg_\C  \mathcal{Y}_{(\LL_0, \Lambda)}  \cdot     a_{\LL_0}^+(0 , \varphi)   ,
\]
where $a_{\LL_0}^+(0 ) \in S_{\LL_0}^\vee$ is defined by
(\ref{eq:HolPart}), and $\deg_\C  \mathcal{Y}_{(\LL_0,\Lambda)} $
is defined in Remark \ref{rem:degree}.

If we take $\varphi$ to be the characteristic function
of $\mathfrak{r}^{-1} \LL_{0,f}/\LL_{0,f}$, as in  Remark \ref{rem:lots of pairings}, this formula reduces to
\[
  [  \widehat{\mathcal{T}}_\LL^{\otimes R_\Lambda(m, \mathfrak{r}) }  :  \mathcal{Y}_{ ( \LL_0, \Lambda) } ]
  = -   \deg_\C \mathcal{Y}_{(\LL_0, \Lambda)}  \cdot
  a_{\LL_0}^+(0  ,\mathfrak{r}) \cdot R_\Lambda(m,\mathfrak{r})
\]
for all $m\in \Q_{\ge 0}$ and $\mathfrak{r}\mid\mathfrak{d}_\kk$.
\end{theorem}

The proof requires some preparation.  First we state the  Chowla-Selberg formula in a form
suited to our purposes. Suppose $E$ is an elliptic curve over $\C$ with complex multiplication by $\co_\kk$.
Fix a model of $E$ over a finite extension $K/\Q$ contained in $\C$ and  large enough that $E$ has
everywhere good reduction, and let $\pi:\mathcal{E} \to \Spec(\co_K)$ be the N\'eron model of $E$ over $\co_K$.
Let $\omega$ be a nonzero rational
section of the line bundle $\pi_*\Omega^1_{\mathcal{E}/\co_K}$ on $\Spec(\co_K)$ with divisor
\[
\mathrm{div}(\omega) =\sum_{\mathfrak{q}} m(\mathfrak{q}) \cdot \mathfrak{q},
\]
where the sum is over the closed points $\mathfrak{q}\in \Spec(\co_K)$.
The \emph{Faltings height} of $E$ is defined as
\[
h_\mathrm{Falt}(E) =  \frac{1}{[K:\Q]} \left(
 \sum_\mathfrak{q} \log (\mathrm{N}(\mathfrak{q})) \cdot m(\mathfrak{q})
- \frac{1}{2}\sum_{  \tau: K \to \C    }  \log \left|
\int_{\mathcal{E}^\tau(\C)}  \omega^\tau \wedge \overline{\omega^\tau}\;  \right| \right).
\]
It  is independent of the choice of $K$, the model of $E$ over $K$, and
the section $\omega$.  The Chowla-Selberg formula implies \cite{Co} that
\begin{equation}\label{chowla-selberg}
- 2h_\mathrm{Falt}(E)   =    \log(2\pi) + \frac{1}{2} \log |d_\kk|   +  \frac{ L'(\chi_\kk,0) }{ L(\chi_\kk , 0)  }.
\end{equation}

Recall that $\mathcal{Y}_{(\LL_0,\Lambda)}$   carries over it a universal triple of
abelian schemes $(A^\univ_0,A^\univ_1,B^\univ)$.  Define a line bundle
\[
\mathrm{coLie}(A^\univ_0) = \pi_* \Omega^1_{A^\univ_0/\mathcal{Y}_{(\LL_0,\Lambda)}}
\]
on $\mathcal{Y}_{(\LL_0,\Lambda)}$, where $\pi: A^\univ_0 \to \mathcal{Y}_{(\LL_0,\Lambda)}$
is the structure morphism.   A vector $\omega \in \mathrm{coLie}(A^\univ_{0,y})$ in the fiber at
a complex point $y\in \mathcal{Y}_{(\LL_0,\Lambda)}(\C)$
is a global  holomorphic $1$-form on $A^\univ_{0,y}(\C)$, and we denote by
\[
\widehat{\mathrm{coLie}}(A^\univ_0)  \in \widehat{\mathrm{Pic}}( \mathcal{Y}_{(\LL_0,\Lambda)} )
\]
the line bundle $\mathrm{coLie}(A^\univ_0)$ endowed with the metric
\begin{equation}\label{faltings metric}
|| \omega ||_y^2 = \left|  \int_{A^\univ_{0,y}(\C) }   \omega \wedge \overline{\omega} \; \right|.
\end{equation}
With this definition,
\begin{equation}\label{metrized colie}
\widehat{\deg} \  \widehat{\mathrm{coLie}}(A^\univ_0)
=  \sum_{ (A_0, A_1,B) \in \mathcal{Y}_{(\LL_0,\Lambda)  }(\C) }
\frac{2 h_{\mathrm{Falt}}(A_0)}{|\Aut(A_0,A_1,B)| }.
\end{equation}

Of course $\Lie(A_0^\univ)$ is  isomorphic to the dual of
$\mathrm{coLie}(A_0^\univ)$.  Denote by
\begin{equation}\label{Faltings bundle}
\widehat{\Lie}(A^\univ_0)  \in \widehat{\mathrm{Pic}}( \mathcal{Y}_{(\LL_0,\Lambda)} )
\end{equation}
the line bundle $\Lie(A^\univ_0)$ with the  metric dual to (\ref{faltings metric}).
More explicitly, if we endow  $\mathfrak{A}_0=H_1(A^\univ_{0,y}(\C) ,\Z)$ with its
hermitian form $h_{\mathfrak{A}_0}$ as in Section \ref{ss:moduli},
then $\Lie(A^\univ_{0,y}) \iso \mathfrak{A}_{0\R}$ as real vector spaces, and
$
|| v ||_y^2 =  |d_\kk|^{- \frac{1}{2}}  h_{\mathfrak{A}_0}(v,v)
$
for any $v\in \Lie(A^\univ_{0,y})$.

\begin{lemma}\label{lem:metrized lie}
The metrized line bundle $\widehat{\mathrm{Lie}}(A^\univ_0)$ satisfies
\[
\frac{1}{\deg_\C \mathcal{Y}_{(\LL_0, \Lambda)} }
\cdot \widehat{\deg} \ \widehat{\Lie}(A^\univ_0)
=   \log(2\pi)+  \frac{ 1}{2 } \log |d_\kk|  +   \frac{ L'(\chi_\kk,0) }{ L(\chi_\kk , 0)  }.
\]
Of course we may  define $\widehat{\Lie}(A^\univ_1)$ in the same manner as (\ref{Faltings bundle}),
and the stated equality holds with $A^\univ_0$ replaced by $A^\univ_1$.
\end{lemma}

\begin{proof}
Combine (\ref{metrized colie}) and the Chowla-Selberg formula (\ref{chowla-selberg}).
\end{proof}

For any positive  $c\in\R$, define the \emph{twisted trivial bundle}
\[
\widehat{\bm{1}}(c)  \in \widehat{\mathrm{Pic}}( \mathcal{Y}_{(\LL_0,\Lambda)} )
\]
as the structure sheaf $\co_{   \mathcal{Y}_{(\LL_0,\Lambda)} }$
endowed with the metric $|| f||^2_y = c \cdot |f(y)|^2$.
It is clear from the definitions that
\begin{equation}\label{twist degree}
\widehat{\deg} \  \widehat{\bm{1}} (c)
=- \log(c) \cdot \deg_\C \mathcal{Y}_{(\LL_0, \Lambda)}  .
\end{equation}

\begin{lemma}\label{lem:metrized pullback}
There is an isomorphism
\[
\widehat{\mathcal{T}}_\LL |_{ \mathcal{Y}_{(\LL_0,\Lambda)} } \iso
\widehat{\Lie}(A_0^\univ) \otimes
\widehat{\Lie}(A_1^\univ)
 \otimes \widehat{\bm{1}} \left(  16\pi^3 e^\gamma  \right)
\]
of metrized line bundles on $\mathcal{Y}_{(\LL_0,\Lambda)}$.
\end{lemma}

\begin{proof}
Recall that $\mathcal{M}_\LL$ carries a universal pair of abelian schemes
$(A^\univ_0,A^\univ)$, and that $A^\univ$ comes with a universal
$\co_{\mathcal{M}_\LL}$-submodule   $\mathcal{F}^\univ \subset \Lie(A^\univ)$.
  By definition of the morphism
 $\mathcal{Y}_{(\LL_0,\Lambda)} \to\mathcal{M}_\LL$, the universal objects over
 $\mathcal{M}_\LL$ and $\mathcal{Y}_{(\LL_0,\Lambda)}$  are
 related\footnote{There is a mild abuse of notation: we are using $A_0^\univ$ to denote both
 the universal elliptic curve over $\mathcal{M}_\LL$,  and the universal elliptic curve
 over $\mathcal{Y}_{(\LL_0,\Lambda)}$.}   by
 \[
 (A^\univ_0,A^\univ)_{/  \mathcal{Y}_{(\LL_0,\Lambda)}}
 \iso (A^\univ_0, A^\univ_1 \times B^\univ),
 \]
 and the isomorphism
 \[
 \Lie(A^\univ)|_{ \mathcal{Y}_{(\LL_0,\Lambda) } }
  \iso \Lie(A_1^\univ\times B^\univ)
 \]
 identifies
 $
 \mathcal{F}^\univ|_{ \mathcal{Y}_{(\LL_0,\Lambda) } } \iso \Lie(B^\univ).
 $
 In particular there is a canonical  isomorphism
\[
 \mathcal{T}_\LL |_{ \mathcal{Y}_{(\LL_0,\Lambda) } }
 \iso \Hom(\mathrm{Fil}(A^\univ_0) , \Lie(A^\univ_1) ).
\]

For any elliptic curve $A_0\to \Spec(R)$ over a ring, the short exact sequence
\[
0 \to \mathrm{Fil}(A_0) \to H_1^{dR}(A_0) \to \Lie(A_0) \to 0
\]
of $R$-modules is dual to
\[
0 \to H^0(A_0,\Omega^1_{A_0/R}) \to H^1_{dR}(A_0) \to H^1(A_0,\co_{A_0})\to 0,
\]
and there is a canonical identification $H^1(A_0,\co_{A_0}) \iso \Lie(A_0^\vee)$.  In particular
there is canonical perfect  pairing $\Lie(A_0^\vee) \otimes_R \mathrm{Fil}(A_0) \to \co_S$,
and identifying $\Lie(A_0) \iso \Lie(A_0^\vee)$ via the unique principal polarization, we obtain
a perfect pairing
\begin{equation}\label{deRham pairing}
\Lie(A_0) \otimes_R \mathrm{Fil}(A_0) \to R.
\end{equation}
Applying this with $A_0=A_0^\univ$ yields the second isomorphism in
\[
\mathcal{T}_\LL |_{ \mathcal{Y}_{(\LL_0,\Lambda) } }
\iso
\Hom(\mathrm{Fil}(A^\univ_0) , \Lie(A^\univ_1) )
  \iso  \Lie(A^\univ_0)  \otimes \Lie(A^\univ_1).
\]

All that remains is to keep track of  the metrics  under this isomorphism.
This is routine, once one knows an explicit formula for the pairing (\ref{deRham pairing})
when  $A_0\in \mathcal{M}_{(1,0)}(\C)$ is the complex elliptic curve with homology
$\mathfrak{A}_0= H_1(A_0(\C),\Z)$, as in the discussion surrounding (\ref{homology}).
Taking $e$ and $\overline{e}$ as in the proof of Proposition
\ref{prop:analytic taut},  the compositions
\[
\mathfrak{A}_{0\R} \to \mathfrak{A}_{0\C} \iso H_1^{dR}(A_0) \map{e}
e H_1^{dR}(A_0) \iso \Lie(A_0)
\]
and
\[
\mathfrak{A}_{0\R} \to \mathfrak{A}_{0\C} \iso H_1^{dR}(A_0) \map{ \overline{ e}}
\overline{ e } H_1^{dR}(A_0) \iso \mathrm{Fil}(A_0)
\]
are $\kk_\R$-linear isomorphisms.  Thus the pairing (\ref{deRham pairing})
corresponds to a pairing
\[
\mathfrak{A}_{0\R} \times \mathfrak{A}_{0\R} \to \C,
\]
which is hermitian with respect to the action of $\kk_\R\iso \C$.
Using  the proof of \cite[Proposition 4.4]{Ho1}, one can show that this pairing is
$
 2 \pi |d_\kk|^{-1/2}  \cdot h_{\mathfrak{A}_0}.
$
The rest of the proof is elementary linear algebra, and is left to the reader.
\end{proof}

\begin{proof}[Proof of Theorem \ref{thm:taut degree}]
Combining  Lemmas \ref{lem:metrized lie} and \ref{lem:metrized pullback} with  (\ref{twist degree})
shows that
\begin{align*}
[ \widehat{\mathcal{T}}_\LL : \mathcal{Y}_{(\LL_0,\Lambda) } ]
& =
\widehat{\deg} \ \widehat{\Lie}(A^\univ_0)  +  \widehat{\deg} \ \widehat{\Lie}(A^\univ_1)
+  \widehat{\deg} \ \widehat{\bm{1}} \left(  16\pi^3 e^\gamma  \right) \\
&= \deg_\C \mathcal{Y}_{(\LL_0, \Lambda)}  \left(  2 \frac{ L'(\chi_\kk,0) }{ L(\chi_\kk , 0)  } +
\log \left|  \frac{   d_\kk  }{ 4 \pi } \right|  -  \gamma \right),
\end{align*}
and comparing with Proposition \ref{prop:eisenstein coeff} completes the proof.
\end{proof}


\section{The intersection formula}
\label{s:main theorem}


Again, fix a pair $(\LL_0,\Lambda)$ as in  (\ref{CM genus}), and set $\LL=\LL_0\oplus\Lambda$
as in Remark  \ref{rem:sum}.  Recall from Section \ref{ss:harmonic divisors}
the finite dimensional $\C$-vector space $S_\LL$ endowed with the Weil
representation $\omega_\LL:\SL_2(\Z) \to \Aut(S_\LL)$, and a commuting
action of a finite group $\Delta$.


\subsection{The main result}


Let $f  \in H_{2-n}(\omega_\LL)$ be a $\Delta$-invariant harmonic Maass form with
holomorphic part
\[
f^+(\tau) = \sum_{ \substack {m \in \Q \\ m\gg -\infty} } c^+(m) \cdot q^m .
\]
Let $c^+(0,0)$ denote the value of $c^+(0)\in S_\LL$ at the trivial coset in $\mathfrak{d}_\kk^{-1} \LL_f/\LL_f$.
Attached to this $f$ we have, from Sections  \ref{ss:compact} and \ref{ss:KR green},  an arithmetic divisor
\[
\widehat{\mathcal{Z}}^\mathrm{total}_\LL( f )
\in \widehat{\mathrm{CH}}^1_\C(\mathcal{M}_\LL^*).
\]

\begin{definition}
The \emph{arithmetic theta lift of $f$} is the class
\[
\widehat{\Theta}_\LL(f)
= \widehat{\mathcal{Z}}^\mathrm{total}_\LL( f ) + c^+(0,0) \cdot \widehat{\mathcal{T}}^*_\LL
\in \widehat{\mathrm{CH}}^1_\C(\mathcal{M}_\LL^*),
\]
where $\widehat{\mathcal{T}}^*_\LL$ is the metrized cotautological bundle (\ref{cotaut def}).
\end{definition}

The main result of this paper is the following formula, which relates an arithmetic intersection
multiplicity to the derivative of an $L$-function.  The proof will occupy the remainder of Section \ref{s:main theorem}.

\begin{theorem}\label{thm:arithmetic degree}
The arithmetic theta lift satisfies
\[
[ \widehat{\Theta}_\LL( f )  :  \mathcal{Y}_{(\LL_0,\Lambda)} ]
=   -   \deg_\C \mathcal{Y}_{(\LL_0, \Lambda)}  \cdot L'( \xi(f)  , \theta_\Lambda ,0),
\]
where
$
\xi : H_{2-n}(\omega_\LL) \to S_n(\overline{\omega}_\LL)
$
is the complex-conjugate-linear homomorphism of (\ref{defxi}), and the $L$-function on the right  is (\ref{eq:vectorRankin-Selberg}).
\end{theorem}


\subsection{A special case}


 We will prove  Theorem \ref{thm:arithmetic degree} by first verifying it for the forms $f_{m,\mathfrak{r}}$ of
Lemma \ref{lem:f-mr}, in which case the claim is that
\begin{equation}\label{main prelim}
 [ \widehat{\mathcal{Z}}_\LL( f_{m,\mathfrak{r}} )  : \mathcal{Y}_{(\LL_0,\Lambda)} ]
+   c^+_{m,\mathfrak{r}}(0,0)  \cdot   [ \widehat{ \mathcal{T} }_\LL : \mathcal{Y}_{(\LL_0 , \Lambda )} ]
=   -  \deg_\C \mathcal{Y}_{(\LL_0, \Lambda)} \cdot
L' \big( \xi(f_{m,\mathfrak{r}} ) ,\theta_\Lambda , 0 \big).
 \end{equation}
We have enough information now to prove this equality under some restrictive hypotheses.

Fix   $m_1,m_2\in \Q_{\ge 0}$. In Section \ref{ss:zero cycles} we defined an $\co_\kk$-stack
$\mathcal{X}_{(\LL_0,\Lambda)} (m_1, m_2 , \mathfrak{r} )$ equipped with a
finite, unramified, and representable morphism
\[
\mathcal{X}_{(\LL_0,\Lambda)} (m_1, m_2 , \mathfrak{r} ) \to \mathcal{Y}_{(\LL_0,\Lambda)}.
\]
If $m_1>0$, then $\mathcal{X}_{(\LL_0,\Lambda)} (m_1, m_2 , \mathfrak{r} )$ has dimension
zero, and defines a divisor on $\mathcal{Y}_{(\LL_0,\Lambda)}$,
necessarily supported in nonzero characteristic.
By endowing this divisor with the  trivial Green function, we obtain an arithmetic divisor
\begin{equation}\label{arithmetic zero cycles}
\widehat{ \mathcal{X} }_{(\LL_0,\Lambda)} (m_1, m_2 , \mathfrak{r} )
\in \widehat{\mathrm{Div} }  (\mathcal{Y}_{(\LL_0,\Lambda)}).
\end{equation}

 \begin{proposition}\label{prop:final zero cycle}
Suppose  $m=m_1+m_2$ with $m_1\in \Q_{>0}$ and $m_2\in \Q_{\ge 0}$.
The arithmetic divisor (\ref{arithmetic zero cycles}) satisfies
 \begin{align*}
\widehat{\deg}\   \widehat{\mathcal{X}}_{(\LL_0,\Lambda)} (m_1, m_2 , \mathfrak{r} )
&  =
-  \deg_\C \mathcal{Y}_{(\LL_0, \Lambda)} \cdot
a^+_{\LL_0}(m_1,\mathfrak{r}) \cdot   R_\Lambda(m_2,\mathfrak{r}) .
 \end{align*}
 \end{proposition}

\begin{proof}
This follows  by comparing Theorem \ref{thm:zero cycles} with  Proposition \ref{prop:eisenstein coeff}.
Abbreviate
\[
\mathcal{X} = \mathcal{X}_{(\LL_0,\Lambda)} (m_1, m_2 , \mathfrak{r} ).
\]
If $|\mathrm{Diff}_{\LL_0}(m_1)| >1$ then both sides of the desired equality are $0$,
so assume $\mathrm{Diff}_{\LL_0}(m_1) = \{p\}$ for a prime $p$, necessarily nonsplit in $\kk$,
and let $\mathfrak{p}$ be the prime of $\kk$ above $p$.   Theorem  \ref{thm:zero cycles} implies
\begin{align*}\nonumber
&\widehat{\deg}\  \widehat{ \mathcal{X}}_{(\LL_0,\Lambda)} (m_1, m_2 , \mathfrak{r} )
=
\log( \mathrm{N}(\mathfrak{p}))
\sum_{x\in \mathcal{X}(\F_\mathfrak{p}^\alg) } \frac{ \mathrm{length}(\co_{\mathcal{X},x}) }{ |\Aut(x) | } \\
&=
\frac{h_\kk   \log(p)  }{w_\kk |\Aut(\Lambda) | } \cdot  \ord_p(p m_1 )   \cdot
R_\Lambda(m_2, \mathfrak{r}) \cdot \rho\left(  \frac{m_1\mathrm{N}(\mathfrak{s})}{p^\epsilon} \right) ,
\end{align*}
where $\mathfrak{s}$ is the prime-to-$\mathfrak{p}$ part of $\mathfrak{r}$, and $\epsilon$ is defined
by (\ref{epsilon}).
On the other hand,  Proposition \ref{prop:eisenstein coeff} tells us that
\[
a^+_{\LL_0}(m_1, \mathfrak{r} )
=
- \frac{ w_\kk \log(p) }{2 h_\kk}  \cdot \ord_p(p m_1)
\cdot  \rho \left( \frac{m_1  |d_\kk| }{p^\epsilon}  \right)
\sum_{ \substack{  \mu\in \mathfrak{d}_\kk^{-1} \LL_{0,f} / \LL_{0,f} \\ Q(\mu) =m_1 }  } 2^{s(\mu)}
\varphi_\mathfrak{r}(\mu),
\]
where $\varphi_\mathfrak{r}$ is the characteristic function of
$\mathfrak{r}^{-1} \LL_{0,f} / \LL_{0,f} \subset \mathfrak{d}_\kk^{-1} \LL_{0,f} / \LL_{0,f}$.

The proposition follows from the above equalities and Remark \ref{rem:degree},  once we prove
\[
\rho \left( \frac{m_1  |d_\kk| }{p^\epsilon}  \right)
\sum_{ \substack{  \mu\in \mathfrak{d}_\kk^{-1} \LL_{0,f} / \LL_{0,f} \\ Q(\mu) =m_1 }  } 2^{s(\mu)}
\varphi_\mathfrak{r}(\mu)
=  2^{o(d_\kk) } \rho\left(  \frac{m_1\mathrm{N}(\mathfrak{s})}{p^\epsilon} \right)  .
\]
Both sides factor as a product of local terms, and the equality of local terms for primes not dividing $d_\kk$
is obvious.  It therefore suffices to prove, for every prime $\ell\mid d_\kk$, the relation
\begin{equation}\label{local combo}
\rho_\ell  (m_1  |d_\kk| )
\sum_{ \substack{  \mu\in \mathfrak{d}_\kk^{-1} \LL_{0,\ell} / \LL_{0,\ell} \\ Q(\mu) =m_1 }  }
2^{s_\ell(\mu)}  \varphi_{\mathfrak{r},\ell}(\mu)
=  2 \rho_\ell(  m_1\mathrm{N}(\mathfrak{s}) ),
\end{equation}
where $\rho_\ell(k)$ is the number of ideals in $\co_{\kk,\ell}$ of norm $k\Z_\ell$,
\[
s_\ell(\mu) = \begin{cases}
1 &\mbox{if } \mu=0 \\
0&\mbox{otherwise,}
\end{cases}
\]
and $\varphi_{\mathfrak{r},\ell}$ is the characteristic function of
$\mathfrak{r}^{-1} \LL_{0,\ell} / \LL_{0,\ell} \subset
\mathfrak{d}_\kk^{-1} \LL_{0,\ell} / \LL_{0,\ell}$.
\begin{enumerate}
\item[Case 1:]
If $\ord_\ell(m_1) \ge 0$,  then only the term $\mu=0$ contributes to the left hand side
of (\ref{local combo}), both sides of the equality are equal to $2$, and we are done.

\item[Case 2:] If $\ord_\ell(m_1)<-1$,  then $\rho_\ell(m_1 |d_\kk|)= \rho_\ell(m_1 \mathrm{N}(\mathfrak{s}))=0$,
and we are done.

\item[Case 3:] If $\ord_\ell(m_1)=-1$ and $\ell\nmid \mathrm{N}(\mathfrak{r})$, then
$\rho_\ell(m_1 \mathrm{N}(\mathfrak{s}))=0$.  On the left hand side of (\ref{local combo}),
the assumption $\ord_\ell(m_1)=-1$ implies that any $\mu$ appearing in the sum must be nonzero,
and hence $\varphi_{\mathfrak{r},\ell}(\mu) =0$.  Thus in this case both sides of (\ref{local combo}) vanish.

\item[Case 4:] If $\ord_\ell(m_1)=-1$ and $\ell\mid  \mathrm{N}(\mathfrak{s})$, then
$\rho_\ell(m_1 |d_\kk|) = \rho_\ell(m_1 \mathrm{N}(\mathfrak{s})) =1$.  Let $\mathfrak{l}\subset \co_\kk$
be the prime determined by $\ell\co_\kk=\mathfrak{l}^2$.
As  $\ell \not\in \mathrm{Diff}_{\LL_0}(m_1)$, the rank one
$\kk_\mathfrak{l}$-hermitian space $\LL_{0,\ell}\otimes_{\Z_\ell} \Q_\ell$ represents $m_1$.
 It follows from the self-duality of
$\LL_{0,\ell}$  that $\mathfrak{d}_\kk^{-1} \LL_{0,\ell}$ represents $m_1$, and
from this it is easy to see that the  rank one $\co_\kk/\mathfrak{l}$-quadratic space
$\mathfrak{d}_\kk^{-1} \LL_{0,\ell} / \LL_{0,\ell}$
has two distinct nonzero solutions to $Q(\mu)=m_1$.  Thus
\[
\sum_{ \substack{  \mu\in \mathfrak{d}_\kk^{-1} \LL_{0,\ell} / \LL_{0,\ell} \\ Q(\mu) =m_1 }  }
2^{s_\ell(\mu)}  =2
\]
and again (\ref{local combo}) holds.

\item[Case 5:]  If   $\ord_\ell(m_1)=-1$ and  $\ell=p$, then
 $\rho_\ell(m_1\mathrm{N}(\mathfrak{s}))=0$.  On the left hand side
of (\ref{local combo}), the sum over $\mu$ is
empty: any $\mu \in  \mathfrak{d}_\kk^{-1} \LL_{0,p} / \LL_{0,p}$ representing
$m_1\in \Q_p/\Z_p$ could be lifted to  $\mu \in  \mathfrak{d}_\kk^{-1} \LL_{0,p}$ representing
$m_1\in \Q_p$,  contradicting $p\in \mathrm{Diff}_{\LL_0}(m_1)$.  Thus both sides of
(\ref{local combo}) vanish.
\end{enumerate}
This exhausts all cases, and completes the proof.
\end{proof}

We can now prove  (\ref{main prelim})  under the simplifying
hypothesis  $R_\Lambda(m,\mathfrak{r})=0$.  Recall from Remark \ref{rem:decomposition} that
under this hypothesis
\[
\mathcal{Z}_\LL(m,\mathfrak{r}) \cap \mathcal{Y}_{ (\LL_0,\Lambda )}
\iso
\bigsqcup_{   \substack{  m_1 \in \Q_{>0} \\ m_2 \in \Q_{\ge 0} \\ m_1+m_2 =m  }   }
 \mathcal{X}_{ ( \LL_0,\Lambda ) } (m_1, m_2 , \mathfrak{r} ) ,
 \]
 and each stack appearing on the right has dimension zero.  Proposition
 \ref{prop:final zero cycle} shows that
\begin{align*}
\frac{1}{ \deg_\C \mathcal{Y}_{(\LL_0, \Lambda)}  }  \cdot
 I  ( \mathcal{Z}_\LL( m,\mathfrak{r} )  : \mathcal{Y}_{(\LL_0,\Lambda)} )
 & =
 \frac{1}{ \deg_\C \mathcal{Y}_{(\LL_0, \Lambda)}  }
  \sum_{   \substack{  m_1 \in \Q_{>0} \\ m_2 \in \Q_{\ge 0} \\ m_1+m_2 =m  }   }
  \widehat{\deg}\   \widehat{ \mathcal{X}}_{(\LL_0,\Lambda)} (m_1, m_2 , \mathfrak{r} )  \\
 & =  -
 \sum_{   \substack{  m_1 \in \Q_{>0}   \\ m_2 \in \Q_{ \ge 0} \\  m_1+m_2 =m  }   }
 a^+_{\LL_0}(m_1,\mathfrak{r})  \cdot  R_\Lambda(m_2,\mathfrak{r})  ,
\end{align*}
while Corollary \ref{cor:cm values} shows that
\begin{align*}
\frac{1}{\deg_\C \mathcal{Y}_{(\LL_0, \Lambda)} } \cdot
 \Phi_\LL(\mathcal{Y}_{(\LL_0 , \Lambda)},  f_{m,\mathfrak{r}} )
 &=  -  L' \big( \xi(f_{m,\mathfrak{r}} ) ,\theta_\Lambda , 0 \big)\\
 &\phantom{=}{}+  c^+_{m,\mathfrak{r}} (0,0)  \cdot
 a^+_{\LL_0}(0,\mathfrak{r})  \cdot R_\Lambda(0,\mathfrak{r})   \\
& \phantom{=}{}+   \sum_{   \substack{  m_1 \in \Q_{>0}  \\ m_2 \in \Q_{ \ge 0} \\  m_1+m_2 =m  }   }
 a^+_{\LL_0}(m_1, \mathfrak{r} ) \cdot R_\Lambda(m_2, \mathfrak{r})  .
\end{align*}
Adding these together gives
\[
 \frac{1}{ \deg_\C \mathcal{Y}_{(\LL_0, \Lambda)}  } \cdot
   [ \widehat{\mathcal{Z}}_\LL( f_{m,\mathfrak{r}} )  : \mathcal{Y}_{(\LL_0,\Lambda)} ]
  =     c^+_{m,\mathfrak{r}} (0,0) \cdot  a^+_{\LL_0}(0,\mathfrak{r}) \cdot R_\Lambda(0,\mathfrak{r})
    -   L' \big( \xi(f_{m,\mathfrak{r}} ),\theta_\Lambda , 0\big) ,
\]
and  an application of Theorem \ref{thm:taut degree}  completes the proof.
Proving  (\ref{main prelim}) in general requires treating improper intersections,
and requires a bit more work.


\subsection{The adjunction formula}
\label{ss:adjunction}


The proof of  (\ref{main prelim}) in full generality revolves around the study of a canonical section
\begin{equation}\label{adjunction section}
\bm{\sigma}_{m,\mathfrak{r}} \in \Gamma\big(
\mathcal{Z}^\heartsuit_{\LL}(m,\mathfrak{r}) |_{ \mathcal{Y}_{(\LL_0,\Lambda)}   }   \big)
\end{equation}
of the line bundle
\begin{equation}\label{magic bundle}
\mathcal{Z}^\heartsuit_{\LL}(m,\mathfrak{r}) =
\mathcal{Z}_{\LL}(m,\mathfrak{r}) \otimes \mathcal{T}_\LL^{\otimes-R_\Lambda(m,\mathfrak{r})}
\end{equation}
restricted to $\mathcal{Y}_{(\LL_0,\Lambda)}$.
This subsection is devoted to the construction of (\ref{adjunction section}),
and the calculation of its divisor, which the reader may find in Proposition
\ref{magic section divisor}  below.

Because of the minor nuisance that the natural maps $\mathcal{Y}\to \mathcal{M}$ and
$\mathcal{Z}(m,\mathfrak{r}) \to \mathcal{M}$ are not  closed immersions, the section
(\ref{adjunction section}) will be  constructed by patching together sections on an \'etale open
cover.  Accordingly, we define a \emph{sufficiently small \'etale open subscheme} of $\mathcal{M}_\LL$
to be a scheme $U$ together with an \'etale morphism $U \to \mathcal{M}_\LL$
such that
\begin{enumerate}
\item
on each connected component  $Z\subset \mathcal{Z}_{ \LL} (m,\mathfrak{r})_{/U}$
the natural map $Z \to U$ is a closed immersion,
\item
on each connected component  $Y \subset \mathcal{Y}_{(\LL_0,\Lambda) /U}$
the natural map $Y \to U$ is a closed immersion, and  the universal object $(A_0,A_1,B)$
over $Y$ satisfies  $L(A_0,B) \iso \Lambda$.
\end{enumerate}
As in the discussion following Definition \ref{def:KR divisors},  the stack $\mathcal{M}_\LL$
admits a finite cover by sufficiently small \'etale open subschemes.

Fix a sufficiently small \'etale open subscheme $U \to \mathcal{M}_\LL$,  and
connected components
 \begin{align}\label{etale components}
 Z & \subset \mathcal{Z}_\LL(m,\mathfrak{r})_{/U} \\
 Y & \subset \mathcal{Y}_{(\LL_0,\Lambda) /U } \nonumber.
 \end{align}
  The smoothness of $\mathcal{Y}_{(\LL_0 , \Lambda)}$ over $\co_\kk$, together
with our hypotheses on $U$, imply that  $Y$ is a reduced and irreducible one-dimensional closed
subscheme of $U$.  The closed subscheme $Z\subset U$
is  perhaps neither reduced nor irreducible, but an easy deformation theory argument
shows that the generic fiber of $\mathcal{Z}_\LL(m,\mathfrak{r})$ is smooth,
and hence $Z_{/\kk}$ is a smooth variety  of dimension $n-1$.  The intersection
$Z\cap Y=Z\times_U Y$ is a closed subscheme of $Y$, and hence is either all of $Y$ or is
of dimension $0$.

\begin{definition}
Given connected components (\ref{etale components}), we say that
\begin{enumerate}
\item
$Z$ is  \emph{$Y$-proper} if  $Z \cap Y$ has dimension $0$,
\item
$Z$ is \emph{$Y$-improper} if  $Z\cap Y =Y$.
\end{enumerate}
\end{definition}

\begin{proposition}\label{prop:improper}
The number of $Y$-improper components
$Z\subset \mathcal{Z}_\LL(m,\mathfrak{r})_{/U}$ is $R_\Lambda(m,\mathfrak{r})$.
\end{proposition}

\begin{proof}
Let $\eta  \in  Y$ be the  generic point, so that $k(\eta)$ is a finite extension of
$\kk$, and let $\overline{\eta} \to \eta$ be the geometric generic point above $\eta$.
Denote by  $(A_{0,\eta},A_{1,\eta}, B_\eta)$
and  $(A_{0, \overline{\eta} },A_{1,\overline{\eta} }, B_{\overline{\eta} })$  the pullbacks to  $\eta$
and $\overline{\eta}$ of the  universal object over $\mathcal{Y}_{(\LL_0,\Lambda)}$.   The fiber
\[
\mathcal{Z}_\LL(m,\mathfrak{r})_{\overline{\eta}}  = \mathcal{Z}_\LL(m,\mathfrak{r})_{/U} \times_U \overline{\eta}
\]
is a disjoint union of copies of
$\overline{\eta}$, one for every
\[
\lambda \in \mathfrak{r}^{-1} L( A_{0,\overline{\eta}}  , A_{1, \overline{\eta}}  \times B_{ \overline{\eta} })
\]
satisfying $\langle \lambda , \lambda \rangle =m$.  Under the decomposition
\ref{ortho sum},  any such  $\lambda$ takes the form  $\lambda=\lambda_1+\lambda_2$.
The map $\lambda_1:A_{0, \overline{\eta}} \to A_{1, \overline{\eta}}$ must vanish because of signature
considerations, and so  all such $\lambda$ lie in
\[
\mathfrak{r}^{-1} L( A_{0, \overline{\eta}}  ,  B_{ \overline{ \eta}  }) \iso \mathfrak{r}^{-1} \Lambda.
\]
It follows that $\mathcal{Z}_\LL(m,\mathfrak{r})_{\overline{\eta}}$
is a disjoint union of $R_\Lambda(m,\mathfrak{r})$ copies of $\overline{\eta}$.  Moreover,
 our definition of a sufficiently small \'etale open guarantees that
$\mathfrak{r}^{-1} L( A_{0, \eta }  ,  B_\eta) \iso \mathfrak{r}^{-1} \Lambda$,
and so all such $\lambda$ are already defined over $\eta$.  In other words,
$\mathcal{Z}_\LL(m,\mathfrak{r})_{ \eta }$
is a disjoint union of $R_\Lambda(m,\mathfrak{r})$ copies of $\eta$, and the claim follows easily.
\end{proof}

As in the proof of Lemma  \ref{lem:local ring}, write $\co_\kk = \Z[\Pi]$ and define elements of $\co_\kk\otimes_\Z \co_U$ by
\[
J = \Pi \otimes 1 - 1\otimes \Pi , \qquad
\overline{J} = \overline{\Pi} \otimes 1 - 1\otimes \Pi .
\]
An elementary calculation shows that the sequence
\[
\cdots \map{\overline{J}}   \co_\kk\otimes_\Z \co_U \map{J}
 \co_\kk\otimes_\Z \co_U \map{\overline{J}} \co_\kk\otimes_\Z \co_U \map{J} \cdots.
\]
is exact.

\begin{lemma}\label{lem:J div}
Every geometric point  $y\to Y$ admits an  affine \'etale neighborhood
$\Spec(R) \to U$ with the following property:  letting $A$ denote the pullback to $R$ of the universal object via
$U\to \mathcal{M}\to\mathcal{M}_{(n-1,1)}$,  the $R$-module $\Lie(A)$ is free, and admits a basis
$\epsilon_1,\ldots,\epsilon_n$ such that
\begin{enumerate}
\item
 $\epsilon_1,\ldots,\epsilon_{n-1}$ is a basis for the universal
subsheaf $\mathcal{F} \subset \Lie(A)$,
\item
 the operator  $J\in \co_\kk \otimes_\Z R$ on $\Lie(A)$ has the form
\begin{equation}\label{j form}
J = \left[\begin{matrix}
0 & \cdots & 0 & j_1 \\
\vdots & \ddots&  \vdots& \vdots \\
0 & \cdots & 0 & j_n
\end{matrix}\right] \in M_n(R)
\end{equation}
for some $j_1,\ldots, j_n \in R$  satisfying $(j_1, \ldots, j_n) = (j_n) = \mathfrak{d}_\kk R$,
 \item
there is unique  $J_0 \in R^\times$ such that $J=\delta_\kk \circ J_0$ as endomorphisms
of $ \Lie(A)/\mathcal{F}$, where  $\delta_\kk \in \co_\kk$ is any generator of  $\mathfrak{d}_\kk$.
\end{enumerate}
\end{lemma}

\begin{proof}
Certainly there is an affine \'etale neighborhood over which
\[
0 \to \mathrm{Fil}(A) \to H_1^{dR}(A) \to \Lie(A) \to 0
\]
is an exact sequence of free $R$-modules.  If $\epsilon_1\ldots,\epsilon_n$
is any basis of $H_1^{dR}(A)$ such that $\epsilon_1\ldots,\epsilon_{n-1}$
generates $\mathcal{F}$, then the matrix of $J$ is (\ref{j form})
for some $j_1,\ldots, j_n \in R$, simply because $J\mathcal{F}=0$.
Futhermore, $J$ acts on the quotient $\Lie(A)/\mathcal{F}$
through the complex conjugate of the structure map $\co_\kk\to R$,
and so $j_n=\overline{\Pi}-\Pi$.  It is easy to check that $(\overline{\Pi}-\Pi)\co_\kk=\mathfrak{d}_\kk$,
and so   $(j_n)=\mathfrak{d}_\kk R$.

  To prove that $(j_1,\ldots, j_n) =\mathfrak{d}_\kk R$,
after possibly shrinking the \'etale neighborhood  $\Spec(R)$,
it suffices to prove this equality after replacing  $R$ by the completion of the \'etale local ring
at $y$.  The proof of  \cite[Proposition 3.2.3]{Ho3} shows that the ideal $(j_1,\ldots, j_n)$ is principal.
Everything we have said so far holds for any geometric point of $U$.  Now we exploit the
hypothesis that $y$ is a geometric point of $Y$.
Let $I\subset R$ be the ideal
defining the closed subscheme $Y \times_U \Spec(R) \subset \Spec(R)$,
and let $A'$ be the reduction of $A$ to $R/I$. By definition of the morphism $\mathcal{Y}\to \mathcal{M}$,
the abelian scheme $A'$  comes with a decomposition
$A' \iso A_1\times B$, and  the subsheaf $\mathcal{F}' \subset \Lie(A')$ is $\mathcal{F}'=\Lie(B)$.
In particular, $\mathcal{F}'$ admits the $\co_\kk$-stable, and hence $J$-stable, $\co_Y$-direct summand
$\Lie(A_0)$. Thus there is \emph{some} basis of $\Lie(A')$ with respect to which
\[
J = \left[\begin{matrix}
0 & \cdots & 0 & 0 \\
\vdots & \ddots&  \vdots& \vdots \\
0 & \cdots & 0 & 0\\
0 & \cdots & 0 & j_n
\end{matrix}\right] \in M_n(R/I).
\]
As the ideal of $R/I$ generated by the entries of $J$ is independent of the choice of basis,
we see that
$
(j_1,\ldots, j_n)=(j_n)=(\delta_\kk)
$
in $R/I$.  Now  pick a generator $\gamma \in R$ of the principal ideal $(j_1,\ldots, j_n)$.
  We have  shown that $(\delta_\kk) \subset (\gamma)$,  with equality after reducing modulo $I$.
  Furthermore, $R/I$ is an integral domain of characteristic $0$, as $Y$ is reduced, irreducible, and
  flat over $\co_\kk$.  It follows that if we  write
$\delta_\kk = u \gamma$ with $u\in R$, then $u$ is a unit in $R/I$,  and hence
is also a unit in the local ring $R$.  This shows that $(j_1,\ldots, j_n)=(\delta_\kk)$ in $R$.

For the existence and uniqueness of the unit $J_0\in R^\times$ note  that $\mathcal{M}$, hence also $R$, is flat over $\co_\kk$,
and so $\delta_\kk \in R$ is not a zero divisor.   Thus $j_n$ is uniquely divisible by the image of $\delta_\kk$
under $\co_\kk \to \End_R(\Lie(A)/\mathcal{F})) =R$. Dividing $j_n$ by this image defines the desired unit $J_0$.
\end{proof}

Keeping $Z$ and $Y$ as in (\ref{etale components}), denote by $\mathcal{I}_Z \subset \co_U$ the
ideal sheaf defining the closed subscheme $Z \hookrightarrow U$, and by
$\co(Z)= \mathcal{I}_Z^{-1}$ the line bundle on $U$ determined
by the divisor $Z$.   Let $\mathcal{I}_Y \subset \co_U$ be the ideal sheaf defining the closed
subscheme $Y\hookrightarrow U$.
The \emph{first order infinitesimal neighborhood} of $Y$ is the closed
subscheme  $\widetilde{Y} \hookrightarrow U$  defined by the ideal  sheaf $\mathcal{I}_Y^2 \subset \co_U$.
The picture is
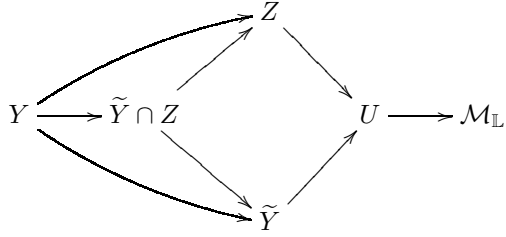
\[
\xymatrix{
& &  {Z} \ar[dr] \\
{Y} \ar@/^/[urr]\ar@/_/[drr] \ar[r]&  {\widetilde{Y} \cap Z}  \ar[ur] \ar[dr] & & U \ar[r]&\mathcal{M}_\LL\\
&  &  {\widetilde{Y}} \ar[ur]
}
\]
where $\widetilde{Y} \cap Z =\widetilde{Y} \times_U Z$.
Let $(A_0, A , \lambda )$ be the pullback to $Y$  of the universal object over
$\mathcal{Z}_\LL(m,\mathfrak{r})$.
 Of course the pair $( A_0 , A)$ has a canonical extension
$( \widetilde{A}_0, \widetilde{A})$ to $\widetilde{Y}$, obtained by pulling back the universal
pair over $\mathcal{M}$ via $\widetilde{Y} \to U \to \mathcal{M}$, but there is no such
canonical extension of $\lambda$ to $\widetilde{Y}$.  Indeed, $\widetilde{Y}\cap Z$ is the
maximal closed subscheme of $\widetilde{Y}$ over which $\lambda$ extends to an element of
$\mathfrak{r}^{-1}L(\widetilde{A}_0,\widetilde{A})$ satisfying the vanishing condition of
(\ref{extra vanishing}).

Recall that, by virtue of the moduli problem defining $\mathcal{M}_{(n-1,1)}$, the $\co_{\widetilde{Y}}$-module
$\Lie(\widetilde{A})$ comes equipped with a  corank one submodule $\widetilde{\mathcal{F}}$.
We will now construct a canonical $\co_{\widetilde{Y}}$-module map
\[
\bm{obst}(\lambda) : \mathrm{Fil}(\widetilde{A}_0) \to \Lie(\widetilde{A} ) /  \widetilde{\mathcal{F}},
\]
the \emph{obstruction to deforming $\lambda$}, whose zero locus subscheme  is $\widetilde{Y} \cap Z$.
The scheme $U$  may be covered by open subschemes $\{U_i\}$ with the property that on each $U_i$
either $\mathrm{N}(\mathfrak{r}) \in \co_{U_i}^\times$  or   $\mathfrak{r}\co_{U_i}=\mathfrak{d}_\kk \co_{U_i}$.
In the construction of $\bm{obst}(\lambda)$  we are free to assume that $U$ itself satisfies one of these two properties.

First assume  $\mathrm{N}(\mathfrak{r}) \in \co_U^\times$.  Under this hypothesis, $\lambda$ determines an $\co_\kk$-linear map
$\lambda : H_1^{dR}(A_0) \to H_1^{dR}(A)$ of $\co_Y$-modules, which,  by the deformation theory
of \cite[Proposition 2.1.6.4]{Lan} extends canonically to an $\co_\kk$-linear map
$
\widetilde{ \lambda } : H_1^{dR}( \widetilde{A}_0) \to H_1^{dR}(\widetilde{A})
$
of $\co_{\widetilde{Y}}$-modules. Define $\bm{obst}(\lambda)$ as the composition
\begin{equation}\label{first obst}
\mathrm{Fil}( \widetilde{A}_0) \to
H_1^{dR}( \widetilde{A}_0) \map{\widetilde{\lambda}} H_1^{dR}(\widetilde{A} )
\to   \Lie(\widetilde{A} ) / \widetilde{\mathcal{F}}.
\end{equation}

Now assume $\mathfrak{r} \co_U = \mathfrak{d}_\kk \co_U$.  As above, by deformation theory the
map $\delta_\kk \lambda : A_0 \to A$ induces a map
\[
\widetilde{ \delta_\kk \lambda} : H_1^{dR}(\widetilde{A}_0) \to H_1^{dR}(\widetilde{A}).
\]
Once  again using $\mathrm{Fil}(\widetilde{A}_0) = J H_1^{dR}(\widetilde{A}_0)$, as in the proof of \cite[Proposition 2.1.2]{Ho2},
define $\bm{obst}(\lambda)$ as the composition
\begin{equation}\label{second obst}
\mathrm{Fil}(\widetilde{A}_0)  = J H_1^{dR}( \widetilde{A}_0) \map{ J s \mapsto \widetilde{\delta_\kk \lambda} (s)} H_1^{dR}(\widetilde{A} )
\to   \Lie(\widetilde{A} ) / \widetilde{\mathcal{F}} \map{J_0}  \Lie(\widetilde{A} ) / \widetilde{\mathcal{F}},
\end{equation}
where $J_0 \in \co_U^\times$ is as in Lemma \ref{lem:J div}.

\begin{remark}
To see that (\ref{second obst}) is well-defined, suppose  $Js_1=Js_2$.  This implies that
$s_1 -s_2 \in  \overline{J} H_1^{dR}(\widetilde{A}_0)$.  The signature condition on $\widetilde{\mathcal{F}}$
implies that $\overline{J}$ annihilates $\Lie(\widetilde{A})/\widetilde{\mathcal{F}}$,
and therefore  $\widetilde{ \delta_\kk \lambda}(s_1) = \widetilde{\delta_\kk \lambda}(s_2)$
in    $\Lie(\widetilde{A})/ \widetilde{\mathcal{F}}$.
\end{remark}

\begin{remark}
If  both conditions  $\mathrm{N}(\mathfrak{r}) \in \co_U^\times$ and
$\mathfrak{r} \co_U = \mathfrak{d}_\kk \co_U$ are satisfied  then $\delta_\kk \in \co_U^\times$, and
 the relation $J=\delta_\kk \circ J_0$ guarantees that the compositions
(\ref{first obst}) and (\ref{second obst}) agree.
\end{remark}

\begin{lemma}
The zero locus subscheme of $\bm{obst}(\lambda)$ is $\widetilde{Y} \cap Z$.
\end{lemma}

\begin{proof}
First assume that  $\mathrm{N}(\mathfrak{r}) \in \co_U^\times$.
Using the notation of (\ref{first obst}), denote by $\bm{obst}^*(\lambda)$ the composition
\[
\mathrm{Fil}( \widetilde{A}_0) \to
H_1^{dR}( \widetilde{A}_0) \map{\widetilde{\lambda}} H_1^{dR}(\widetilde{A} )
\to   \Lie(\widetilde{A} ) ,
\]
so that $\bm{obst}(\lambda)$ is the composition
\[
\mathrm{Fil}( \widetilde{A}_0)\map{ \bm{obst}^*(\lambda) } \Lie(\widetilde{A} ) \to \Lie(\widetilde{A} ) / \widetilde{\mathcal{F}}.
\]
By deformation theory, the zero locus subscheme of $\bm{obst}^*(\lambda)$ is the maximal closed subscheme of $\widetilde{Y}$
over which  $\lambda$ extends to an element of $\mathfrak{r}^{-1}L(\widetilde{A}_0,\widetilde{A})$.  For this extension
the vanishing of (\ref{extra vanishing}) is automatic by Remark \ref{rem:mostly vanishing}, and the hermitian norm of the
extension is equal to the hermitian norm of $\lambda$.  It follows that
the zero locus subscheme of $\bm{obst}^*(\lambda)$ is $\widetilde{Y}\cap Z$.  Thus we are reduced proving that
$\bm{obst}^*(\lambda)$ and $\bm{obst}(\lambda)$ have the same zero locus subscheme.

If $d_\kk\in \co_U^\times$ the argument is simple, and exploits the splitting
\[
\co_\kk \otimes_\Z \co_{\widetilde{Y}} \iso \co_{\widetilde{Y}} \times \co_{\widetilde{Y}}.
\]
The orthogonal idempotents on the right hand side induce a splitting
$N=e N \oplus \overline{e}N$ of any $\co_\kk \otimes_\Z \co_{\widetilde{Y}}$-module $N$,
in which $eN$ is the maximal submodule on which $\co_\kk$ acts through the structure
map $\co_\kk\to \co_{\widetilde{Y}}$, and $\overline{e}N$ is the maximal submodule on which $\co_\kk$ acts
through the complex conjugate.  Kr\"amer's signature condition on $\widetilde{\mathcal{F}}$  implies that
$\widetilde{\mathcal{F}} = e \Lie(\widetilde{A})$, and so $\widetilde{\mathcal{F}}$ admits a canonical
complementary  summand $\overline{e} \Lie(\widetilde{A})$
on which $\co_\kk$ acts through the complex conjugate of the structure morphism.
The image of  $\bm{obst}^*(\lambda)$ is contained in
$\overline{e} \Lie(\widetilde{A})\iso \Lie(\widetilde{A})/\widetilde{\mathcal{F}}$.
Thus $\bm{obst}^*(\lambda)$ vanishes if and only if $\bm{obst}(\lambda)$ vanishes, as desired.

Returning to the general case  (but still assuming $\mathrm{N}(\mathfrak{r}) \in \co_U^\times$),
fix a geometric point $y\in \widetilde{Y}(\F)$  and let $\bm{R}$ be the  \'etale local ring of $\widetilde{Y}$ at $y$.
Denote by $(\bm{A}_0 , \bm{A})$ the pullback of $(\widetilde{A}_0 , \widetilde{A} )$ through
$\Spec(\bm{R}) \to V \to \widetilde{Y}$, and similarly denote by
\[
\widetilde{\bm{\lambda}} : H_1^{dR}(\bm{A}_0 ) \to H_1^{dR}(\bm{A})
\]
the pullback of $\widetilde{\lambda}   : H_1^{dR}(\widetilde{A}_0 ) \to H_1^{dR}(\widetilde{A})$.
Using Lemma \ref{lem:J div}, there is an $\bm{R}$-basis  $\epsilon_1,\ldots,\epsilon_n$
of $\Lie(\bm{A})$ such that $\epsilon_1,\ldots,\epsilon_{n-1}$ generates
the corank one $\bm{R}$-submodule $\mathcal{F}_{\bm{A}} \subset \Lie(\bm{A})$,
and  $J$ acts on $\Lie(\bm{A})$ as
\[
J = \left[\begin{matrix}
0 & \cdots & 0 & j_1 \\
\vdots & \ddots&  \vdots& \vdots \\
0 & \cdots & 0 & j_n
\end{matrix}\right] \in M_n(\bm{R}),
\]
where $j_1,\ldots, j_n\in \bm{R}$ satisfy $(j_1,\ldots, j_n)=(j_n)$.
Fix an $\co_\kk\otimes_\Z\bm{R}$-module generator $\sigma\in H_1^{dR}(\bm{A}_0)$, and write
\[
 \widetilde{\bm{\lambda}}(\sigma) = \lambda_1\epsilon_1+\cdots+\lambda_n\epsilon_n \in \Lie(\bm{A})
\]
with $\lambda_i\in\bm{R}$.  Using
\[
\mathrm{Fil}(\bm{A}_0) = J  H_1^{dR}(\bm{A}_0) =J\sigma,
\]
 the image of $\bm{obst}(\lambda)|_{\bm{R}}$ is generated by
\[
\widetilde{\bm{\lambda}}( J \sigma)=\lambda_n j_n \epsilon_n \in \Lie(\bm{A}) / \mathcal{F}_{\bm{A}},
\]
while the image of $\bm{obst}^*(\lambda)|_{\bm{R}}$ is generated by
\[
 \widetilde{\bm{\lambda}}( J \sigma) = J \cdot  \widetilde{\bm{\lambda}}(\sigma)
= \lambda_n( j_1 \epsilon_1+ \cdots + j_n\epsilon_n) \in \Lie(\bm{A}).
\]
As $\lambda_n (j_1,\ldots,j_n)  =   \lambda_n ( j_n)$, the maximal quotient of $\bm{R}$ in which
$\bm{obst}(\lambda)|_{\bm{R}}$ vanishes is the same as the maximal quotient in which
$\bm{obst}^*(\lambda)|_{\bm{R}}$ vanishes.

Now assume that $\mathfrak{r} \co_U = \mathfrak{d}_\kk \co_U$.  The zero locus subscheme of
$\bm{obst}^*(\lambda)$ is the maximal closed subscheme of $\widetilde{Y}$ over which
$\lambda$ extends to an element of  $\mathfrak{r}^{-1} L(\widetilde{A}_0,\widetilde{A})$ satisfying the vanishing
of (\ref{extra vanishing}).  By deformation theory, this is the same as the maximal closed
subscheme over which the compositions
\[
\mathrm{Fil}(\widetilde{A}_0) \to  H_1^{dR}(\widetilde{A}_0) \map{ \widetilde{\delta_\kk \lambda} } H_1^{dR}(\widetilde{A}) \to  \Lie(\widetilde{A})
\]
and
\[
H_1^{dR}(\widetilde{A}_0)  \map{ \widetilde{\delta_\kk \lambda} } H_1^{dR}(\widetilde{A})  \to
\Lie(\widetilde{A}) \to \Lie(\widetilde{A}) / \widetilde{\mathcal{F}}
\]
 vanish.  Using $J\widetilde{\mathcal{F}}=0$ and
 $ \mathrm{Fil}(\widetilde{A}_0) = J H_1^{dR}(\widetilde{A}_0),$
one  easily checks that the vanishing of the second composition implies the vanishing of the first.
As $J_0\in R^\times$,  the vanishing of the second composition is equivalent to the vanishing of $\bm{obst}(\lambda)$.  Thus
$\bm{obst}(\lambda)$ and $\bm{obst}^*(\lambda)$ have the same zero locus subscheme.
\end{proof}

The following result  is reminiscent of the classical adjunction isomorphism
as in \cite[Lemma 9.1.36]{Liu}, however our result is particular to the moduli space $\mathcal{M}$.
Indeed, it is a statement about the cotautological bundle $\mathcal{T}$, which we have
defined using the moduli interpretation of $\mathcal{M}$.

\begin{theorem}[Adjunction]\label{thm:adjunction}
Assuming that $Z$ is $Y$-improper, there is a canonical isomorphism
\begin{equation}\label{adjunction}
\co(Z)|_Y \iso \mathcal{T}_\LL |_Y
\end{equation}
of line bundles on $Y$.
\end{theorem}

\begin{proof}
View the obstruction
\[
\bm{obst}(\lambda) : \mathrm{Fil}(\widetilde{A}_0) \to \Lie(\widetilde{A} ) /  \widetilde{\mathcal{F}}
\]
as a section $\bm{obst}(\lambda) \in \Gamma(  \mathcal{T}_\LL |_{\widetilde{Y}  })$
with zero locus subscheme $\widetilde{Y} \cap Z$. Under the inclusion $\co_U \subset \co(Z)$ of $\co_U$-modules,
the constant function $1$ on $U$ defines a section $\bm{s}\in \Gamma(\co(Z))$ with
zero locus  subscheme $Z$.  Hence the restriction $\bm{s}|_{\widetilde{Y}} \in \Gamma(\co(Z)|_{\widetilde{Y} })$
also has zero locus subscheme $\widetilde{Y} \cap Z$.

After passing to a Zariski open cover of $U$, we are free to assume that $U=\Spec(R)$ is
affine, and that the line bundles $\co(Z)$ and $\mathcal{T}_\LL |_U$ are trivial.
Fix isomorphisms of $R$-modules $\Gamma(\co(Z))\iso R$ and $\Gamma(\mathcal{T}_\LL |_U) \iso R$,
and let $I\subset R$ be the ideal defining the closed subscheme $Y\subset U$.  Note that
$R/I\iso \co_Y$ is an integral domain of characteristic $0$, and hence $I$ is prime.
Let $f,g\in R/I^2$ be the elements corresponding to the sections $\bm{obst}(\lambda)$ and $\bm{s}|_{\widetilde{Y}}$.
As these sections have the same zero locus subscheme, $(f)=(g)$  and  we may write
$f=vg$ and $g=uf$ for some $u,v \in R/I^2$.  In particular $g\cdot (1-uv)=0$.

We claim that if  $x\in R/I^2$ satisfies $gx=0$, then $x\in I/I^2$.  Suppose not.
The element $g\in R/I^2$ comes with a lift to $R$, defined by $\bm{s}$, and we fix any lift
of $x$ to $R$, necessarily with $x\not\in I$.  In particular $x$ is a unit in the localization $R_I$.  Therefore
$g \in I^2 R_I$ and the natural surjection $R_I \to R_I/(g)$ induces an  isomorphism on tangent spaces.
This is a contradiction, as we know from the smoothness of the generic fibers of $U$ and $Z$ that
 $R_I$ and $R_I/(g)$ are smooth $\kk$-algebras of dimensions $n-1$ and $n-2$, respectively.

If we apply the above to $x= 1-uv$ we see that $1=uv$ in $R/I$.  Therefore  the map
\[
 \Gamma(  \mathcal{T}_\LL|_{\widetilde{Y}} )   \iso R/I^2 \map{u} R/I^2 \iso \Gamma(\co(Z)|_{\widetilde{Y} }),
\]
which takes $\bm{obst}(\lambda) \mapsto \bm{s}|_{\widetilde{Y}}$, becomes an isomorphism  after tensoring with
$R/I \iso \co_Y$. Although $g=uf$ does not determine $u$ uniquely, any other such $u$  has the same image
in $R/I$.   In view of the discussion above, the desired  isomorphism (\ref{adjunction}) may be defined as
 follows: Zariski locally on $\widetilde{Y}$ there is a homomorphism
 \[
 \mathcal{T}_\LL|_{\widetilde{Y}}  \to  \co(Z)|_{\widetilde{Y}}
 \]
 satisfying $\bm{obst}(\lambda) \mapsto \bm{s}|_{\widetilde{Y}}$.  Such a homomorphism is not unique,
 but any two have the same restriction to $Y$.  This restriction is an isomorphism, and these isomorphisms
 patch together over a Zariski cover.
\end{proof}

At last we construct the promised section (\ref{adjunction section}).
Fix one connected component $Y\subset \mathcal{Y}_{(\LL_0,\Lambda)/U}$, and
regard $\mathcal{Z}_\LL(m,\mathfrak{r})$ as a line bundle on $\mathcal{M}_\LL$.
Its pullback to a line bundle on $U$ satisfies
\[
\mathcal{Z}_\LL(m,\mathfrak{r}) |_U \iso \bigotimes_{ Z  } \co(Z),
\]
where the tensor product is over all connected components  $Z \subset \mathcal{Z}_\LL(m,\mathfrak{r})_{/U}$.
Combining the adjunction isomorphism (\ref{adjunction}) with Proposition \ref{prop:improper} yields an isomorphism
\[
\mathcal{T}_\LL^{R_\Lambda(m,\mathfrak{r})} |_Y  \iso
 \bigotimes_{ \text{$Y$-improper $Z$ }} \co(Z)|_Y,
\]
and hence an isomorphism
\[
\mathcal{Z}_\LL(m,\mathfrak{r})|_Y \iso  \mathcal{T}_\LL^{R_\Lambda(m,\mathfrak{r})} |_Y
\otimes \bigotimes_{ \text{$Y$-proper $Z$ }} \co(Z)|_Y,
\]
which we rewrite as
\begin{equation}\label{sweet isomorphism}
\mathcal{Z}^\heartsuit_\LL(m,\mathfrak{r})|_Y \iso
\bigotimes_{ \text{$Y$-proper $Z$} } \co(Z)|_Y.
\end{equation}
Each line bundle $\co(Z)\supset \co_U$ on $U$ has a canonical section $\bm{s} \in \Gamma( \co (Z))$,
corresponding to the constant function $1$ in $\co_U$, satisfying
\[
 \mathrm{div}(\bm{s}|_Y)   = \mathrm{div}(\bm{s})\cap Y =  Z\cap Y
\]
as divisors on $Y$. Therefore (\ref{sweet isomorphism}) determines a section
\[
\bm{\sigma}_{m,\mathfrak{r}}|_Y \in \Gamma( \mathcal{Z}^\heartsuit_\LL(m,\mathfrak{r})|_Y)
\]
corresponding to the section $\otimes \bm{s} |_Y$ on the right hand side of (\ref{sweet isomorphism}),
and this section  satisfies
\begin{equation}\label{sweet divisor}
\mathrm{div}(\bm{\sigma}_{m,\mathfrak{r}}|_Y) = \sum_{  \text{$Y$-proper $Z$} }  ( Z\cap Y).
\end{equation}
Note that each  $ \bm{s}|_Y$ appearing in the tensor product is nonvanishing at the
generic point of $Y$, precisely because the tensor product is over only  the $Y$-proper $Z$'s.
In particular $\bm{\sigma}_{m,\mathfrak{r}}|_Y$ is nonzero.

By repeating the above construction on each connected component $Y$ of
$\mathcal{Y}_{(\LL_0,\Lambda)/U}$ we obtain
a section of the pullback of $\mathcal{Z}^\heartsuit_\LL(m,\mathfrak{r})$ to
$\mathcal{Y}_{(\LL_0,\Lambda)/U}$.   As $U$ varies over a cover of $\mathcal{M}_\LL$ by sufficiently
small \'etale opens, these sections (being truly canonical) agree on the overlaps, and
the desired section (\ref{adjunction section}) is defined by patching them together.

\begin{proposition}\label{magic section divisor}
As divisors on $\mathcal{Y}_{( \LL_0,\Lambda)}$, we have
\[
\mathrm{div}(\bm{\sigma}_{m,\mathfrak{r}})  =
\sum_{   \substack{  m_1 \in \Q_{>0} \\ m_2 \in \Q_{\ge 0} \\ m_1+m_2 =m  }   }
\mathcal{X}_{ ( \LL_0,\Lambda ) } (m_1, m_2 , \mathfrak{r} ) .
\]
\end{proposition}

\begin{proof}
As in the construction of $\bm{\sigma}_{m,\mathfrak{r}}$, fix a sufficiently small \'etale open subscheme
$U \to \mathcal{M}_\LL$ and a connected component $Y\subset \mathcal{Y}_{(\LL_0,\Lambda) / U}$.
It follows from (\ref{scheme-theoretic decomp}) that there is an isomorphism of $Y$-schemes
\[
 \bigsqcup_{  Z } ( Z\cap Y )
\iso
\bigsqcup_{   \substack{  m_1, m_2 \in \Q_{\ge 0} \\ m_1+m_2 =m  }   }
 \mathcal{X}_{(\LL_0,\Lambda)} (m_1, m_2 , \mathfrak{r} )_{/Y}
\]
where the disjoint union on the left is over all connected components
$Z\subset\mathcal{Z}_\LL(m,\mathfrak{r})_{/U}$.
Each side of this isomorphism has a well-defined $0$-dimensional part: the disjoint union of
all its $0$-dimensional connected components.  Taking the $0$-dimensional parts,
using Theorem \ref{thm:zero cycles} and Proposition \ref{prop:bad parts} for the right hand side,
and then viewing the $0$-dimension parts as divisors on $Y$, we find
\[
 \sum_{  \text{$Y$-proper $Z$} } ( Z\cap Y )
=
\sum_{   \substack{  m_1 \in \Q_{>0} \\ m_2 \in \Q_{\ge 0}  \\ m_1+m_2 =m  }   }
 \mathcal{X}_{(\LL_0,\Lambda)} (m_1, m_2 , \mathfrak{r} )_{/Y}.
\]
Combining this with (\ref{sweet divisor}) shows that
\[
\mathrm{div}(\bm{\sigma}_{m,\mathfrak{r}})_{/Y}
 = \sum_{   \substack{  m_1 \in \Q_{>0} \\ m_2 \in \Q_{\ge 0}  \\ m_1+m_2 =m  }   }
 \mathcal{X}_{(\LL_0,\Lambda)} (m_1, m_2 , \mathfrak{r} )_{/Y},
\]
and the claim follows immediately.
\end{proof}


\subsection{Adjunction in the complex fiber}


\label{sect:compladj}

Fix one point $y\in \mathcal{Y}_{(\LL_0,\Lambda)}(\C)$.
We will give a purely analytic construction of the fiber of  $\bm{\sigma}_{m,\mathfrak{r}}$
at $y$ using the complex uniformization (\ref{uniformization}).
Recall from  Section \ref{ss:complex uniformization} that to the point $y$
there is associated a triple $(\mathfrak{A}_0,\mathfrak{A}_1,\mathfrak{B})$ of
hermitian $\co_\kk$-modules such that $L(\mathfrak{A}_0,\mathfrak{B}) \iso \Lambda$ and
\[
L(\mathfrak{A}_0,\mathfrak{A}_1) \in \mathrm{gen}( \LL_0(\infty) ).
\]
Set $\mathfrak{A}=\mathfrak{A}_1\oplus \mathfrak{B}$ so that
$L(\mathfrak{A}_0,\mathfrak{A}) \in \mathrm{gen}(\LL(\infty))$ and
\[
L(\mathfrak{A}_0,\mathfrak{A}) \iso L(\mathfrak{A}_0,\mathfrak{A}_1) \oplus \Lambda.
\]
Recall that the connected component
of $\mathcal{M}_\LL(\C)$ containing $y$ admits an orbifold presentation
\[
 \Gamma \backslash \mathcal{D}\to  \mathcal{M}_\LL(\C)
\]
in which $\mathcal{D}$ is the space of negative lines in $L(\mathfrak{A}_0,\mathfrak{A})_\R$,
and that under this presentation the point $y$ corresponds to the negative line
$
L(\mathfrak{A}_0,\mathfrak{A}_1)_\R  \subset  L(\mathfrak{A}_0,\mathfrak{A})_\R.
$

Denote by $Z(m,\mathfrak{r})$ the pullback to  $\mathcal{D}$ of the divisor
$\mathcal{Z}_\LL(m,\mathfrak{r})$. By (\ref{KR uniformization}) the corresponding line bundle  is
\begin{equation}\label{analytic line bundle}
Z(m,\mathfrak{r}) \iso
\bigotimes_{  \substack{\lambda \in \mathfrak{r}^{-1}L(\mathfrak{A}_0, \mathfrak{A}) \\ \langle \lambda,\lambda\rangle =m  } } \co(\lambda).
\end{equation}
On the right hand side  $\co(\lambda)$ is the line bundle on $\mathcal{D}$ defined by the divisor $\mathcal{D}(\lambda)$
of negative lines orthogonal to $\lambda$.   We must explain the meaning of the  infinite tensor product on the right.
Denote by  $\bm{s}(\lambda)$  the constant function $1$ on $\mathcal{D}$,
viewed as a section of $\co(\lambda)$.  For  any open set $U\subset \mathcal{D}$ with compact closure
there are only finitely many $\lambda \in \mathfrak{r}^{-1}L(\mathfrak{A}_0, \mathfrak{A})$
satisfying $\langle \lambda,\lambda\rangle =m$ for which $\mathcal{D}(\lambda) \cap U \not=\emptyset$.
For $\lambda$ not in this finite set the section $\bm{s}(\lambda)$ is nonvanishing on $U$, and
defines a trivialization of $\co(\lambda)$.  Thus,  after restricting to any such   $U$ all but finitely many
of the factors of  (\ref{analytic line bundle}) are trivialized, and the meaning of
(\ref{analytic line bundle}) is clear.  The line bundle (\ref{analytic line bundle})  has a canonical section
\[
\bm{s}_{m,\mathfrak{r}} =
 \bigotimes_{  \substack{\lambda \in \mathfrak{r}^{-1}L(\mathfrak{A}_0, \mathfrak{A}) \\ \langle \lambda,\lambda\rangle =m  } } \bm{s}(\lambda)
\]
corresponding to the constant function $1$ in $\co_\mathcal{D} \subset Z(m,\mathfrak{r})$.

Let $T$ denote the pullback of the cotautological bundle $\mathcal{T}_\LL$ to $\mathcal{D}$.
The irreducible components of the divisor $Z(m,\mathfrak{r})$ passing through $y$ are
indexed by the set
\begin{equation}\label{analytic components}
\{ \lambda\in \mathfrak{r}^{-1}\Lambda : \langle \lambda,\lambda\rangle =m \} \subset L(\mathfrak{A}_0, \mathfrak{A}).
\end{equation}
Recall from Proposition  \ref{prop:analytic taut} that at every point  $z\in \mathcal{D}$
there is a canonical isomorphism
\[
T_z \iso \Hom_\C(z,\C),
\]
which was called $\fiber^\vee$, but which we now  suppress  from the notation.
 For each $z\in\mathcal{D}$ and each $\lambda$ in (\ref{analytic components}), denote by  $\lambda_z$
 the orthogonal projection of $\lambda$ to $z$.  There is a unique
holomorphic section $\bm{obst}^\mathrm{an}(\lambda) \in \Gamma( \mathcal{D} , T)$  whose fiber at every point $z$
satisfies
\[
\bm{obst}^\mathrm{an}_z(\lambda)   = \langle \cdot , \lambda_z \rangle.
\]
Of course the zero locus of $\bm{obst}^\mathrm{an}(\lambda)$ is the divisor $\mathcal{D}(\lambda)$,
and hence there is a unique isomorphism of line bundles $\co(\lambda)\iso T$ satisfying
$\bm{s}(\lambda) \mapsto \bm{obst}^\mathrm{an}(\lambda)$.  This isomorphism
is the analytic analogue of  the adjunction isomorphism of Theorem \ref{thm:adjunction}.

Define a holomorphic section
\[
\bm{obst}^\mathrm{an}_{m,\mathfrak{r}} =
 \bigotimes_{  \substack{\lambda \in \mathfrak{r}^{-1}\Lambda \\ \langle \lambda,\lambda\rangle =m } } \bm{obst}^\mathrm{an}(\lambda)
\]
of
\[
T^{R_\Lambda(m,\mathfrak{r}) }
=  \bigotimes_{  \substack{\lambda \in \mathfrak{r}^{-1}\Lambda \\ \langle \lambda,\lambda\rangle =m  } } T.
\]
 The pullback  of (\ref{magic bundle}) to $\mathcal{D}$ is
\begin{equation}\label{complex bundle coords}
Z^\heartsuit(m,\mathfrak{r}) \iso  Z(m,\mathfrak{r}) \otimes
T^{\otimes - R_\Lambda(m,\mathfrak{r}) },
\end{equation}
which has the holomorphic section
$\bm{s}_{m,\mathfrak{r}} \otimes (\bm{obst}^\mathrm{an}_{m,\mathfrak{r} })^{-1}$.

\begin{lemma}
The fiber at $y$ of $\bm{s}_{m,\mathfrak{r}} \otimes (\bm{obst}^\mathrm{an}_{m,\mathfrak{r}})^{-1}$
agrees with the fiber at $y$ of the section $\bm{\sigma}_{m,\mathfrak{r}}$.
\end{lemma}

\begin{proof}
The main thing to explain is the relation between the analytically constructed section $\bm{obst}^\mathrm{an}(\lambda)$
and the algebraically constructed section $\bm{obst}(\lambda)$ of the previous subsection.  We will express
both constructions in terms of parallel transport with respect to the Gauss-Manin connection.
 Let $\bm{R} =\widehat{\co}_{\mathcal{D},y}$ be the completion of the ring of germs of holomorphic
functions at $y$.  Equivalently, $\bm{R}$ is the completed \'etale local ring of $\mathcal{M}_{\LL/\C}$
at $y$, a power series ring over $\C$ in $n-1$ variables.
Let $(A_0,A) \in \mathcal{M}_\LL(\C)$ be the pair represented by the point $y$,
and let $(\bm{A}_0,\bm{A})$ be the universal deformation of $(A_0,A)$ over $\bm{R}$.
If $\mathfrak{m} \subset \bm{R}$ is the maximal ideal, set $\widetilde{R}=\bm{R}/\mathfrak{m}^2$.
Thus $\Spec(\widetilde{R})$
is the first order infinitesimal neighborhood of $y$.  Denote by $(\widetilde{A}_0, \widetilde{A})$
the reduction of $(\bm{A}_0,\bm{A})$ to $\widetilde{R}$.  Let
\[
\bm{T} = \Hom_{\bm{R}}  ( \mathrm{Fil}(\bm{A}_0) , \Lie(\bm{A}) / \mathcal{F}_{\bm{A}} )
\]
be the pullback to $\bm{R}$ of the cotautological bundle, and  let $\widetilde{T}$ be
its reduction to $\widetilde{R}$.

Each $\lambda$ in the set (\ref{analytic components}) determines a $\C$-linear map
$\lambda: H_1^{dR}(A_0) \to H_1^{dR}(A),$
which, using the Gauss-Manin connection, has a canonical deformation
\[
\bm{\lambda} : H_1^{dR}(\bm{A}_0) \to H_1^{dR}(\bm{A})
\]
defined by parallel transport.  See \cite{Lan} and \cite{Vo}
for the Gauss-Manin connection, and \cite{BeOg} for the algebraic theory of parallel transport.
The composition
\[
\mathrm{Fil}(\bm{A}_0) \to H_1^{dR}(\bm{A}_0) \map{\bm{\lambda}}
H_1^{dR}(\bm{A}) \to  \Lie(\bm{A}) / \mathcal{F}_{\bm{A}} ,
\]
viewed as an element of $\bm{T}$, is precisely the pullback of $\bm{obst}^\mathrm{an}(\lambda)$
to $\bm{T}$.
 On the other hand, the reduction of $\bm{\lambda}$ to $\widetilde{R}$
is precisely the map
\[
\widetilde{\lambda} : H_1^{dR}(\widetilde{A}_0) \to H_1^{dR}(\widetilde{A})
\]
appearing in (\ref{first obst}).  Thus the reduction map $\bm{T} \to \widetilde{T}$ sends
$\bm{obst}^\mathrm{an}(\lambda) \mapsto \bm{obst}(\lambda)$.  With this in mind, the rest of the proof
follows by tracing through the definitions of the two sections in question.
\end{proof}

Define a  metrized line bundle on $\mathcal{M}_\LL$ by
\begin{equation}\label{metrized magic bundle}
 \widehat{\mathcal{Z}}^\heartsuit_\LL(f_{m,\mathfrak{r}}) =
  \widehat{\mathcal{Z}}_\LL(f_{m,\mathfrak{r}})
  \otimes    \widehat{\mathcal{T}}_\LL^{\otimes - R_\Lambda(m,\mathfrak{r})}.
\end{equation}

\begin{proposition}\label{prop:complex adjunction}
For any point $y\in \mathcal{Y}_{(\LL_0,\Lambda)} (\C)$, the section $\bm{\sigma}_{m,\mathfrak{r}}$
constructed in Section \ref{ss:adjunction} satisfies
\[
-\log || \bm{\sigma}_{m,\mathfrak{r}} ||_y^2 = \Phi_\LL ( y , f_{m,\mathfrak{r}})
\]
with respect to the metric on $\mathcal{Z}^\heartsuit_\LL(f_{m,\mathfrak{r}})|_y$
determined by (\ref{metrized magic bundle}).
\end{proposition}

\begin{proof}
 The metrized line bundle (\ref{metrized magic bundle})  pulls back to a metrized line bundle on $\mathcal{D}$,
whose underlying line bundle is (\ref{complex bundle coords}).  It is easy to compute the norm
of the section $\bm{s}_{m,\mathfrak{r}} \otimes (\bm{obst}^\mathrm{an}_{m,\mathfrak{r}})^{-1}$
with respect to this metric.  For any $z\in \mathcal{D}$ not contained in the support of $Z(m,\mathfrak{r})$,
the section $\bm{s}_{m,\mathfrak{r}}$ satisfies
\[
- \log || \bm{s}_{m,\mathfrak{r}} ||^2_z = \Phi_\LL(z, f_{m,\mathfrak{r}}),
\]
by definition of the metric on $\widehat{\mathcal{Z}}_{\LL}(f_{m,\mathfrak{r}})$, while
\[ \log || \bm{obst}^\mathrm{an}_{m,\mathfrak{r}} ||_z^2
= \sum_{  \substack{\lambda \in \mathfrak{r}^{-1}\Lambda \\ \langle \lambda,\lambda\rangle =m  } }
\log || \bm{obst}^\mathrm{an}(\lambda) ||_z^2  \\
= \sum_{  \substack{\lambda \in \mathfrak{r}^{-1}\Lambda \\ \langle \lambda,\lambda\rangle =m  } }
\big( \log | \langle \lambda_z,\lambda_z \rangle |  + \log( 4\pi) +\gamma \big),
\]
by definition (\ref{taut metric}) of the metric on $\widehat{\mathcal{T}}_\LL$. It follows that
\begin{align*}
- \log|| \bm{\sigma}_{m,\mathfrak{r}} ||_y^2
& =  \lim_{z\to y}  \big( - \log || \bm{s}_{m,\mathfrak{r}} ||^2_z +
\log || \bm{obst}^\mathrm{an}_{m,\mathfrak{r}} ||_z^2  \big) \\
& = \lim_{z\to y}  \Big(
\Phi_\LL(z, f_{m,\mathfrak{r}} ) +
 \sum_{  \substack{\lambda \in \mathfrak{r}^{-1}\Lambda \\ \langle \lambda,\lambda\rangle =m  } }
\big( \log | \langle  \lambda_z  ,  \lambda_z \rangle |  + \log( 4\pi) +\gamma \big)
  \Big).
\end{align*}
By  Corollary \ref{cor:sing}, this is    the value at $y$ of the (discontinuous) function
$\Phi_\LL(z,f_{m,\mathfrak{r}})$.
\end{proof}


\subsection{Completion of the proof}


Again we consider the metrized line bundle
\[
 \widehat{\mathcal{Z}}^\heartsuit_\LL(f_{m,\mathfrak{r}}) =
  \widehat{\mathcal{Z}}_\LL(f_{m,\mathfrak{r}})
  \otimes    \widehat{\mathcal{T}}_\LL^{\otimes - R_\Lambda(m,\mathfrak{r})}
\]
on $\mathcal{M}_\LL$.  Its restriction to $\mathcal{Y}_{(\LL_0,\Lambda)}$
has a canonical nonzero section (\ref{adjunction section}), which
determines an arithmetic divisor
\[
\widehat{\mathrm{div} }(\bm{\sigma}_{m,\mathfrak{r}}) = (\mathrm{div}(\bm{\sigma}_{m,\mathfrak{r}}) , -\log || \bm{\sigma}_{m,\mathfrak{r}}||^2)
\in \widehat{\mathrm{Div}}(  \mathcal{Y}_{(\LL_0,\Lambda)} )
\]
satisfying
\[
[  \widehat{\mathcal{Z}}^\heartsuit_\LL(f_{m,\mathfrak{r}}) : \mathcal{Y}_{\LL_0,\Lambda} ]
= \widehat{\deg}\ \widehat{\mathrm{div}} (\bm{\sigma}_{m,\mathfrak{r}}) .
\]
Proposition \ref{magic section divisor} implies that the arithmetic divisor
\[
\widehat{\mathrm{div} }_\mathrm{fin}(\bm{\sigma}_{m,\mathfrak{r}}) = (\mathrm{div}(\bm{\sigma}_{m,\mathfrak{r}}) , 0 )
\]
satisfies
\begin{equation}\label{magic section finite}
\widehat{\deg} \  \widehat{\mathrm{div} }_\mathrm{fin}(\bm{\sigma}_{m,\mathfrak{r}})  =
\sum_{   \substack{  m_1 \in \Q_{>0} \\ m_2 \in \Q_{\ge 0} \\ m_1+m_2 =m  }   }
\widehat{\deg} \ \widehat{\mathcal{X}}_{ ( \LL_0,\Lambda ) } (m_1, m_2 , \mathfrak{r} ) .
\end{equation}
On the other hand, it is immediate from Proposition \ref{prop:complex adjunction}  that
the arithmetic divisor
\[
\widehat{\mathrm{div} }_\infty(\bm{\sigma}_{m,\mathfrak{r}}) = (0 , -\log || \bm{\sigma}_{m,\mathfrak{r}}||^2)
\]
satisfies
\begin{equation}\label{magic section infinite}
\widehat{\deg} \  \widehat{\mathrm{div} }_\infty(\bm{\sigma}_{m,\mathfrak{r}})    =
\Phi_\LL(  \mathcal{Y}_{(\LL_0,\Lambda)} , f_{m,\mathfrak{r}}  ) .
\end{equation}
At last we have all the necessary ingredients to prove the main result.

\begin{proof}[Proof of Theorem \ref{thm:arithmetic degree}]
First we treat the case $f=f_{m,\mathfrak{r}}$.
Combining (\ref{magic section finite}) with Proposition  \ref{prop:final zero cycle} shows that
 \[
 \widehat{\deg} \ \widehat{\mathrm{div} }_\mathrm{fin}(\bm{\sigma}_{m,\mathfrak{r}})  =
   -  \deg_\C \mathcal{Y}_{(\LL_0, \Lambda)}
   \sum_{   \substack{  m_1 \in \Q_{>0} \\ m_2 \in \Q_{\ge 0} \\ m_1+m_2 =m  }   }
  a^+_{\LL_0}(m_1 ,\mathfrak{r})  \cdot   R_\Lambda(m_2, \mathfrak{r}) .
 \]
The  $m_1=0$ term absent from the right hand side instead appears in the equality
\[
  [  \widehat{\mathcal{T}}_\LL^{\otimes R_\Lambda(m,\mathfrak{r}) }  :  \mathcal{Y}_{ ( \LL_0, \Lambda) } ]
  = -   \deg_\C \mathcal{Y}_{(\LL_0, \Lambda)}
   \cdot a_{\LL_0}^+(0 ,\mathfrak{r} ) \cdot  R_\Lambda(m, \mathfrak{r})
\]
of Theorem \ref{thm:taut degree}, and combining all of this with  (\ref{magic section infinite})
gives the final equality in
\begin{eqnarray*}\lefteqn{
[  \widehat{\mathcal{Z}}_\LL( f_{m,\mathfrak{r}} ) : \mathcal{Y}_{( \LL_0,\Lambda}) ] } \\
& = &
 [  \widehat{\mathcal{T}}_\LL^{ \otimes R_\Lambda(m,\mathfrak{r}) } :  \mathcal{Y}_{ ( \LL_0, \Lambda) } ]
 +   [  \widehat{\mathcal{Z}}^\heartsuit_\LL( f_{ m,\mathfrak{r}}  ) : \mathcal{Y}_{(\LL_0,\Lambda)} ]
 \\
& = &
[  \widehat{\mathcal{T}}_\LL^{ \otimes R_\Lambda(m,\mathfrak{r}) } :  \mathcal{Y}_{ ( \LL_0, \Lambda) } ]
+  \widehat{\deg} \  \widehat{\mathrm{div} }_\mathrm{fin}(\bm{\sigma}_{m,\mathfrak{r}})
+  \widehat{\deg} \  \widehat{\mathrm{div} }_\infty(\bm{\sigma}_{m,\mathfrak{r}})   \\
& = &
 \Phi_\LL(\mathcal{Y}_{(\LL_0,\Lambda)} , f_{m,\mathfrak{r}} )
 -  \deg_\C \mathcal{Y}_{(\LL_0, \Lambda)}
 \sum_{   \substack{  m_1,m_2 \in \Q_{\ge 0} \\ m_1+m_2 =m  }   }
  a^+_{\LL_0}(m_1,\mathfrak{r})  \cdot  R_\Lambda(m_2,\mathfrak{r} ) .
 \end{eqnarray*}
 Corollary \ref{cor:cm values} shows that
\begin{align*}
 \frac{1}{ \deg_\C \mathcal{Y}_{(\LL_0, \Lambda)} } \cdot
  \Phi_\LL(\mathcal{Y}_{(\LL_0 , \Lambda)},  f_{m , \mathfrak{r}} )
 &= -  L' \big( \xi(f_{m,\mathfrak{r}}) ,\theta_\Lambda , 0 \big)\\
&\phantom{=}{}+   c^+_{m,\mathfrak{r}}(0,0) \cdot  a^+_{\LL_0}(0 ,\mathfrak{r} ) \cdot R_\Lambda(0,\mathfrak{r})  \\
& \phantom{=}{}+   \sum_{   \substack{  m_1,m_2 \in \Q_{\ge 0} \\  m_1+m_2 =m  }   }
 a^+_{\LL_0}(m_1 , \mathfrak{r}) \cdot R_\Lambda(m_2,\mathfrak{r}) ,
\end{align*}
and hence
\begin{align*}
&\frac{1}{ \deg_\C \mathcal{Y}_{(\LL_0, \Lambda)}  } \cdot
[  \widehat{\mathcal{Z}}_\LL(f_{m,\mathfrak{r}}) : \mathcal{Y}_{( \LL_0,\Lambda ) } ] \\
&=
 -     L' \big( \xi(f_{m,\mathfrak{r}}) ,\theta_\Lambda ,0 \big)
 +    c^+_{m,\mathfrak{r}}(0,0) \cdot  a^+_{\LL_0}(0 ,\mathfrak{r} ) \cdot R_\Lambda(0,\mathfrak{r}) .
\end{align*}
Another application of Theorem \ref{thm:taut degree} then shows that
\begin{align*}
& [  \widehat{\mathcal{Z}}_\LL(f_{m,\mathfrak{r}}) : \mathcal{Y}_{ (\LL_0,\Lambda) } ] \\
&=
 -    \deg_\C \mathcal{Y}_{(\LL_0, \Lambda)} \cdot L' \big( \xi(f_{m,\mathfrak{r}} ) ,\theta_\Lambda ,0 \big)
 - c^+_{m,\mathfrak{r}} (0,0) \cdot [ \widehat{\mathcal{T}}_\LL : \mathcal{Y}_{(\LL_0 , \Lambda)} ].
\end{align*}
This completes the proof of (\ref{main prelim}), and hence the proof of Theorem \ref{thm:arithmetic degree}
when $f=f_{m,\mathfrak{r}}$.

 If $f=c^+(0)$ is a constant function (this can only happen when
$n=2$) then $\mathcal{Z}_\LL(f)=0$  and Theorems \ref{thm:CM value} and
\ref{thm:taut degree} imply
\begin{align*}
[ \widehat{\Theta}_\LL( f )  :  \mathcal{Y}_{(\LL_0,\Lambda)} ]   & =
\Phi_\LL( \mathcal{Y}_{(\LL_0,\Lambda)} , f) + c^+(0,0) \cdot
[\widehat{\mathcal{T}}_\LL : \mathcal{Y}_{(\LL_0,\Lambda)}] \\
& =
- \deg_\C \mathcal{Y}_{(\LL_0, \Lambda)} \cdot L'( \xi(f)  , \theta_\Lambda ,0) \\
& \quad
+ \deg_\C \mathcal{Y}_{(\LL_0, \Lambda)}  \cdot
  \{ c^+(0) , \mathcal{E}_{\LL_0} \otimes \theta_\Lambda \}
+ c^+(0,0) \cdot  [\widehat{\mathcal{T}}_\LL : \mathcal{Y}_{(\LL_0,\Lambda)}]
 \\
&= - \deg_\C \mathcal{Y}_{(\LL_0, \Lambda)} \cdot L'( \xi(f)  , \theta_\Lambda ,0).
\end{align*}
Thus Theorem \ref{thm:arithmetic degree} also holds for constant forms.
The decomposition (\ref{f decomp}) implies that the space  $H_{2-n}(\omega_\LL)^\Delta$ is spanned by the constant forms
and the $f_{m,\mathfrak{r}}$'s, and so the desired equality follows by  linearity.
\end{proof}


\begin{thebibliography}{ABCD}



\bibitem[BeOg]{BeOg} {\em P. Berthelot and A. Ogus},
Notes on Crystalline Cohomology.
Princeton University Press, (1978).

\bibitem[Bo1]{Bo1}
\emph{R. E. Borcherds},
Automorphic forms with singularities on Grassmannians.
Invent. Math. \textbf{132} (1998), 491--562.

\bibitem[Bo2]{Bo2}
\emph{R. Borcherds},
 The Gross-Kohnen-Zagier theorem in higher dimensions.
 Duke Math. J. \textbf{97} (1999), 219--233.
Correction in: Duke Math J. \textbf{105} No. 1 p.183--184.


\bibitem[Br1]{Br1} \emph{J. H. Bruinier},
Borcherds products on  $\Orth(2,l)$ and Chern classes of Heegner divisors.
Springer Lecture Notes in Mathematics {\bf 1780}, Springer-Verlag (2002).

\bibitem[B{B}K]{BBK}\emph{J. H. Bruinier, J. Burgos and U. K\"uhn},
Borcherds products and arithmetic intersection theory on Hilbert modular surfaces.
Duke Math. J. {\bf 139} (2007), 1--88.


\bibitem[BF]{BF} {\em J. H. Bruinier and J. Funke},
On two geometric theta lifts.
Duke Math. Journal. {\bf 125} (2004), 45--90.



\bibitem[BY]{BY1} \emph{J. H. Bruinier and T.H.~Yang},
Faltings heights of CM cycles and derivatives of $L$-functions.
Invent. Math. {\bf 177} (2009), 631--681.

\bibitem[BKK]{BKK} {\em J. Burgos, J. Kramer, and U. K\"uhn},
Cohomological arithmetic Chow groups.
J. Inst. Math. Jussieu. {\bf 6}  (2007), 1--178.

\bibitem[Co]{Co} \emph{P.~Colmez},
P\'eriods des vari\'et\'es ab\'eliennes \`a multiplication complexe.
Ann. Math. {\bf 138}  (1993), 625--683 .

\bibitem[FGA]{FGA}{\em B. Fantechi, L. G\"ottsche, L. Illusie, S. Kleiman, N. Nitsure, A. Vistoli},
Fundamental Algebraic Geometry: Grothendieck's FGA Explained.
Mathematical Surveys and Monographs, Vol.~123,
American Mathematical Society, 2005.

\bibitem[Gr]{Gr} {\em B. Gross},
On canonical and quasi-canonical liftings.
Invent. Math., {\bf 84} (1986), 321--326.

\bibitem[GZ]{GZ} {\em B. Gross and D. Zagier},
Heegner points and derivatives of L-series.
Invent. Math. {\bf 84} (1986), 225--320.

\bibitem[Hof]{Hof} {\em E. Hofmann},
Automorphic products on unitary groups.
Dissertation, Technische Universit\"at Darmstadt (2011).

\bibitem[Ho1]{Ho1}{\em B. Howard},
Intersection theory on Shimura surfaces II.
Invent. Math. {\bf 183}, No. 1,  (2011), 1--77.


\bibitem[Ho2]{Ho2}{\em B. Howard},
Complex multiplication cycles and Kudla-Rapoport divisors.
Annals of Math., {\bf 176}  (2012), 1097--1171.

\bibitem[Ho3]{Ho3}{\em B. Howard},
Complex multiplication cycles and Kudla-Rapoport divisors II.
To appear in Amer. J. Math.



\bibitem[Kr]{Kr} \emph{N. Kr\"amer},
Local models for ramified unitary groups.
Abh. Math. Sem. Univ. Hamburg {\bf 73} (2003), 67--80.



\bibitem[Ku2]{Ku:Duke}  \emph{S. Kudla},
Algebraic cycles on Shimura varieties of orthogonal type,
Duke Math. J.  {\bf 86}  (1997),  no. 1, 39--78.


\bibitem[Ku3]{Ku:Integrals}\emph{S. Kudla},
Integrals of Borcherds forms.
Compositio Math. \textbf{137} (2003), 293--349.


\bibitem[Ku4]{Ku4}\emph{S. Kudla},
Special cycles and derivatives of Eisenstein series.
Heegner points and Rankin L-series, 243--270, Math. Sci. Res. Inst. Publ., 49, Cambridge Univ. Press, 2004.




\bibitem[KR1]{KR1} {\em S. Kudla and M. Rapoport},
Special cycles on unitary Shimura varieties I. Unramified local theory.
Invent. Math. {\bf 184} (2011), no. 3, 629--682.

\bibitem[KR2]{KR2} {\em S. Kudla and M. Rapoport},
Special cycles on unitary Shimura varieties II.  Global theory.
Preprint.


\bibitem[KRY1]{KRY1}{\em S.~Kudla, M.~Rapoport, and T.H.~Yang},
On the Derivative of an Eisenstein series of Weight One.
Internat. Math. Res. Notices 1999, no. 7, 347--385.


\bibitem[KRY2]{KRY2}{\em S.~Kudla, M.~Rapoport, and T.H.~Yang},
Modular Forms and Special Cycles on Shimura Curves.
Princeton University Press, (2006).


\bibitem[Lan]{Lan} \emph{K.W.~Lan},
Arithmetic Compacitifactions of PEL-Type Shimura Varieties, London Math. Soc. Monograph,  36.
Princeton University Press, Princeton, (2013).


\bibitem[Liu]{Liu} \emph{Q.~Liu},
Algebraic Geometry and Arithmetic Curves.
Oxford University Press, (2002).



\bibitem[Nek]{Nek}
{\em J. ~Nekov\'a\v{r}},
On the $p$-adic height of Heegner cycles.
 Math. Ann.  {\bf 302} (1995), 609--686.

\bibitem[Pa]{Pa} {\em G. Pappas},
On the arithmetic moduli schemes of PEL Shimura varieties.
J. Algebraic Geom. {\bf 9} (2000), no. 3, 577--605.



\bibitem[Sch]{Sch} \emph{N. R. Scheithauer},
Some constructions of modular forms for the Weil representation of $\SL_2(\Z)$.
Preprint (2011).


\bibitem[Scho]{Scho}
\emph{J.~P.~Schofer},  Borcherds forms and generalizations of singular moduli.
 J. Reine Angew. Math. {\bf 629} (2009), 1--36.

\bibitem[Se73]{Se73}
\emph{J.~P.~Serre}, A Course in Arithmetic.
Graduate Texts in Mathematics, No. 7. Springer-Verlag, New York-Heidelberg, 1973.

\bibitem[Se00]{Se00}
\emph{J.~P.~Serre}, Local Algebra.
 Springer Monographs in Mathematics. Springer-Verlag, Berlin, 2000.


\bibitem[SABK]{SABK} {\em C. Soul\'e, D. Abramovich, J.-F. Burnol, and J. Kramer},
Lectures on Arakelov Geometry.
 Cambridge Studies in Advanced Mathematics {\bf 33},
Cambridge University Press, Cambridge (1992).


\bibitem[Vi]{Vi} \emph{A. Vistoli},
Intersection theory on algebraic stacks and on their moduli spaces.
Invent. Math. {\bf 97} (1989), no. 3, 613--670.

\bibitem[Vo]{Vo} \emph{C. Voisin},
Hodge Theory and Complex Algebraic Geometry I.
Cambridge University Press (2002).


\bibitem[Ya]{YaColmez} {\em T.H. ~Yang},
Chowla-Selberg formula and Colmez's conjecture.
Can. J. Math. {\bf 132} (2010), 456-472.



\bibitem[Zh]{Zh} {\em S.~Zhang},
 Heights of {H}eegner cycles and derivatives of {$L$}-series.
 Invent. Math., {\bf 130} (1997), 99--152.


\end{thebibliography}
\end{document}